\documentclass{amsart}

\newcommand{\cB}{\mathcal{B}}
\newcommand{\cD}{\mathcal{D}}
\newcommand{\cE}{\mathcal{E}}\newcommand{\cF}{\mathcal{F}}
\newcommand{\cG}{\mathcal{G}}\newcommand{\cH}{\mathcal{H}}

\newcommand{\cP}{\mathcal{P}}

\newcommand{\cS}{\mathcal{S}}

\newcommand{\bR}{\mathbb{R}}

\usepackage{lmodern}
\usepackage[dvipsnames,svgnames,x11names,hyperref]{xcolor}
\usepackage{mathtools,url,graphicx,verbatim,amssymb,enumerate,stmaryrd,nicefrac}
\usepackage[pagebackref,colorlinks,citecolor=Mahogany,linkcolor=Mahogany,urlcolor=Mahogany,filecolor=Mahogany]{hyperref}
\usepackage{microtype}
\usepackage[margin=1.5in]{geometry}
\usepackage{setspace}
\usepackage{tikz, tikz-cd}
\usetikzlibrary{matrix,calc,arrows,patterns,decorations.pathreplacing}
\usepackage{caption}
\newtheorem{theorem}{Theorem}[section]
\newtheorem{lemma}[theorem]{Lemma}
\newtheorem{proposition}[theorem]{Proposition} 
\newtheorem{corollary}[theorem]{Corollary}

\newtheorem*{corollary*}{Corollary}

\newtheorem{atheorem}{Theorem}

\newtheorem{acorollary}[atheorem]{Corollary}

\theoremstyle{definition}
\newtheorem{definition}[theorem]{Definition}
\newtheorem{example}[theorem]{Example} 
\newtheorem{remark}[theorem]{Remark}

\newtheorem{convention}[theorem]{Convention}

\newcommand{\cat}[1]{\mathsf{#1}}
\newcommand{\mr}[1]{{\rm #1}}

\newcommand{\rel}{\mathbin{\text{rel}}}
\newcommand{\la}{\lvert}
\newcommand{\ra}{\rvert}

\newcommand{\longto}{\longrightarrow}

\date{\today}

\author{Alexander Kupers}
\title{Three applications of delooping to $h$-principles}
\address{Harvard Department of Mathematics \\
	One Oxford Street \\
	Cambridge MA, 02138, USA}
\email{kupers@math.harvard.edu}          
\thanks{AK is supported by NSF grant DMS-1803766, was partially supported by NSF grant DMS-1105058, and was supported by the Danish National Research Foundation through the Centre for Symmetry and Deformation (DNRF92).}
\begin{document}

\begin{abstract}In this paper we give three applications of a method to prove $h$-principles on closed manifolds. Under weaker conditions this method proves a homological $h$-principle, under stronger conditions it proves a homotopical one. The three applications are as follows: a homotopical version of Vassiliev's $h$-principle, the contractibility of the space of framed functions, and a version of Mather-Thurston theory.\end{abstract}

%\keywords{$H$-principles, generalized Morse functions, framed functions, foliations}

\maketitle

\section{Introduction} In this paper we give three applications to geometric topology of two influential ideas in algebraic topology. The first idea goes back to Smale's work on immersions and concerns ``$h$-principles'' reducing geometric problems to homotopy-theoretic ones \cite{smaleimm}. The other is ``delooping,'' which goes back to the recognition principle for iterated loop spaces \cite{M}. We will apply these techniques to give general conditions under which an $h$-principle holds on closed manifolds, and check these conditions are satisfied in three examples. These examples will be a homotopical version of Vassiliev's $h$-principle, the contractibility of the space of framed functions, and a version of Mather-Thurston theory.

\subsection{$H$-principles on closed manifolds} Following Gromov, in this paper we study \emph{invariant topological sheaves $\Psi$ on $n$-manifolds} \cite{gromovhp}. The prototypical examples are sheaves of smooth functions which do not have singularities of a certain type. More precisely, such sheaves assign to each open subset $U$ of a manifold the space of smooth functions on $U$ whose singularities do not belong to a set $\cD$ of singularities fixed beforehand. This describes a sheaf since one can restrict and glue such functions, and it is invariant if and only if the set $\cD$ is invariant under diffeomorphisms. 

Given an invariant topological sheaf $\Psi$, one can construct a germ map $j$ to a sheaf $\Psi^f$ which has values weakly equivalent to a space of sections of a bundle with fiber $\Psi(\bR^n)$. For the sheaf of smooth functions whose singularities do not belong to the subset $\cD$ of the space of $r$-jets of smooth functions, $\Psi^f$ is the space of sections of the subbundle of the $r$-jet bundle given by the complement of $\cD$, and $j$ sends a function to its $r$-jet. Thus the map $j$ is given by the inclusion of functions into ``formal functions.'' Alternatively, one can think it as the inclusion of ``holonomic sections'' into all sections.

To say that $\Psi$ satisfies a \emph{(parametrized) $h$-principle} is to say that the germ map $j$ is a weak equivalence \cite{emhpbook}. Spaces of sections are purely homotopy-theoretic and can be understood using the calculational tools of homotopy theory, so an $h$-principle reduces a geometric problem to a more tractable homotopy-theoretic problem. Thus $h$-principles provide a valuable connection between geometry and homotopy theory, and it is desirable to have general conditions under which $\Psi$ satisfies an $h$-principle. 

Gromov gives such conditions for sheaves on open manifolds \cite{gromovhp}. A manifold is said to be \emph{open} if none of its path components is compact. We denote the elements of $\Psi$ on $M$ that satisfy a boundary condition $b$ near $\partial M$ by $\Psi(M\rel b)$.

\begin{theorem}[Gromov] \label{thm.gromov} Let $\mr{CAT} = \mr{Diff}$, $\mr{PL}$ or $\mr{Top}$, and suppose $\Psi$ is a microflexible $\mr{CAT}$-invariant topological sheaf on $n$-dimensional manifolds. 
	
	Then there exists a flexible $\mr{CAT}$-invariant sheaf $\Psi^f$, whose values are weakly equivalent to a space of sections with fiber $\Psi(\bR^n)$, and a map of sheaves $j \colon \Psi \to \Psi^f$ such that
	\[j \colon \Psi(M \rel b) \longto \Psi^f(M \rel j(b))\]
	is a weak equivalence for all open $\mr{CAT}$-manifolds $M$ of dimension $n$ and boundary conditions $b$ near $\partial M$.
\end{theorem}

To prove our results we use an extension of this theory to closed manifolds. This extension uses delooping techniques to show that a sufficient condition for $\Psi$ to satisfy an $h$-principle on closed manifolds is that it is \emph{group-like}. This condition can take one of the following forms, the first visibly weaker than the second (for detailed definitions see Subsections \ref{subsec.condh} and \ref{subsec.condw}):
\begin{description}
	\item[(H)] the connected components of the values of $\Psi$ on the cylinder $S^{n-1} \times [0,1]$ form a groupoid,
	\item[(W)] (H) holds and composing with an element representing an identity is a weak equivalence.
\end{description}

\begin{atheorem}\label{thm.main} Let $\mr{CAT} = \mr{Diff}$, $\mr{PL}$ or $\mr{TOP}$, and suppose $\Psi$ is a microflexible $\mr{CAT}$-invariant topological sheaf on $n$-dimensional manifolds. 
	\begin{enumerate}[(i)]
		\item If $\Psi$ is satisfies condition (H), then the map
		\[j \colon \Psi(M\rel b) \xrightarrow{\cong_{H_*}}
		\Psi^f(M\rel j(b))\] 
		is homology equivalence for all $\mr{CAT}$-manifolds $M$ of dimension $n$ and boundary conditions $b$ near $\partial M$.
		\item If $\Psi$ is satisfies condition (W), then the map
		\[j \colon \Psi(M\rel b) \overset{\simeq}{\longrightarrow}
		\Psi^f(M\rel j(b))\] 
		is weak equivalence for all $\mr{CAT}$-manifolds $M$ of dimension $n$ and  boundary conditions $b$ near $\partial M$.
	\end{enumerate}

\end{atheorem}

Though variations of this result are known among experts, we give its proof. In particular, Michael Weiss, S\o ren Galatius, Nathan Perlmutter and Chris Schommer-Pries independently have obtained a version of part (ii) in unpublished work. Furthermore, Emanuele Dotto gave a way to deduce relative $h$-principles from non-relative ones using similar ideas and as an application gives an alternative proof of the Madsen-Weiss theorem \cite{dotto}. The earliest source of the ``categorical approach'' to $h$-principles known to the author is a collection of lecture notes due to Michael Weiss \cite{weissimm}.

\subsection{Applications} The goal of this paper is to give three applications of Theorem \ref{thm.main}. The first is the \emph{contractibility of the space of framed functions}, see Corollary \ref{cor.framed}. It was originally proven by Eliashberg-Mishachev \cite{emframed}, and Galatius used techniques similar to ours in an unpublished proof.

\begin{acorollary}\label{cor.framedfunctions} The space $\cG_\mr{fr}(M \rel b)$ of framed functions is contractible for all smooth manifolds $M$ and boundary conditions $b$ near $\partial M$.\end{acorollary}

Our second application is a \emph{generalization of Vassiliev's $h$-principle} \cite{vassiliev}, see Corollary \ref{cor.vassiliev2}. It will be a consequence of Thom's jet transversality theorem and implies the $h$-principle for generalized Morse functions proven in \cite{emgmf}. 

\begin{acorollary}\label{cor.vassiliev}Suppose $Z$ is a smooth manifold and $\cD$ is a closed $\mr{Diff}$-invariant stratified subset of codimension at least $n+2$ of the space of $r$-jets of smooth functions $\bR^n \to Z$. Let $\cF(-,\cD)$ denote the space of functions to $Z$ with $r$-jet avoiding $\cD$, then the map
	\[j \colon \cF(M,\cD \rel b) \longto \cF^f(M,\cD \rel j(b))\]
	is a homology equivalence for all smooth manifolds $M$ of dimension $n$ and boundary conditions $b$ near $\partial M$. If $\cD$ is additionally $\mr{Diff}(Z)$-invariant, then this map is in fact a weak equivalence. \end{acorollary}

Finally, we give short proof of a version of \emph{Mather-Thurston theory for foliations}. Let $\mr{Fol}_\mr{CAT}(-)$ denote the space of codimension $n$ $\mr{CAT}$-foliations as in Definition \ref{def.folspace}. There is a unique such foliation $\cF_0$ on a single manifold $M$, which we can take as a boundary condition near $\partial M$.

\begin{acorollary}\label{cor.thurston} The map $j\colon \mr{Fol}_\mr{CAT}(M \rel \cF_0) \to \mr{Fol}_{\mr{CAT}}^f(M \rel j(\cF_0))$ is a homology equivalence for all $\mr{CAT}$-manifolds $M$ of dimension $n$.
\end{acorollary}

\subsection{Conventions} We fix some conventions.

\begin{convention}Fix a category $\mr{CAT}$ of manifolds: $\mr{CAT} = \mr{Diff}$, $\mr{PL}$, $\mr{TOP}$. All our manifolds are second countable and metrizable (which implies Hausdorff and paracompact).\end{convention}

\begin{convention}The category of spaces $\cat{S}$ will be $\cat{Sh^{CAT}}$ as in Appendix \ref{app.spaces}.\end{convention}

\subsection{Acknowledgements}Alexander Kupers is supported by NSF grant DMS-1803766. We would like to thank S\o ren Galatius for sharing his ideas on $h$-principles and much helpful advice, Oscar Randal-Williams and Johannes Ebert for pointing out a mistake in an earlier version of this paper, and Sam Nariman for comments on an early draft. We thank the referee for helpful comments and suggestions, and Chris Schommer-Pries and Michael Weiss for historical comments. Furthermore, we would like to thank Yasha Eliashberg, Daniel Alvarez-Gavela, Jeremy Miller, and Nathan Perlmutter for stimulating discussions about various $h$-principles.

\tableofcontents

\section{Invariant topological sheaves} \label{sec.setup} In this section we discuss the objects appearing in our $h$-principle.

\subsection{Sheaves} In this subsection we will define the $\mr{CAT}$-invariant topological sheaves that are the subject of this paper. To do so, we start by defining topological categories of different types of manifolds and embeddings. For us, \emph{topological category} has a set or class of objects and spaces of morphisms. In many examples one can make the objects form a set by requiring their underlying sets are subsets of some sufficiently large set, e.g.~$\bR^\infty$ in the next definition; we will ignore such issues.

\begin{definition}Let $\cat{Mfd}^\mr{CAT}_n$ be the topological category defined as follows:
	\begin{itemize}
		\item Objects are $n$-dimensional $\mr{CAT}$-manifolds without boundary.
		\item Morphism spaces are given by spaces of $\mr{CAT}$-embeddings: that is, $\mr{Hom}(M,N)$ is the object $\mr{Emb}^\mr{CAT}(M,N)$ of $\cat{Sh^{CAT}}$ which assigns to a parametrizing manifold $P$ the set of $\mr{CAT}$-maps $P \times M \to P \times N$ over $P$ that are $\mr{CAT}$-isomorphisms onto their image.
\end{itemize}\end{definition}

We can now define $\mr{CAT}$-invariant topological sheaves.

\begin{definition}\label{def.sheaf} Let $\mr{CAT}$ be $\mr{Diff}$, $\mr{PL}$ or $\mr{Top}$. \begin{itemize}
		\item A \emph{$\mr{CAT}$-invariant topological presheaf on $n$-dimensional $\mr{CAT}$-manifolds} is a continuous functor $\Psi \colon (\cat{Mfd}_n^\mr{CAT})^\mr{op} \to \cat{S}$. That is, there are continuous maps $\mr{Emb}(M,M') \times \Psi(M') \to \Psi(M)$ compatible with composition and identities. 
		\item A \emph{$\mr{CAT}$-invariant topological sheaf on $n$-dimensional $\mr{CAT}$-manifolds} is a presheaf $\Psi$ that has the property that for all open covers $\{U_i\}_{i \in I}$ of $M$ the following diagram is an equalizer in $\cat{S}$:
		\[\begin{tikzcd} \Psi(M) \rar &  \prod_{i \in I} \Psi(U_i) \arrow[shift left=.5ex]{r} \arrow[shift left=-.5ex]{r} &  \prod_{i,j \in I} \Psi(U_i \cap U_j). \end{tikzcd}\]
		\item A \emph{map of $\mr{CAT}$-invariant topological sheaves} is a continuous natural transformation.\end{itemize}\end{definition}
	
For $e \in \mr{Emb}(M,M')$, we shall often use $e^*$ for the induced \emph{pullback} map $\Psi(M') \to \Psi(M)$. The exception is when $e$ is inclusion of codimension zero submanifold $M = U$ of $M'$, in which case we use $(-)|_U \colon \Psi(M') \to \Psi(U)$. It is more intuitive for $\phi \in \mr{CAT}(M,M')$ to think of $(\phi^{-1})^*$ as a \emph{pushforward} map and thus we introduce the notation $\phi_* \coloneqq (\phi^{-1})^*$. Though we will not use this, we point out that the gluing property implies that the functoriality of $\mr{CAT}$-invariant topological sheaf automatically extends from embeddings to immersions.
%\textbf{The reason for restricting to finite covers in the previous definition is that filtered colimits commute only with \emph{finite} products, so this restriction is necessary to make sure that filtered colimits of sheaves are sheaves.}

\begin{example}\label{exam.sheaves} 
	We give a number of examples of $\mr{CAT}$-invariant topological sheaves:
	\begin{itemize}
		\item Fix a space $X$, and let $\Psi(M) = \mr{Map}(M,X)$ be the space of continuous maps in compact-open topology. If $\cat{S} = \cat{Sh^{CAT}}$, then the value on a parametrizing manifold $P$ is the set of continuous map $P \times M \to X$. This generalizes to spaces of sections of bundles naturally associated to $M$, like the following example.
		\item Suppose $\mr{CAT} = \mr{Diff}$, and let $\mr{Riem}(M)$ be the space of Riemannian metrics on $M$, topologized as a subspace of the sections of $T^* M \otimes T^* M$ in the weak $C^\infty$-topology. If $\cat{S} = \cat{Sh^{CAT}}$, then the value on a parametrizing manifold $P$ is a subset of the set of sections of $\pi^*(T^*M \otimes T^*M)$ over $P \times M$, where $\pi \colon P \times M \to M$ denotes the projection.
		\item Let $\mr{CAT} = \mr{PL}$, or $n \neq 4$ and $\mr{CAT} = \mr{Top}$. For a $\mr{CAT}$-manifold $M$ let $\mr{Sm}(M)$ be the object of $\cat{Sh^{CAT}}$ which assigns to parametrizing manifold $P$ the set of commutative diagrams
		\[\begin{tikzcd} P \times M \arrow{rr}{\cong}[swap]{F} \arrow{rd}[swap]{\pi_2} & & N \arrow{ld}{f} \\
		& P, & \end{tikzcd}\]
		%	\xymatrix{P \times M \ar[rr]_-\cong^-F \ar[rd]_-{\pi_2} & & N \ar[ld]^-f \\
		%		& P & }\]
		with $\pi_2$ the projection, $N$ a smooth manifold, $f$ a smooth submersion, and $F$ a $\mr{CAT}$-isomorphism. Then the $\mr{CAT}$-invariant topological sheaf $M \mapsto \mr{Sm}(M)$ is the subject of smoothing theory \cite{burglashof,kirbysiebenmann}.
\end{itemize}\end{example}

\subsection{Boundary conditions} \label{subsec.bdy} It is often useful to make relative definitions. In our case that means defining boundary conditions near closed subsets and considering subspaces of $\Psi$ of elements satisfying such boundary conditions. This is the content of the following definitions. 

\subsubsection{Manifolds with corners} As the terminology indicates, boundary conditions will often be imposed near the boundary of a manifold. Our sheaves are not defined on manifolds with boundary, but there is a canonical way to extend $\Psi$ to manifolds with boundary (or more generally corners, when $\mr{CAT} = \mr{Diff}$), as long as these are submanifolds of an object of $\cat{Mfd}^\mr{CAT}_n$.

%A \emph{collaring} of a submanifold with corners $M$ of $\bR^\infty$ is a pair of an object $W$ of $\cat{Mfd}_n^\mr{CAT}$ and an embedding $M \to W$.   %However, the  preferred choices as above will form a final subcategory in $\cat{I}_M$.

\begin{definition}\label{def.psicorners} If $M$ is a manifold with corners contained in $N \in \cat{Mfd}_n^\mr{CAT}$, then let $\cat{I}_M$ be the directed set of open neighborhoods $U$ of $M$ in $N$ (where there is a unique morphism $U \to U'$ if $U' \subset U$). We define $\Psi(M \subset N)$ as the colimit (taken in $\cat{S}$)
	\[\Psi(M \subset N) \coloneqq \underset{U \in \cat{I}_M}{\mr{colim}}\, \Psi(U).\]
\end{definition}

This colimit exists because our choice $\cat{Sh^{CAT}}$ for $\cat{S}$ is cocomplete, given by sheafifying the colimit of presheaves, which is just the pointwise colimit of sets. These colimits are well-behaved, in contrast with colimits in $\cat{Top}$.

\begin{convention}If $M$ is a manifold with corners contained in $N \in \cat{Mfd}_n^\mr{CAT}$ and $N$ is clear from the context (e.g.~the simplex $\Delta^n$ sitting inside the extended simplex $\Delta^n_e$) we write $\Psi(M)$ for $\Psi(M \subset N)$.\end{convention}

\subsubsection{Boundary conditions} Our next goal is to define boundary conditions and subspaces of $\Psi(M)$ of elements satisfying a boundary condition.

\begin{definition}Let $M$ be a manifold with corners and $A \subset M$ be a closed subset of $M$. Let $\cat{I}_A$ be the directed set of open neighborhoods $U \subset M$ of $A$ (where there is a unique morphism $U \to U'$ if $U' \subset U$). The \emph{set of boundary conditions of $\Psi$ near $A$} is given by the colimit (taken in $\cat{Set}$)
	\[\cB_\mr{\Psi}(A) \coloneqq \underset{U \in \cat{I}_A}{\mr{colim}}\, \Psi(U).\]\end{definition}

By definition of the colimit there is a map of sets $\beta_A \colon \Psi(U) \to \cB_\mr{\Psi}(A)$ for all neighborhoods $U \subset M$ of $A$. Thus any element $f$ of $\Psi(U)$ gives rise to a boundary condition $\beta_A(f)$ near $A$. We say that $f \in \Psi(U)$ \emph{satisfies} $b \in \cB_\mr{\Psi}(A)$ if $\beta_A(f) = b$.

\begin{definition}\label{def.withbdycondition} Let $M$ be a manifold with corners, $A \subset M$ be a closed subset of $M$ and  $b \in \cB_\mr{\Psi}(A)$ a boundary condition. Let $\cat{I}_b$ be the directed set of pairs $(U,f)$ of open neighborhoods $U$ of $A$ and elements $f$ of $\Psi(U)$ satisfying $b$ (where there is a unique morphism $(U,f) \to (U',f')$ if $U' \subset U$ and $f|_{U'} = f'$). 
	
	We let $\Psi(M \rel (U,f))$ be the subspace of $\Psi(M)$ of those $g$ such that $g|_{U \cap M} = f|_{U \cap M}$, then we define the \emph{space of elements of $\Psi$ satisfying $b$} as the colimit (taken in $\cat{S}$)
	\[\Psi(M  \rel b) \coloneqq \underset{(U,f) \in \cat{I}_b}{\mr{colim}}\, \Psi(M \rel (U,f)).\]
\end{definition}

Boundary conditions can be restricted and pulled back:
\begin{itemize}\item  If $A' \subset A$, there is a map of directed sets $\cat{I}_A \to \cat{I}_{A'}$ given by considering $(U,f)$ with $U$ a neighborhood of $A$ as a neighborhood of $A'$. This induces a \emph{restriction} map $(-)|_{A'} \colon \cB_\mr{\Psi}(A) \to \cB_\mr{\Psi}(A')$. We say that $b \in \cB_\Psi(A)$ \emph{extends} $B \in \cB_\Psi(A')$ if $b|_{A'} = B$.
	\item If $\phi \colon M \to M'$ is an embedding, there is the map of directed sets $\phi^{-1} \colon \cat{I}_{A'} \to \cat{I}_{\phi^{-1}(A')}$ given by $U' \mapsto \phi^{-1}(U')$. This induces a \emph{pullback} map $\phi^* \colon \cB_\Psi(A') \to \cB_\Psi(\phi^{-1}(A'))$. \end{itemize}

\section{Group-like invariant topological sheaves}In this section we make precise the conditions mentioned in the introduction.

\subsection{Microflexibility} \label{subsec.microflex} As in Gromov's theory \cite{gromovhp}, a major role is played by the notion of a microfibration:

\begin{definition}\label{def.microfib} A map $g \colon E \to B$ is a \emph{microfibration} if for all $i \geq 0$ and each commutative diagram
	\[\begin{tikzcd} \Delta^i \times \{0\} \rar \dar & E \dar{p} \\
	\Delta^i \times [0,1] \rar & B \end{tikzcd}\]
	%\[\xymatrix{\Delta^i \times \{0\} \ar[r] \ar[d] & E \ar[d]^p \\
	%\Delta^i \times [0,1] \ar[r] & B}\]
	there exists an $\epsilon > 0$ and a lift over $\Delta^i \times [0,\epsilon]$.\end{definition}

As expected, the condition of microflexibility requires certain maps to microflexible.

\begin{definition}\label{def.microflex} We say that a $\mr{CAT}$-invariant topological sheaf $\Psi$ is \emph{microflexible} if for all pairs of $Q \subset R$ of compact subsets in $M$, the restriction map
	\[\Psi(R \subset M) \longto \Psi(Q \subset M)\]
	is a microfibration.\end{definition}

Having a condition for all compact subsets makes microflexiblity seem hard to check. However, the following lemma say it suffices to check this condition only for particular $Q \subset R$.

\begin{lemma}\label{lem.microflexvariants}A $\mr{CAT}$-invariant topological sheaf $\Psi$ is microflexible for all pairs $Q \subset R$ of compact subsets if and only if it is microflexible for all pairs $Q \subset R$ where $Q$ is $n$-dimensional submanifold admitting a finite handle decomposition and $R$ is obtained from $Q$ by adding a single handle.
\end{lemma}

\begin{proof}The direction $\Rightarrow$ is obvious. For $\Leftarrow$ we note that since a composition of microfibrations is a microfibration, the right hand side implies the case where $R$ is obtained from $Q$ by attaching finitely many handles. To obtain the left hand side, we note that for every pair of opens $(U,V)$ containing $(K,L)$ there exists a pair of compact $n$-dimensional manifolds with finite handle decompositions $(P,Q)$ such that $(K,L) \subset (P,Q) \subset (U,V)$. To see this, subdivide the handles in some handle decomposition of $M$ sufficiently many times. In the case of compact topological 4-manifolds, which might not admit a handle decomposition, one must first remove some points and use that every topological 4-manifold without compact path components is smoothable.
\end{proof}

\subsection{Condition (H)} \label{subsec.condh} We now start discussing senses in which a sheaf can be group-like. The statements will involve non-unital categories, i.e.\ a category without specified identity elements.

\begin{definition}\label{def.catc} We define a topological non-unital category $\cat{C}^\Psi$ as follows: 
	\begin{itemize}\item Objects are pairs $\la t,b \ra$ of a real number $t>0$ and a boundary condition $b \in \cB_\Psi(S^{n-1} \times \{t\})$.
		\item The space of morphisms from $\la t_0,b_0 \ra$ to $\la t_1,b_1 \ra$ is given by
		\[\cat{C}^\Psi\left(\la t_0,b_0 \ra,\la t_1,b_1 \ra\right) \coloneqq \begin{cases} \varnothing & \text{if $t_0 \geq t_1$,} \\ 
		\Psi(S^{n-1} \times [t_0,t_1] \rel b_0 \sqcup b_1) &  \text{if $t_0 < t_1$.}\end{cases}\]
		\item Composition of $f_0 \in  \cat{C}^\Psi(\la t_0,b_0\ra,\la t_1,b_1\ra)$ and $f_1 \in \cat{C}^\Psi(\la t_1,b_1\ra,\la t_2,b_2\ra)$ is given by concatenation
		\begin{align*}\circledcirc\colon \cat{C}^\Psi\left(\la t_0,b_0 \ra,\la t_1,b_1 \ra\right) \times\cat{C}^\Psi\left(\la t_1,b_1\ra,\la t_2,b_2\ra \right) &\longto \cat{C}^\Psi\left(\la t_0,b_0 \ra,\la t_2,b_2 \ra \right)\\
		(f_0,f_1) &\longmapsto \begin{cases} f_0 & \text{on $S^{n-1} \times [t_0,t_1]$,} \\
		f_1 & \text{on $S^{n-1} \times [t_1,t_2]$,}\end{cases}\end{align*}
		glued using the sheaf property of $\Psi$ and the fact that $f_0$ and $f_1$ both coincide with $b_1$ on an open neighborhood of $S^{n-1} \times \{t_1\}$ in $S^{n-1} \times \bR$. \end{itemize}\end{definition}
	
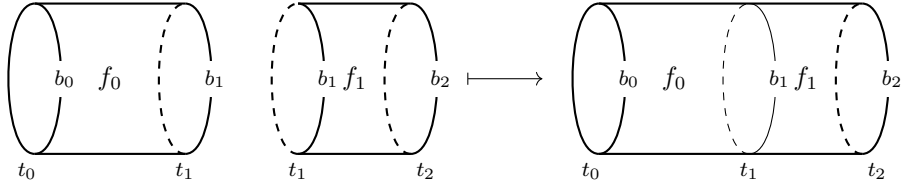
\begin{figure}
		\centering
		\begin{tikzpicture}
		\begin{scope}[xshift=-7.5cm]
		\draw [thick] (0,1) to[out=0,in=0,looseness=.6] (0,-1);
		\draw [thick]  (0,-1) to[out=180,in=180,looseness=.6] (0,1);
		\node at (-.1,-1) [below] {\footnotesize $t_0$};
		\node at (0.4,0) [fill=white] {\footnotesize $b_0$};	 	
		
		\draw [thick] (2,1) to[out=0,in=0,looseness=.6] (2,-1);
		\draw [thick,dashed] (2,-1) to[out=180,in=180,looseness=.6] (2,1);
		\node   at (2,-1) [below] {\footnotesize $t_1$};	
		\node   at (2.4,0) [fill=white] {\footnotesize $b_1$};	 
		
		\draw [thick] (0,1) -- (2,1);
		\draw [thick] (0,-1) -- (2,-1);		
		
		\node at (1,0) {$f_0$};
		\end{scope}

		\begin{scope}[xshift=-6cm]
		\draw [thick] (2,1) to[out=0,in=0,looseness=.6] (2,-1);
		\draw [thick,dashed] (2,-1) to[out=180,in=180,looseness=.6] (2,1);
		\node at (2,-1) [below] {\footnotesize $t_1$};	
		\node at (2.4,0) [fill=white] {\footnotesize $b_1$};	 
		
		\draw [thick] (3.5,1) to[out=0,in=0,looseness=.6] (3.5,-1);
		\draw [thick] [dashed] (3.5,-1) to[out=180,in=180,looseness=.6] (3.5,1);
		\node at (3.7,-1) [below] {\footnotesize $t_2$};	
		\node at (3.9,0) [fill=white] {\footnotesize $b_2$};
		
		\draw [thick] (2,1) -- (3.5,1);
		\draw [thick] (2,-1) -- (3.5,-1);		
		
		\node at (2.75,0) {$f_1$};
		\end{scope}
		
		\draw [|->] (-1.75,0) -- (-.75,0);
		
		\draw [thick] (0,1) to[out=0,in=0,looseness=.6] (0,-1);
		\draw [thick]  (0,-1) to[out=180,in=180,looseness=.6] (0,1);
		\node at (-.1,-1) [below] {\footnotesize $t_0$};
		\node at (0.4,0) [fill=white] {\footnotesize $b_0$};	 	
		
		\draw (2,1) to[out=0,in=0,looseness=.6] (2,-1);
		\draw [dashed] (2,-1) to[out=180,in=180,looseness=.6] (2,1);
		\node at (2,-1) [below] {\footnotesize $t_1$};	
		\node at (2.4,0) [fill=white] {\footnotesize $b_1$};	 
		
		\draw [thick] (3.5,1) to[out=0,in=0,looseness=.6] (3.5,-1);
		\draw [thick] [dashed] (3.5,-1) to[out=180,in=180,looseness=.6] (3.5,1);
		\node at (3.7,-1) [below] {\footnotesize $t_2$};	
		\node at (3.9,0) [fill=white] {\footnotesize $b_2$};	
		
		\draw [thick] (0,1) -- (3.5,1);
		\draw [thick] (0,-1) -- (3.5,-1);		
		
		\node at (1,0) {$f_0$};
		\node at (2.75,0) {$f_1$};	
		\end{tikzpicture}
		\caption{The composition of $f_0$ and $f_1$ in $\cat{C}^\Psi$.}
		\label{fig.composition}
	\end{figure}

For condition (H), what is relevant is not the actual category $\cat{C}^\Psi$ but its category of path components.

\begin{definition}\label{def.catpi0c} From $\cat{C}^\Psi$ we construct a non-unital category $[\cat{C}^\Psi]$ as follows:
	\begin{itemize} \item Objects are boundary conditions $b \in \cB_\Psi(S^{n-1} \times \{1\})$.
		\item Morphisms from $b_0$ to $b_1$ are given by the set $\pi_0( \Psi(S^{n-1} \times [1,2] \rel b_0 \sqcup b_1))$.
		\item Composition is given by concatenation and rescaling
		\begin{align*}\circledcirc \colon [\cat{C}^\Psi](b_0,b_1) \times [\cat{C}^\Psi](b_1,b_2) &\longto [\cat{C}^\Psi](b_0,b_2) \\
		([f_0],[f_1]) &\longmapsto [(\mr{id}_{S^{n-1}} \times \lambda)_* \left((\mr{id}_{S^{n-1}} \times \rho)_*(f_1) \circledcirc f_0\right)],\end{align*}
		where $\mr{id}_{S^{n-1}} \times \lambda \colon S^{n-1} \times [1,3] \to S^{n-1} \times [1,2]$ is obtained from a $\mr{CAT}$-isomorphism $\lambda \colon [1,3] \to [1,2]$ equal to $t \mapsto t$ near $1$ and to $t \mapsto t-1$ near $3$, and $\rho \colon [1,2] \to [2,3]$ is the $\mr{CAT}$-isomorphism given by $t \mapsto t+1$. To check composition is well-defined and associative, one uses that $\Psi$ is $\mr{CAT}$-invariant.\end{itemize}\end{definition}

Having identities is a property, not a structure. An element $f \in \cat{C}(c,c)$ is an identity element if composition on the left (resp.\ right) with it induces the identity on all morphism sets $\cat{C}(c,c')$ (resp.\ $\cat{C}(c',c)$). If a non-unital category $\cat{C}$ admits identity elements, these are unique: if $\cat{C}(c,c)$ contains identity elements $\mr{id}_c$ and $\mr{id}'_c$, then these satisfy $\mr{id}_c = \mr{id}_c \mr{id}'_c = \mr{id}'_c$. Recall that a unital category is a groupoid if each morphism has an inverse.

\begin{definition}\label{def.h} We say that $\Psi$ satisfies \emph{condition (H)} if the category $[\cat{C}^\Psi]$ is a groupoid.
\end{definition}

\begin{example}\label{exam.grouplikesingle}If $\pi_0(\Psi(S^{n-1} \times [1,2] \rel b_0 \sqcup b_1))$ consists of a single element for all boundary conditions $b_0$ and $b_1$, then $[\cat{C}^\Psi]$ is a groupoid.
\end{example}

\subsection{Condition (W)} \label{subsec.condw} The next condition is a strengthening of condition (H).

\begin{definition}\label{def.compcategory} The \emph{components} of a category $\cat{C}$ are given by the set of objects of $\cat{C}$ under the equivalence relation generated by declaring $x \sim x'$ if there is a morphism $x \to x'$.\end{definition}

The set of components of $\cat{C}$ is in natural bijection with $\pi_0 (B\cat{C})$, the path components of its classifying space.

\begin{definition}We say a boundary condition $b \in \cB_\Psi(S^{n-1})$ is \emph{fillable} if $\Psi(D^n \rel b) \neq \varnothing$.	
\end{definition}

If condition (H) is satisfied, then the full subcategory of $[\cat{C}^\Psi]$ on fillable boundary conditions is a union of path components. This definition may seem ambiguous, as $S^{n-1} \subset S^n$ bounds two different disks $D^n_+$ and $D^n_-$. However, part (i) of Theorem \ref{thm.main}, which does not depend on condition (W) or the notion of fillable boundary conditions, implies the following:

\begin{lemma}Let $b \in \cB_\Psi(S^{n-1})$, and suppose that $\Psi$ satisfies condition (H). Then $\Psi(D^n_+ \rel b) \neq \varnothing$ if and only if $\Psi(D^n_- \rel b) \neq \varnothing$.\end{lemma}

\begin{proof}Let $\overline{b}$ denote the boundary condition obtained from $b$ by reflection in $S^{n-1}$, then we may equivalently prove that $\Psi(D^n_+ \rel b) \neq \varnothing$ if and only if $\Psi(D^n_+ \rel \overline{b}) \neq \varnothing$. By symmetry, it suffices to prove the implication $\Rightarrow$. If condition (H) holds, then Theorem \ref{thm.main}(i) says that 
	\[H_0(\Psi(D^n_+ \rel b)) \longto H_0(\Psi^f(D^n_+ \rel j(b)))\]
is an isomorphism, and similarly for $\overline{b}$. As $j(\overline{b}) = \overline{j(b)}$, we may thus assume that $\Psi = \Psi^f$. But for such sheaves we may easily construct an element in $\Psi^f(S^{n-1} \times [0,1] \rel b \sqcup \overline{b})$ by taking the section to be independent of $[0,1]$; concatenation with this element proves $\Psi^f(D^n_+ \rel \overline{b})$ is non-empty if $\Psi^f(D^n_+ \rel b)$ is non-empty.\end{proof}

\begin{definition}\label{def.d} We say $\Psi$ satisfies \emph{condition (W)} if it satisfies condition (H) and in each fillable component of $[\cat{C}^\Psi]$, there exist boundary conditions $b_s$, $b_e$, real numbers $t_s<t_e$, and an element $\iota \in \cat{C}^\Psi(\la t_s,b_s \ra, \la t_e,b_e \ra)$, such that all maps $\iota \circledcirc -$ are weak equivalences.\end{definition}

\begin{example}Condition (W) is particularly easy to check in the situation of Example \ref{exam.grouplikesingle}; one just needs to find a single element $\iota$ which acts by a weak equivalence. Often this takes the form of a ``constant'' or ``linear'' element.\end{example}

\subsection{Flexible sheaves} \label{subsec.flexible} We will describe a closely related class of sheaves that always satisfy conditions (H) and (W). These were studied by Gromov \cite{gromovhp}.

\begin{definition}\label{def.flex}  We say that a $\mr{CAT}$-invariant topological sheaf $\Gamma$ is \emph{flexible} if for all pairs of $Q \subset R$ of compact subsets in $M$, the restriction map
	\[\Gamma(R \subset M) \longto \Gamma(Q \subset M)\]
	is a Serre fibration.\end{definition}

As in Lemma \ref{lem.microflexvariants}, it suffices to check this only for those $Q \subset R$ where $Q$ is a $n$-dimensional submanifold with a finite handle decomposition and $R$ is obtained from $Q$ by attaching a single handle.

\begin{example}If $\Gamma$ is obtained by taking sections of some natural bundle over $M$, then $\Gamma$ will be flexible, as restricting sections along a cofibration is a Serre fibration. In particular, the sheaf $\mr{Riem}$ of Riemannian metrics from Example \ref{exam.sheaves}, is flexible.\end{example}

As Serre fibrations are microfibrations, a flexible sheaf is microflexible. It also always satisfies conditions (H) and (W).

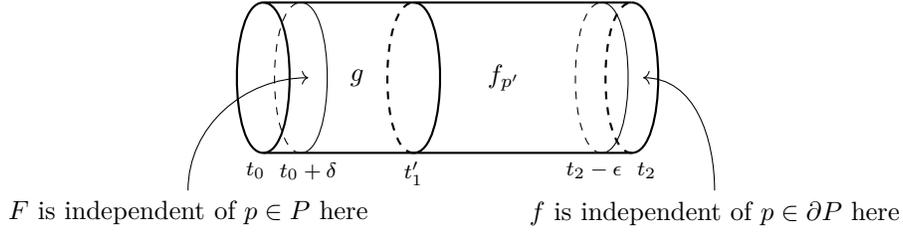
\begin{figure}
	\centering
	\begin{tikzpicture}
		\draw [thick] (0,1) to[out=0,in=0,looseness=.6] (0,-1);
		\draw [thick]  (0,-1) to[out=180,in=180,looseness=.6] (0,1);
		\node at (-.1,-1) [below] {\footnotesize $t_0$};
		
		\draw (0.5,1) to[out=0,in=0,looseness=.6] (0.5,-1);
		\draw [dashed] (0.5,-1) to[out=180,in=180,looseness=.6] (0.5,1);
		\node at (0.6,-1) [below] {\footnotesize $t_0+\delta$};	
		
		\draw [thick] (2,1) to[out=0,in=0,looseness=.6] (2,-1);
		\draw [thick] [dashed] (2,-1) to[out=180,in=180,looseness=.6] (2,1);
		\node at (2,-1) [below] {\footnotesize $t'_1$};	
			
		\draw (4.5,1) to[out=0,in=0,looseness=.6] (4.5,-1);
		\draw [dashed] (4.5,-1) to[out=180,in=180,looseness=.6] (4.5,1);
		\node at (4.4,-1) [below] {\footnotesize $t_2-\epsilon$};	
		
		\draw [thick] (4.9,1) to[out=0,in=0,looseness=.6] (4.9,-1);
		\draw [thick] [dashed] (4.9,-1) to[out=180,in=180,looseness=.6] (4.9,1);
		\node at (5.1,-1) [below] {\footnotesize $t_2$};	
		
		\draw [thick] (0,1) -- (4.9,1);
		\draw [thick] (0,-1) -- (4.9,-1);		
		
		\node at (1.25,0) {$g$};
		\node at (3.2,0) {$f_{p'}$};
		
		\node at (6,-1.8) {$f$ is independent of $p \in \partial P$ here};
		\draw [->] (6,-1.5) to[out=90,in=0] (5.05,0);
		
		\node at (-1,-1.8) {$F$ is independent of $p \in P$ here};
		\draw [->] (-1,-1.5) to[out=90,in=180] (.6,0);		
	\end{tikzpicture}
	\caption{The image of $p' \in \partial P$ under the composite $(- \circledcirc g) \circ f$. The construction in the proof of Lemma \ref{lem.secgrouplike} pushes the larger left cylinder $S^{n-1} \times [t_0,t_1]$ into the smaller left cylinder $S^{n-1} \times [t_0,t_0+\delta]$, while not modifying the smaller right cylinder $S^{n-1} \times [t_2-\epsilon,t_2]$.}
	\label{fig.secgrouplike}
\end{figure}

\begin{lemma}\label{lem.secgrouplike} A flexible sheaf $\Gamma$ satisfies conditions (H) and (W).\end{lemma}

\begin{proof}We will show $\Gamma$ has the property that for all $g \in \cat{C}^\Gamma(\la t,b \ra, \la t_1,b_1 \ra)$ the maps
	\begin{align*}g \circledcirc - &\colon \Gamma \left(S^{n-1} \times [t_2,t_1] \rel b_2 \sqcup b_1\right) \longto \Gamma \left(S^{n-1} \times [t_2,t_0] \rel b_2 \sqcup b_0\right), \\
	- \circledcirc g &\colon \Gamma \left(S^{n-1} \times [t_1,t_2] \rel b_1 \sqcup b_2\right) \longto \Gamma\left(S^{n-1} \times [t_0,t_2] \rel b_0 \sqcup b_2\right),\end{align*}
	are weak equivalences. We will give a proof in the second case, the other case being similar. Suppose we are given a parametrizing manifold $P$ and a commutative diagram
	\[\begin{tikzcd} \partial P \dar \rar{f} & \Gamma(S^{n-1} \times [t_1,t_2] \rel b_1 \sqcup b_2) \dar{- \circledcirc g} \\
	P \rar[swap]{F} \arrow[dotted]{ru} & \Gamma(S^{n-1} \times [t_0,t_2] \rel b_0 \sqcup b_2) \end{tikzcd}\]
	then we need to construct a dotted lift making the top triangle commute and the bottom triangle commute up to homotopy rel $\partial P$. Without loss of generality there is a neighborhood $U$ of $\partial P$ in $P$ over which a lift already exists.
	
	Consider the following construction, whose use will become clear later in this proof. Pick $\epsilon > 0$ such that $t_2 - \epsilon > t_1$ and the restriction of $f$ to $S^{n-1} \times [t_2-\epsilon,t_2]$ is independent of $p \in \partial P$. By composing the bottom map with restriction to $S^{n-1} \times ([t_0,t_1] \cup [t_2-\epsilon,t_2])$, we obtain a map 
	\[\tilde{F} \colon P \longto \Gamma\left(S^{n-1} \times ([t_0,t_1] \cup [t_2-\epsilon,t_2]) \subset S^{n-1} \times [t_0,t_2] \rel b_0 \sqcup b_2\right).\]
	There is similarly a $\delta>0$ such that the restriction to $S^{n-1} \times [t_0,t_0+\delta]$ is independent of $p \in P$. 
	
	Now pick a family $\phi_s$ of self-embeddings of $S^{n-1} \times ([t_0,t_1] \cup [t_2-\epsilon,t_2])$ with $s \in [0,\infty)$ such that 
	\begin{enumerate}[(i)]
		\item $\phi_0$ is the identity, 
		\item for each $s$, $\phi_s$ is the identity on $S^{n-1} \times [t_2-\epsilon,t_2]$ and near $S^{n-1} \times \{t_0\}$, 
		\item $\phi_s$ maps $S^{n-1} \times [t_0,t_1]$ into $S^{n-1} \times [t_0,t_0+(t_1-t_0)/(s+1)]$. 
	\end{enumerate}
	
	Then there exists a CAT function $\sigma \colon P \to [0,\infty)$ that is $0$ near $\partial P$ and so that $(t_1-t_0)/(\sigma(p)+1) < \delta$ for $p \in P \setminus U$. Consider the family 
	\begin{align*}P \times [0,1] &\longto \Gamma\left(S^{n-1} \times ([t_0,t_1] \cup [t_2-\epsilon,t_2]) \subset S^{n-1} \times [t_0,t_2] \rel b_0 \sqcup b_2\right) \\
	(p,s) &\longmapsto \phi_{s\sigma(p)}^* \tilde{F}(p).\end{align*}
	The relevant property of this family is that for $s=1$, the restriction of $\phi_{\sigma(p)}^* \tilde{F}(p)$ to $S^{n-1} \times [t_0,t_1]$ is pulled back from $S^{n-1} \times [t_0,t_0+\delta]$, where it is independent of $p \in P$.
	
	Thus, at $s=1$ this family coincides with the result at $s=1$ of the same construction applied to the map 
	\[f_g \colon P \longto \Gamma\left(S^{n-1} \times ([t_0,t_1] \cup [t_2-\epsilon,t_2]\right) \subset S^{n-1} \times [t_0,t_2] \rel b_0 \sqcup b_2)\] that on $S^{n-1} \times [t_0,t_1]$ is equal to $g$ and on $S^{n-1} \times [t_2-\epsilon,t_2]$ is equal to the restriction of $f$, which is in fact independent of $p \in \partial P$ and hence extendable to $P$. We can concatenate $(p,s) \mapsto \phi_{s\sigma(p)}^*\tilde{F}(p)$ and $(p,s) \mapsto \phi_{(1-s)\sigma(p)}^*f_g(p)$ to a family 
	\[P \times [0,2] \longto \Gamma\left(S^{n-1} \times ([t_0,t_1] \cup [t_2-\epsilon,t_2]) \subset  S^{n-1} \times [t_0,t_2] \rel b_0 \sqcup b_2 \right).\]
	Since this is constant near $\partial P$, it fits into a commutative diagram
	\[\begin{tikzcd} P \times \{0\} \cup \partial P \times [0,2] \rar \dar &[-25pt] \Gamma\left(S^{n-1} \times [t_0,t_2] \rel b_0 \sqcup b_2\right) \dar \\
	P \times [0,2] \rar & \Gamma\left(S^{n-1} \times ([t_0,t_1] \cup [t_2-\epsilon,t_2]) \subset S^{n-1} \times [t_0,t_2] \rel b_0 \sqcup b_2\right). \end{tikzcd}\]
	Since the restriction map is a Serre fibration, there is a lift. This lift induces a homotopy  rel $\partial P$  from $f \colon P \to \Gamma(S^{n-1} \times [t_0,t_2] \rel b_0 \sqcup b_2)$ to a map that is in the image of $- \circledcirc g$.
\end{proof}

Thus Theorem \ref{thm.main} says that flexible sheaves satisfy an $h$-principle, a result due to Gromov. Conversely, if a microflexible sheaf satisfies an $h$-principle, then it is flexible.

\begin{lemma}\label{lem.flexibilityopen} Let $\Psi \to \Phi$ be a weak equivalence of $\mr{CAT}$-invariant topological sheaves, and suppose that $\Psi$ is microflexible and $\Phi$ is flexible. Then $\Psi$ is flexible as well.\end{lemma}

\begin{proof}It suffices to check that for all pairs $Q \subset R \subset M$ of $n$-dimensional compact submanifolds with corners in $M$, so that $R$ is obtained from $Q$ by attaching a single handle, the restriction map 
	\[\Psi(R \subset M) \longto \Psi(Q \subset M)\] is a Serre fibration. We know it is a microfibration and there is a weak equivalence $\Psi \to \Phi$ to a flexible sheaf. Note that the fiber over $a \in \Psi(Q \subset M)$ is given by $\Psi(I^k \times I^{n-k} \rel a)$, and similarly for $\Phi$, and by assumption the induced map on fibers is a weak equivalence. Now apply the result in \cite{raptis}, which says that if we have a commutative diagram
	\[\begin{tikzcd} E \rar{g} \dar[swap]{p} & E' \dar{p'} \\
	B \rar[swap]{f} & B' \end{tikzcd}\]
	such that 
	\begin{enumerate}[(i)]
		\item $p$ is a microfibration,
		\item $p'$ is a Serre fibration, 
		\item for all $b \in B$, the induced map on fibers $g \colon p^{-1}(b) \to (p')^{-1}(f(b))$ is a weak equivalence,
	\end{enumerate}
	then $p$ is a Serre fibration.\end{proof}

\section{$H$-principles on closed manifolds}

We will now prove Theorem \ref{thm.main} in the following slightly stronger technical version. 

\begin{theorem}\label{thm.technical} Let $\mr{CAT} = \mr{Diff}$, $\mr{PL}$ or $\mr{TOP}$. Let $j \colon \Psi \to \Gamma$ be a morphism of $\mr{CAT}$-invariant topological sheaves on $n$-manifolds. Suppose that 
	\begin{enumerate}[(i)]
		\item $\Psi$ is microflexible,
		\item  $\Gamma$ is flexible, and
		\item  $\Psi(\bR^n) \to \Gamma(\bR^n)$ is a weak equivalence.
	\end{enumerate}
	Let $M$ be a $\mr{CAT}$-manifold of dimension $n$, $A \subset M$ a closed subset, and $a \in \cB_\Psi(A \subset M)$ a boundary condition.
	\begin{itemize}
		\item If $\Psi$ is satisfies condition (H), then the map
		\[j \colon \Psi(M\rel a) \xrightarrow{\cong_{H_*}}
		\Gamma(M\rel j(a))\] 
		is homology equivalence.
		\item If $\Psi$ is satisfies condition (W), then the map
		\[j \colon \Psi(M\rel a) \overset{\simeq}{\longrightarrow}
		\Gamma(M\rel j(a))\] 
		is weak equivalence.
	\end{itemize}
\end{theorem}

To deduce Theorem \ref{thm.main}, one uses that Gromov constructed for each microflexible $\mr{CAT}$-invariant sheaf $\Psi$ a natural $\mr{CAT}$-invariant flexible sheaf $\Psi^f$ with a map $j \colon \Psi \to \Psi^f$ of sheaves. He further proved that $\Psi^f(M)$ is weakly equivalent to the space of sections of the bundle $\mr{Fr}^\mr{CAT}(TM) \times_{\mr{CAT}(n)} \Psi(\bR^n)$ over $M$, which has fiber $\Psi(\bR^n)$. This is described for $\mr{CAT} = \mr{Diff}$ in Section 2.2.2 of \cite{gromovhp} and in general in Appendix V.A of \cite{kirbysiebenmann}, the latter of which explicitly avoids assuming the existence of a handle decomposition when $\mr{CAT} = \mr{Top}$, which may not exist in dimension 4. We will need the precise statement of Gromov's $h$-principle for open manifolds, previously Theorem \ref{thm.gromov}:

\begin{theorem}[Gromov-Siebenmann] \label{thm.gromovprecise} Let $j \colon \Psi \to \Gamma$ be a map of $\mr{CAT}$-invariant sheaves. Suppose that
	\begin{enumerate}[(i)]
		\item $\Psi$ is microflexible,
		\item $\Gamma$ is flexible, and
		\item $j \colon \Psi(\bR^n) \to \Gamma(\bR^n)$ is a weak equivalence. 
	\end{enumerate}  
	Let $M$ be a manifold, $A \subset M$ be a closed subset so that $M \setminus A$ has no path components with compact closure in $M$, and $a \in \cB_\Psi(A \subset M)$ be a boundary condition. Then
	\[j \colon \Psi(M \rel a) \longto \Gamma(M \rel j(a))\]
	is a weak equivalence.
\end{theorem}

\subsection{Composition of cylinders} We start by proving that under conditions (H) or (W) composition with a morphism in $\cat{C}^\Psi$ is a homology equivalence or a weak equivalence. It will be useful to consider the following intermediate notion.

\begin{definition}\label{def.semiequivalence} A map $f\colon X \to Y$ is a \emph{semi-equivalence} if for all parametrizing manifolds $P$ with boundary $\partial P$ and commutative diagrams
	\[\begin{tikzcd} \partial P \rar \dar & X \dar{f} \\
	P \arrow[dotted]{ru}[description]{L} \rar &  Y, \end{tikzcd}\]
	there is a map $L$ such that the top and bottom triangle commute up to independent homotopy (thus we do \emph{not} require the homotopy for the bottom triangle to be rel $\partial P$).\end{definition}

\begin{lemma}\label{lem.nearequivalence} A semi-equivalence $f\colon X \to Y$ is a homology equivalence and is injective on homotopy groups. \end{lemma}

\begin{proof}We start with the second claim. Suppose that $x \in \pi_i(X)$ becomes null-homotopic after composition with $f$, i.e.~there is a commutative diagram
	\[\begin{tikzcd} \partial D^{i+1} \rar{x} \dar & X \dar{f} \\
D^{i+1} \rar &  Y .\end{tikzcd}\]
Since $f$ is a semi-equivalence, we can find a map $L \colon D^{i+1} \to X$ such that $L|_{\partial P} \sim x$. By gluing this homotopy to $L$ we get a map $D^{i+1} \cong (D^{i+1} \times \{1\}) \cup (\partial D^{i+1} \times [0,1]) \to X$ whose restriction to $\partial P \times \{0\}$ is $x$, exhibiting $x$ as already being zero in $\pi_i(X)$. (The homotopy for the bottom triangle plays no role in this argument.) 

For the first claim we use Lemma \ref{lem.homologequivalencedef} which says that a map is a homology equivalence if and only if it is an oriented bordism equivalence. For injectivity, suppose that $x \in \Omega^\mr{SO}_i(X)$ becomes null-cobordant after composition with $f$, i.e.~there is an $(i+1)$-dimensional compact oriented manifold $P$ and a commutative diagram
\[\begin{tikzcd} \partial P  \rar{x} \dar & X \dar{f} \\
P \rar &  Y .\end{tikzcd}\]
Then the same argument as above with $P$ replacing $D^{i+1}$ proves that $x$ was already zero in $\Omega^\mr{SO}_i(X)$.

For surjectivity, suppose we have an $y \in \Omega^\mr{SO}_i(Y)$, represented by an $i$-dimensional closed oriented manifold $P$ together with a map $P \to Y$. Then we can think of this as a commutative diagram
\[\begin{tikzcd} \partial P = \varnothing  \rar \dar & X \dar{f} \\
P \rar{y} &  Y .\end{tikzcd}\]
Since $f$ is semi-equivalence, there is a map $L \colon P \to Y$ such that $f \circ L \sim y$. Thus $L \colon P \to X$ represents an element in $\Omega^\mr{SO}_i(X)$ which is mapped to $y$ by $f$.
\end{proof}

Let us now return to the task at hand, proving that condition (H) or (W) imply composition in $\cat{C}^\Psi$ is a homology equivalence or a weak equivalence.

\begin{proposition}\label{prop.compcyl} Let $\Psi$ be a $\mr{CAT}$-invariant topological sheaf on $n$-manifolds. For $t_0 < t_1 < t_2$ and an element
	\[g \in \cat{C}^\Psi(\la t_1,b_1 \ra, \la t_2,b_2 \ra ) = \Psi(S^{n-1} \times [t_1,t_2] \rel b_1 \sqcup b_2)\]
	consider the map
	\[\begin{tikzcd} \cat{C}^\Psi(\la t_0,b_0 \ra, \la t_1,b_1 \ra ) = \Psi(S^{n-1} \times [t_0,t_1] \rel b_0 \sqcup b_1) \dar{g \circledcirc -} \\
	\cat{C}^\Psi(\la t_0,b_0 \ra, \la t_2,b_2 \ra ) = \Psi(S^{n-1} \times [t_0,t_2] \rel b_0 \sqcup b_2).\end{tikzcd}\]
	%\[\xymatrix{\cat{C}^\Psi(\la t_0,b_0 \ra, \la t_1,b_1 \ra ) = \Psi(S^{n-1} \times [t_0,t_1] \rel b_0 \sqcup b_1) \ar[d]_{g \circledcirc -} \\
	%\cat{C}^\Psi(\la t_0,b_0 \ra, \la t_2,b_2 \ra ) = \Psi(S^{n-1} \times [t_0,t_2] \rel b_0 \sqcup b_2)}\]
	\begin{enumerate}[(i)]
		\item If $\Psi$ satisfies condition (H), the map $g \circledcirc -$ is a semi-equivalence as in Definition \ref{def.semiequivalence}.
		\item If $\Psi$ satisfies condition (W) and $b_0$ (or equivalently $b_1$ or $b_2)$ is fillable, the map  $g \circledcirc -$ is a weak equivalence.
	\end{enumerate}
\end{proposition}

\begin{proof}We start with part (i). We have to prove that for each compact parametrizing manifold $P$ with boundary $\partial P$ and each commutative diagram
	\begin{equation}\label{eqn.diagcomposableproof1}\begin{tikzcd} \partial P\rar{f} \ar[d] & \Psi\left(S^{n-1} \times [t_0,t_1] \rel b_0 \sqcup b_1\right) \dar{g \circledcirc -} \\
	P \rar[swap]{F} \arrow[dotted]{ru}[description]{L} & \Psi\left(S^{n-1} \times [t_0,t_2] \rel b_0 \sqcup b_2\right) \end{tikzcd}\end{equation}
	%\begin{equation}\label{eqn.diagcomposableproof1}\begin{gathered}\xymatrix{\partial P\ar[r]^-f \ar[d] & \Psi\left(S^{n-1} \times [t_0,t_1] \rel b_0 \sqcup b_1\right) \ar[d]^{g \circledcirc -} \\
	%P \ar[r]_-F \ar@{.>}[ru] & \Psi\left(S^{n-1} \times [t_0,t_2] \rel b_0 \sqcup b_2\right)}\end{gathered}\end{equation}
	there is a dotted lift $L \colon P \to \Psi\left(S^{n-1} \times [t_0,t_1] \rel b_0 \sqcup b_1\right)$ such that the top and bottom triangle independently commute up to homotopy. 
	
	Since our category of spaces $\cat{S}$ is $\cat{Sh^{CAT}}$, there exist an open neighborhood $U_1$ of $\partial S^{n-1} \times \{t_1\}$ such that $f|_{U_1}$ is independent of $p \in \partial P$ and an open neighborhood $U_2$ of $\partial S^{n-1} \times \{t_2\}$ such that $F|_{U_2}$ is independent of $p \in P$. Let $\tau_1 \in (t_0,t_1)$ be such that $S^{n-1} \times [\tau_1,t_1] \subset U_1$ and let $f^\tau$ denote the restriction of $f(p)$ to $S^{n-1} \times [\tau_1,t_1] \subset U_1$ for any $p \in \partial P$ (it is by construction independent of $p$). Similarly, let $\tau_2 \in (t_1,t_2)$ be such that $S^{n-1} \times [\tau_2,t_2] \subset U_2$ and let $g'$ and $g^\tau$ denote the restrictions of $g$ to $S^{n-1} \times [t_1,\tau_2]$ and $S^{n-1} \times [\tau_2,t_2]$ respectively, so that $g^\tau \circledcirc g' = g$. Then we can factor \eqref{eqn.diagcomposableproof1} as
	\begin{equation}\label{eqn.diagcomposableproof2} \begin{tikzcd} \partial P \rar{f'} \dar &[-5pt] \Psi\left(S^{n-1} \times [t_0,\tau_1] \rel b_0 \sqcup \beta_1\right) \rar{f^\tau \circledcirc - } \dar[swap]{g' \circledcirc f^\tau \circledcirc -} &   \Psi\left(S^{n-1} \times [t_0,t_1] \rel b_0 \sqcup b_1\right)\arrow{dl}[description]{g' \circledcirc -}  \dar{g^\tau \circledcirc g' \circledcirc - } \\[10pt]
	P \rar[swap]{F'} &  \Psi\left(S^{n-1} \times [t_0,\tau_2] \rel b_0 \sqcup \beta_2\right)\rar[swap]{g^\tau \circledcirc - }  & \Psi\left(S^{n-1} \times [t_0,t_2] \rel b_0 \sqcup b_2\right) \end{tikzcd}\end{equation}
	where $f'$, $F'$, $\beta_1$ and $\beta_2$ are uniquely determined by demanding  this diagram is a factorization of (\ref{eqn.diagcomposableproof1}).
	
	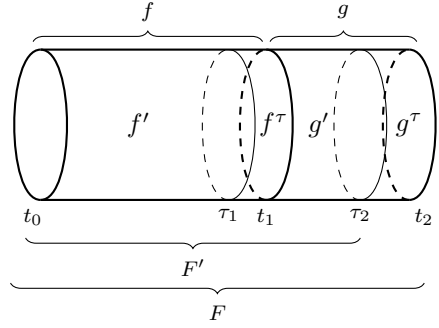
\begin{figure}
		\centering
		\begin{tikzpicture}
		\draw [thick] (0,1) to[out=0,in=0,looseness=.6] (0,-1);
		\draw [thick]  (0,-1) to[out=180,in=180,looseness=.6] (0,1);
		\node at (-.1,-1) [below] {\footnotesize $t_0$};
		
		\draw (2.5,1) to[out=0,in=0,looseness=.6] (2.5,-1);
		\draw [dashed] (2.5,-1) to[out=180,in=180,looseness=.6] (2.5,1);
		\node at (2.5,-1) [below] {\footnotesize $\tau_1$};	
		
		\draw [thick] (3,1) to[out=0,in=0,looseness=.6] (3,-1);
		\draw [thick] [dashed] (3,-1) to[out=180,in=180,looseness=.6] (3,1);
		\node at (3,-1) [below] {\footnotesize $t_1$};	
		
		\draw (4.25,1) to[out=0,in=0,looseness=.6] (4.25,-1);
		\draw [dashed] (4.25,-1) to[out=180,in=180,looseness=.6] (4.25,1);
		\node at (4.25,-1) [below] {\footnotesize $\tau_2$};	
		
		\draw [thick] (4.9,1) to[out=0,in=0,looseness=.6] (4.9,-1);
		\draw [thick] [dashed] (4.9,-1) to[out=180,in=180,looseness=.6] (4.9,1);
		\node at (5.1,-1) [below] {\footnotesize $t_2$};	
		
		\draw [thick] (0,1) -- (4.9,1);
		\draw [thick] (0,-1) -- (4.9,-1);	
		
		\draw [decorate,decoration={brace,amplitude=4pt},xshift=0pt,yshift=3pt]
		(-0.1,1) -- (2.95,1) node [black,midway,xshift=0cm,yshift=.4cm] {\footnotesize $f$};	
		\draw [decorate,decoration={brace,amplitude=4pt},xshift=0pt,yshift=3pt]
		(3.05,1) -- (5,1) node [black,midway,xshift=0cm,yshift=.4cm] {\footnotesize $g$};
		\draw [decorate,decoration={brace,amplitude=4pt},xshift=0pt,yshift=-3pt]
		(4.25,-1.4) -- (-.2,-1.4) node [black,midway,xshift=0cm,yshift=-.4cm] {\footnotesize $F'$};
		\draw [decorate,decoration={brace,amplitude=4pt},xshift=0pt,yshift=-3pt]
		(5.1,-2) -- (-.4,-2) node [black,midway,xshift=0cm,yshift=-.4cm] {\footnotesize $F$};
		
		\node at (1.3,0) {$f'$};
		\node at (3.1,0) {$f^\tau$};
		\node at (3.7,0) {$g'$};
		\node at (4.9,0) {$g^\tau$};
	
		\end{tikzpicture}
		\caption{An illustration of \eqref{eqn.diagcomposableproof2}.}
		\label{fig.compcyl}
	\end{figure}
	
	By condition (H), the morphism $[g']$ of $[\cat{C}^\Psi](b_1,\beta_2)$ has an inverse, which we can represent by an element $(g')^{-1} \in \Psi(S^{n-1} \times [\tau_2,t_2] \rel \beta_2 \sqcup b_1)$. We claim the lift $L\colon P \to \Psi(S^{n-1} \times [t_0,t_1] \rel b_0 \sqcup b_1)$ is given by
	\[L \coloneqq (\mr{id}_{S^{n-1}} \times \lambda)_* [((g')^{-1} \circledcirc -) \circ F'],\]
	where $\lambda \colon [t_0,t_2] \to [t_0,t_1]$ is a $\mr{CAT}$-isomorphism satisfying
	\[\lambda(t) = \begin{cases} t+(t_1-t_2) & \text{near $t_2$,} \\
	t & \text{if $t \in [t_0,\tau_1]$.}\end{cases}\]
	To the sake of readibility, we shall abbreviate $(\mr{id}_{S^{n-1}} \times \lambda)_*$ by $\lambda_*$, and do the same for similar maps.
	
	To prove this is indeed the desired map, we have to check that (a) $L|_{\partial P}$ is homotopic to $(f^\tau \circledcirc -) \circ f' = f$, and (b) $(g^\tau \circledcirc g' \circledcirc -) \circ L$ is homotopic to $(g^\tau \circledcirc -) \circ F' = F$. For part (a) we refer to Figure \ref{fig.compcylhomotopy1}, and write 
	\begin{align*}L|_{\partial P} &= \lambda_* [((g')^{-1} \circledcirc -) \circ F'|_{\partial P}] \\
	&= \lambda_*[(g')^{-1} \circledcirc -) \circ (g'  \circledcirc -) \circ (f^\tau \circledcirc -) \circ f'] \\
	&= \lambda_*[(((g')^{-1} \circledcirc g' \circledcirc f^\tau) \circledcirc -) \circ f'] \\
	&= [\lambda_*((g')^{-1} \circledcirc g' \circledcirc f^\tau) \circledcirc -] \circ \lambda_*(f') \\
	&= [\lambda_*((g')^{-1} \circledcirc g' \circledcirc f^\tau) \circledcirc -] \circ f' \\
	& \sim [f^\tau \circledcirc - ] \circ f' = f.\end{align*}
	The first and second step are definitions, the third is associativity of composition. In the fourth step we use compatibility between composition and $\mr{CAT}$-invariance. In the second-to-last step we use that $\lambda$ fixes $[t_0,\tau_1]$ pointwise to see that $\lambda_*(f') = f'$. In the last step we use that $\lambda_* ((g')^{-1} \circledcirc g' \circledcirc f^\tau)$ is homotopic to $f^\tau$, as both are representatives in $\Psi(S^{n-1} \times [\tau_1,t_1]  \rel \beta_1 \sqcup b_1)$ of the morphism $[f^\tau]$ of $[\cat{C}^\Psi]$.
	
	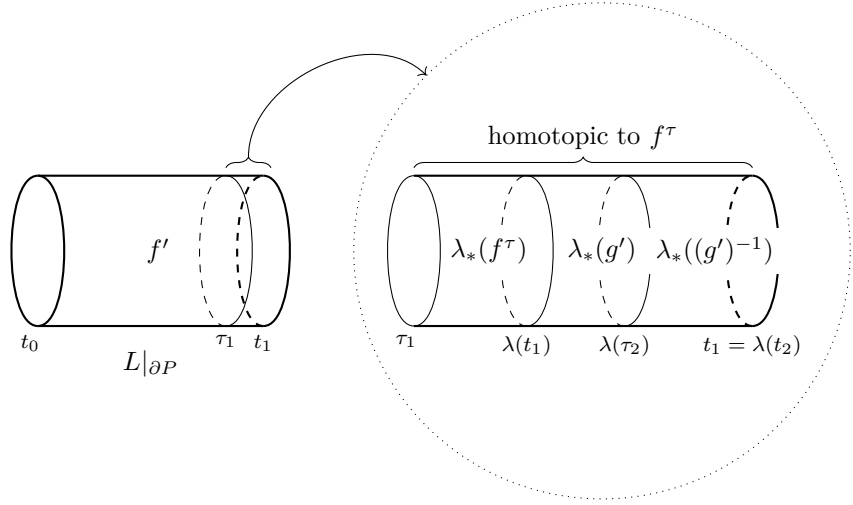
\begin{figure}
		\centering
		\begin{tikzpicture}
		\draw [thick] (0,1) to[out=0,in=0,looseness=.6] (0,-1);
		\draw [thick]  (0,-1) to[out=180,in=180,looseness=.6] (0,1);
		\node at (-.1,-1) [below] {\footnotesize $t_0$};
		
		\draw (2.5,1) to[out=0,in=0,looseness=.6] (2.5,-1);
		\draw [dashed] (2.5,-1) to[out=180,in=180,looseness=.6] (2.5,1);
		\node at (2.5,-1) [below] {\footnotesize $\tau_1$};	
		
		\draw [thick] (3,1) to[out=0,in=0,looseness=.6] (3,-1);
		\draw [thick] [dashed] (3,-1) to[out=180,in=180,looseness=.6] (3,1);
		\node at (3,-1) [below] {\footnotesize $t_1$};	
				
		\draw [thick] (0,1) -- (3,1);
		\draw [thick] (0,-1) -- (3,-1);	
		
		\draw [decorate,decoration={brace,amplitude=4pt},xshift=0pt,yshift=3pt]
		(2.5,1) -- (3.1,1);
		\draw [->] (2.8,1.25) to[out=90,in=135] ([shift=(135:3.3cm)]7.5,0);
		\draw [dotted] (7.5,0) circle (3.3 cm);
		
		\node at (1.6,0) {$f'$};
		\node at (1.5,-1.5) {$L|_{\partial P}$};
		
		\begin{scope}[xshift=5cm]
		\draw (0,1) to[out=0,in=0,looseness=.6] (0,-1);
		\draw (0,-1) to[out=180,in=180,looseness=.6] (0,1);
		\node at (-.1,-1) [below] {\footnotesize $\tau_1$};
		
		\draw (1.5,1) to[out=0,in=0,looseness=.6] (1.5,-1);
		\draw [dashed] (1.5,-1) to[out=180,in=180,looseness=.6] (1.5,1);
		\node at (1.5,-1) [below] {\footnotesize $\lambda(t_1)$};	
		
		\draw (2.8,1) to[out=0,in=0,looseness=.6] (2.8,-1);
		\draw [dashed] (2.8,-1) to[out=180,in=180,looseness=.6] (2.8,1);
		\node at (2.8,-1) [below] {\footnotesize $\lambda(\tau_2)$};	
		
		\draw [thick] (4.5,1) to[out=0,in=0,looseness=.6] (4.5,-1);
		\draw [thick] [dashed] (4.5,-1) to[out=180,in=180,looseness=.6] (4.5,1);
		\node at (4.5,-1) [below] {\footnotesize $t_1 = \lambda(t_2)$};	
		
		\draw [thick] (0,1) -- (4.5,1);
		\draw [thick] (0,-1) -- (4.5,-1);	
		\node at (1,0) [fill=white] {$\lambda_*(f^\tau)$};
		\node at (2.5,0) [fill=white] {$\lambda_*(g')$};
		\node at (4,0) [fill=white] {$\lambda_*((g')^{-1})$};
		\draw [decorate,decoration={brace,amplitude=4pt},xshift=0pt,yshift=3pt]
		(0,1) -- (4.5,1) node [black,midway,xshift=0cm,yshift=.4cm] {homotopic to $f^\tau$};
		\end{scope}

		\end{tikzpicture}
		\caption{An illustration of $L|_{\partial P} \sim f$, using $F'|_{\partial P} = g' \circledcirc f = g' \circledcirc f^\tau \circledcirc f$.}
		\label{fig.compcylhomotopy1}
	\end{figure}
	
	For (b), $(g^\tau \circledcirc g' \circledcirc -) \circ L \sim (g' \circledcirc -) \circ F' = F$, we introduce the notation $\lambda^1 \colon [t_1,t_2] \to [t_2,2t_2-t_1]$ and $\lambda^2 \colon [t_0,2t_2-t_1] \to [t_0,t_2]$ for the $\mr{CAT}$-isomorphisms determined by
	\[\lambda^1(t) =  t+(t_2-t_1), \qquad \text{and} \qquad \lambda^2(t) = \begin{cases} t+(t_1-t_2) & \text{if $t \in [t_2,2t_2-t_1]$,} \\
	\lambda(t) & \text{if $t \in [t_0,t_2]$.} \end{cases}\]
	In other words, $\lambda^1$ is just a translation and $\lambda^2$ an extension of $\lambda$ by a translation. We also pick a $\mr{CAT}$-isomorphism $\lambda^3 \colon [\tau_2,t_2] \to [\lambda(\tau_2),t_2]$ satisfying
	\[\lambda^3(t) = \begin{cases} t+(\lambda(\tau_2)-\tau_2) & \text{near $\tau_2$,} \\
	t & \text{near $t_2$,}\end{cases}\]
	and define a $\mr{CAT}$-isomorphism $\lambda^4 \colon [t_0,t_2] \to [t_0,t_2]$ by
	\[\lambda^4(t) = \begin{cases} \lambda^2(t) = \lambda(t) & \text{if $t \in [t_0,\tau_2]$,} \\
	\lambda^3(t) & \text{if $t \in [\tau_2,t_2]$.}
	\end{cases}\]
	
	Referring to Figure \ref{fig.compcylhomotopy2}, we write
	\begin{align*}(g^\tau \circledcirc g' \circledcirc -) \circ L &= (g^\tau \circledcirc g' \circledcirc -) \circ \lambda_*[((g')^{-1} \circledcirc -) \circ F'] \\
	&= \lambda^2_*[(\lambda^1_*(g^\tau \circledcirc g') \circledcirc -) \circ ((g')^{-1} \circledcirc -) \circ F'] \\
	&= (\lambda^2_*[\lambda^1_*(g^\tau \circledcirc g') \circledcirc (g')^{-1}] \circledcirc -) \circ \lambda^2_*(F') \\
	& \sim [\lambda^3_*(g^\tau) \circledcirc -] \circ \lambda^2_*(F') \\
	& = \lambda^4_*[(g^\tau \circledcirc -) \circ F'] \\
	& \sim (g^\tau \circledcirc -) \circ F' = F.\end{align*}
	The first three steps follow from the definitions, the fourth uses $\lambda^2_* [\lambda^1_*(g^\tau \circ g') \circ (g')^{-1}]$ and $\lambda^3_*(g^\tau)$ are both representatives in $\Psi(S^{n-1} \times [\tau_2,t_2] \rel \beta_2 \sqcup b_2)$ of the morphism $[g^\tau]$ in $[\cat{C}^\Psi]$. The second-to-last is a definition, and the last step uses that $\mr{CAT}$-invariance and the fact that $\lambda^4$ is isotopic to the identity. \\

	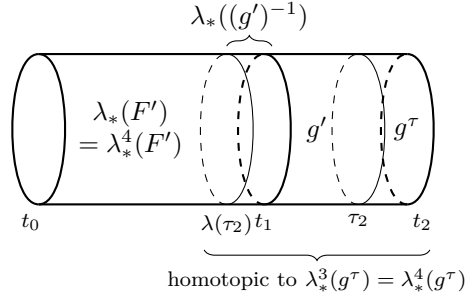
\begin{figure}
		\centering
		\begin{tikzpicture}
		\draw [thick] (0,1) to[out=0,in=0,looseness=.6] (0,-1);
		\draw [thick]  (0,-1) to[out=180,in=180,looseness=.6] (0,1);
		\node at (-.1,-1) [below] {\footnotesize $t_0$};
		
		\draw (2.5,1) to[out=0,in=0,looseness=.6] (2.5,-1);
		\draw [dashed] (2.5,-1) to[out=180,in=180,looseness=.6] (2.5,1);
		\node at (2.5,-1) [below] {\footnotesize $\lambda(\tau_2)$};	
		
		\draw [thick] (3,1) to[out=0,in=0,looseness=.6] (3,-1);
		\draw [thick] [dashed] (3,-1) to[out=180,in=180,looseness=.6] (3,1);
		\node at (3,-1) [below] {\footnotesize $t_1$};	
		
		\draw (4.25,1) to[out=0,in=0,looseness=.6] (4.25,-1);
		\draw [dashed] (4.25,-1) to[out=180,in=180,looseness=.6] (4.25,1);
		\node at (4.25,-1) [below] {\footnotesize $\tau_2$};	
		
		\draw [thick] (4.9,1) to[out=0,in=0,looseness=.6] (4.9,-1);
		\draw [thick] [dashed] (4.9,-1) to[out=180,in=180,looseness=.6] (4.9,1);
		\node at (5.1,-1) [below] {\footnotesize $t_2$};	
		
		\draw [thick] (0,1) -- (4.9,1);
		\draw [thick] (0,-1) -- (4.9,-1);
		
		\node at (1.2,0) {\parbox{1.5cm}{\centering $\lambda_*(F')$ \\ $=\lambda^4_*(F')$}};
		\node at (3.7,0) {$g'$};
		\node at (4.9,0) {$g^\tau$};
		
		\draw [decorate,decoration={brace,amplitude=4pt},xshift=0pt,yshift=3pt]
		(2.5,1) -- (3.1,1) node [black,midway,xshift=0cm,yshift=.4cm] {$\lambda_*((g')^{-1})$};
		
		\draw [decorate,decoration={brace,amplitude=4pt},xshift=0pt,yshift=-3pt]
		(5.2,-1.5) -- (2.2,-1.5) node [black,midway,xshift=0cm,yshift=-.4cm] {\footnotesize homotopic to $\lambda^3_*(g^\tau) = \lambda^4_*(g^\tau)$};
		
		\end{tikzpicture}
		\caption{An illustration of $(g^\tau \circledcirc g' \circledcirc -) \circ L \simeq F$. It explains why $(g^\tau \circledcirc g' \circledcirc -) \circ L$ is homotopic to $\lambda^4_*(g^\tau \circledcirc F)$, and since $\lambda^4 \colon [t_0,t_2] \to [t_0,t_2]$ is isotopic to $\mr{id}_{[t_0,t_2]}$ $\mr{CAT}$-invariance finishes the argument.}
		\label{fig.compcylhomotopy2}
	\end{figure}
	
	We continue with part (ii). The argument for part (i) does not give a weak equivalence because $(g \circledcirc - ) \circ L$ is not homotopic to $F$ rel $\partial P$. It does show that $g \circledcirc -$ is injective on homotopy groups, by Lemma \ref{lem.nearequivalence}.
	
	For surjectivity on homotopy groups under condition (W), fix a fillable path component $P$ of $\cat{C}^\Psi$ and identity element $\iota_P \in \cat{C}^\Psi(\la t_s,b_s \ra,\la t_e,b_e \la)$ in $P$. Let us first consider a morphism $g \in \cat{C}^\Psi(\la t_s,b_s \ra, \la t,b \ra )$ in $P$ with $t_s<t<t_e$. Note that
	\[(\iota_P \circledcirc -) \sim (h \circledcirc - ) \circ (g \circledcirc -)\]
	where $h$ is a representative in $\Psi(S^{n-1} \times [t,t_e] \rel b \sqcup b_e)$ of the morphism $[\iota_P \circ g^{-1}]$ in $[\cat{C}^\Psi]$. As $\iota_P \circledcirc -$ was assumed to be a weak equivalence, $h \circledcirc -$ is surjective on homotopy groups. But it was already known to be injective on homotopy groups, and hence is a weak equivalence. By the 2-out-of-3 property for weak equivalences, so is $g \circledcirc -$. Similarly, $g \circledcirc -$ is a weak equivalence for any $g \in \cat{C}^\Psi(\la t,b \ra, \la t_e,b_e \ra)$ in $P$ with $t_s<t<t_e$.
	
	Now suppose we are given an arbitrary morphism $g \in \cat{C}^\Psi( \la t_0,b_0 \ra, \la t_1,b_1 \ra)$ in $P$. Pick arbitrary $t_s<\tau_0<\tau_1<t_e$, then there is a $\mr{CAT}$-automorphism $\lambda \colon [t_0,t_1] \to [\tau_0,\tau_1]$ that satisfies
	\[\lambda(t) = \begin{cases} t+(\tau_0-t_0) & \text{near $t_0$,} \\
	t+(\tau_1-t_1) & \text{near $t_1$.}\end{cases}\]
	By $\mr{CAT}$-invariance, $g \circledcirc -$ is a weak equivalence if and only if $\lambda_* g \circledcirc -$ is. Thus we may assume $t_s<t_0<t_1<t_e$. The element $\lambda_* g$ lies in $P$ as well, and thus there is a zigzag of morphisms in $[\cat{C}^\Psi]$ between $b_s$ and $b_0$. Since $[\cat{C}]^\Psi$ is a groupoid, this implies there is a morphism from $b_s$ to $b_0$. Take a representative $k$ of this morphism in $\Psi(S^{n-1} \times [t_s,t_0] \rel b_s \sqcup b_0)$. Taking $h$ a representative of $[\iota_P \circ (g \circ k)^{-1}]$ in $\Psi(S^{n-1} \times [t_1,t_e] \rel b_1 \sqcup b_e)$, we get that
	\[(\iota_P \circledcirc -) \sim (h \circledcirc - ) \circ (g \circledcirc - ) \circ (k \circledcirc -).\]
	Since all of the maps except $g \circledcirc -$ are weak equivalences, so is $g \circledcirc -$.
\end{proof}

\subsection{Deformation lemma's}

The following lemma's will be useful in our proofs. They will be used to deform families to families that are locally constant in the parameter space.

\begin{lemma}\label{lem.ak} Suppose we are given a microflexible $\mr{CAT}$-invariant topological sheaf $\Psi$, an open neighborhood $U$ of $S^{n-1} \times \{0\}$ whose closure $\bar{U}$ is contained in $S^{n-1} \times (-1/2,1/2)$, and a family $f \colon \Delta^i \to \Psi(S^{n-1} \times \bR)$.
	
	Then there exists an open neighborhood $V$ of the barycenter $\beta$ contained in the interior of $\Delta^i$, such that for each $\mr{CAT}$-function $\eta \colon \Delta^i \to [0,1]$ with $\mr{supp}(\eta) \subset V$ and $\eta(\beta) = 1$, there is a homotopy $f_s \colon \Delta^i \times [0,1] \to \Psi(S^{n-1} \times \bR)$ satisfying
	\begin{enumerate}[(i)]
		\item \emph{the homotopy starts at $f$}: $f_0 = f$,
		\item \emph{the homotopy is rel $\partial \Delta^i$}: for all $d \in \Delta^i \setminus V$ and $s \in [0,1]$ we have $f_s(d) = f_0(d)$,
		\item \emph{the endpoint of the homotopy is constant near the barycenter}: for all $d \in \eta^{-1}(1)$, we have $f_1(d)|_U = f_1(\beta)|_U$,
		\item \emph{the homotopy is compactly supported}: for all $d \in \Delta^i$ and $s \in [0,1]$, with $W \coloneqq S^{n-1} \times ((-\infty,-1) \cup (1,\infty))$ we have $f_s(d)|_W = f_0(d)|_W$.\end{enumerate}
\end{lemma}

\begin{proof}Pick a family of $\mr{CAT}$-maps $\lambda \colon [0,1] \to \mr{CAT}(\Delta^i,\Delta^i)$  satisfying:
	\begin{enumerate}[(i)]
		\item $\lambda(0) = \mr{id}$,
		\item $\lambda(s)|_{\partial \Delta^i} = \mr{id}$,
		\item for $s \geq 1/4$, $\lambda(s)$ collapses $\Delta^i_{s-1/4}(\beta)$ to the barycenter.
	\end{enumerate}
	Then we can define a map
	\begin{align*} F \colon \Delta^i \times [0,1] &\longto \Psi\left(\bar{W} \sqcup \bar{U} \subset S^{n-1} \times \bR\right) \\
	(d,s) &\longmapsto \begin{cases} f(d) & \text{on $\bar{W}$,} \\
	f(\kappa_s(d)) & \text{on $\bar{U}$.}\end{cases}\end{align*}
	By construction this fits into a commutative diagram
	\[\begin{tikzcd} (\Delta^i \times \{0\}) \cup (\partial \Delta^i \times [0,1]) \cup (\{\beta\} \times [0,1]) \rar{f \circ \pi_1} \dar] &  \Psi(S^{n-1} \times \bR) \dar \\
	\Delta^i \times [0,1] \rar[swap]{F} & \Psi(\bar{W} \sqcup \bar{U} \subset S^{n-1} \times \bR).\end{tikzcd}\] 
	 Since the right map is a microfibration, there exists an open neighborhood $Y$ of $(\Delta^i \times \{0\}) \cup (\partial \Delta^i \times [0,1]) \cup (\{\beta\} \times [0,1])$ in $\Delta^i \times [0,1]$ and a partial lift $L \colon Y \to \Psi(S^{n-1} \times \bR)$. Since $Y$ contains $\{\beta\} \times [0,1]$, there is an open neighborhood $V \subset \mr{int}(\Delta^i)$ of $\beta$ such that $V \times [0,1] \subset Y$. Then the map $\Delta^i \times [0,1] \ni (d,t) \mapsto L(d,t\eta(d))$ is the desired homotopy. \end{proof}

%Now pick a $\mr{CAT}$-function $\eta \colon \Delta^i \times [0,1]$ that
%\begin{enumerate}[(i')]
%	\item equals $1$ near the barycenter $\beta$ of $\Delta^i$, and
%	\item has the property that $(d,t\eta(d)) \in U$ for all $t \in [0,1]$ and $d \in \Delta^i$.
%\end{enumerate}
%Then the map $\Delta^i \times [0,1] \ni (d,t) \mapsto L(d,t\eta(d))$ is the desired homotopy.
%\end{proof}

\begin{lemma}\label{lem.akflex} Suppose we are given a map $j \colon \Psi \to \Gamma$ of $\mr{CAT}$-invariant topological sheaves with $\Psi$ microflexible and $\Gamma$ flexible. Suppose we are further given an open neighborhood $U$ of $S^{n-1} \times \{0\}$ whose closure $\bar{U}$  is contained in $S^{n-1} \times (-1/2,1/2)$. If $j \colon \Psi(S^{n-1} \times \bR) \to \Gamma(S^{n-1} \times \bR)$ is a $\pi_0$-surjection, then for any family $f \colon \Delta^i \to \Gamma(S^{n-1} \times \bR)$, there exist
	\begin{itemize}
		\item an open neighborhood $V$ of the barycenter $\beta$ contained in the interior of $\Delta^i$,
		\item an open neighborhood $V'$ of the barycenter $\beta$ contained in $V$, 
		\item an element $g \in \Psi(S^{n-1} \times \bR)$,
	\end{itemize} 
	such that for each $\mr{CAT}$-function $\lambda \colon \Delta^i \to [0,1]$ with $\mr{supp}(\lambda) \subset V'$ and $\lambda(\beta) = 1$, there is a homotopy $f_s \colon \Delta^i \times [0,1] \to \Psi(S^{n-1} \times \bR)$ satisfying
	\begin{enumerate}[(i)]
		\item \emph{the homotopy starts at $f$}: $f_0 = f$,
		\item \emph{the homotopy is rel $\partial \Delta^i$}: for all $d \in \Delta^i \setminus V$ and $s \in [0,1]$ we have $f_s(d) = f_0(d)$,
		\item \emph{the endpoint of the homotopy is constant near the barycenter}: for all $d \in \lambda^{-1}(1)$, we have $f_1(d)|_{U} = f_1(\beta)|_{U}$,
		\item \emph{the endpoint of the homotopy equal $g$ near the barycenter}: $f_1(\beta)|_{U} = j(g)|_{U}$,
		\item \emph{the homotopy is compactly supported}: for all $d \in \Delta^i$ and $s \in [0,1]$, with $W \coloneqq S^{n-1} \times ((-\infty,-1) \cup (1,\infty))$ we have $f_s(d)|_W = f_0(d)|_W$.\end{enumerate}
\end{lemma}

\begin{proof}We apply Lemma \ref{lem.ak} to $f$ to obtain $V$; then after picking an $\eta$ that is $1$ near $\beta$ and supported in $V$, we may assume without loss of generality that $f_0(d)$ is independent of $d$ on $U$ in a neighborhood $V'$ of the barycenter. 
	
	The hypothesis that $j$ is a $\pi_0$-surjection implies that that for any $s_0 \in \Gamma(S^{n-1} \times \bR)$ we can find a path $s_t \colon [0,1] \longto \Gamma(S^{n-1} \times \bR)$ such that $s_1 = j(g)$ for some $g \in \Psi(S^{n-1} \times \bR)$. We may then define a map 
	\begin{align*}S \colon [0,1] &\longto \Gamma(\bar{W} \sqcup \bar{U} \subset S^{n-1} \times \bR) \\
	t &\longmapsto \begin{cases} s_0 & \text{on $\bar{W}$,} \\
	s_t & \text{on $\bar{U}$.}\end{cases}\end{align*}
	This fits into a commutative diagram
	\[\begin{tikzcd} \{0\} \rar{s_0}  \dar & \Gamma(S^{n-1} \times \bR) \dar \\ {[0,1]} \arrow[dotted]{ur} \rar[swap]{S} & \Gamma(\bar{W} \sqcup \bar{U} \subset S^{n-1} \times \bR) \end{tikzcd}\]
	and flexibility provides a dotted lift $\ell$. Our homotopy is then given by 
	\[(d,s) \longmapsto \begin{cases} f_0(d) & \text{if $\lambda(d) = 0$,} \\
	\ell(\lambda(d)\cdot s) & \text{if $\lambda(d)\neq 0$.}\end{cases}\]
\end{proof}

\subsection{Resolving}\label{sec:resolving} Let $M$ be manifold with embedding $\alpha \colon S^{n-1} \times \bR \hookrightarrow M$. Further suppose we have a closed subset $A \subset M$ disjoint from the image of $\alpha$, and a boundary condition $a \in \cB_\Psi(A \subset M)$. We will ``resolve'' the space $\Psi(M \rel a)$ by cutting along lines $S^{n-1} \times \{t\}$. The result will be a weakly-equivalent semisimplicial space, and for background on semisimplicial spaces we recommend \cite{rwebertsemi}.

\begin{definition}The augmented semisimplicial space $\Psi(M \rel a,\alpha)_\bullet$ is defined as follows.
	\begin{itemize}
		\item The space of $p$-simplices is given by a disjoint union over the indexing set of $(p+1)$-tuples $(\la t_0,b_0 \ra,\ldots,\la t_p,b_p\ra)$ of real numbers $t_i$ and boundary conditions $b_i \in \cB_\Psi(S^{n-1} \times \{t_i\})$ so that $t_0<\ldots<t_p$, of terms given by
		\[\Psi\left(M \setminus S^{n-1} \times [t_0,t_p] \rel b_0 \sqcup b_p \sqcup a\right) \times \prod_{i=0}^{p-1} \Psi\left(S^{n-1} \times [t_i,t_{i+1}] \rel b_i \sqcup b_{i+1}\right).\]
		\item The $(-1)$-simplices are given by $\Psi(M \rel a)$.
		\item The $i$th face map forgets $t_i$ and $b_i$, and uses the sheaf property to glue the elements of $\Psi$.
		\item The augmentation map $\epsilon$ forgets the real numbers and boundary conditions,  and uses the sheaf property to glue the elements of $\Psi$.
	\end{itemize}
\end{definition}

\begin{proposition}\label{prop.resolve}Suppose that $\Psi$ is a $\mr{CAT}$-invariant microflexible sheaf. Then the augmentation $\epsilon$ induces a weak equivalence $\epsilon \colon ||\Psi(M \rel a,\alpha)_\bullet|| \to \Psi(M \rel a)$.\end{proposition}

\begin{proof}Since all our manipulations will be compactly-supported in the image of $\alpha$, we may assume $A = \varnothing$ and consequently simplify the notation. We will prove that $\epsilon$ is a weak equivalence using Lemma \ref{lem.ak}. We do this by showing for all compact parametrizing manifolds $P$ with boundary $\partial P$, there exists a dotted map in each commutative diagram
	\[\begin{tikzcd} \partial P \rar{f} \dar & {||\Psi(M,\alpha)_\bullet||} \dar{\epsilon} \\
	P \rar[swap]{F} \arrow[dotted]{ru} & \Psi(M) \end{tikzcd}\]
	making the top triangle commute and the bottom triangle commute up to homotopy rel $\partial P$. Without loss of generality we may assume that a dotted map extending $f$ exists over a collar $U$ of $\partial P$ in $P$. We shall write $i$ for $\dim(P)$.
	
	\vspace{.5ex}
	
	\noindent \textbf{Claim:} We may assume that for each $p \in P$ there is at least one $t \in \bR$ such that the restriction of $F_{p'}$ near $S^{n-1} \times \{t\}$ is independent of $p'$ near $p$.
	
	\vspace{-.5ex}
	
	\begin{proof}Near $x \in U \setminus \partial P$ the map $f$ may be heuristically described as an element $f(x) \in ||\Psi(M,\alpha)_\bullet||$ given by a map $P \to \Psi(M)$ together with a finite collection of slices $\{t_k\}$ with weights $\{s_k\}$ in $[0,1]$ summing to $1$. By compactness of $\partial P$, we may assume that over $U$ only a finite collection $\{t_\ell\}_{\ell \in L}$ of values $t_\ell \in \bR$ is used for the slices. 
		
		Pick $(i+2)$ pairwise disjoint intervals $\{[a_j-\eta_j,a_j+\eta_j]\}_{0 \leq j \leq i+1}$ in $\bR$ which are disjoint from the finite set $\{t_\ell\}_{\ell \in L}$. We also pick for each $p \in \mr{int}(P)$ an embedded simplex $\Delta^{i}(p) \subset \mr{int}(P)$ with barycenter mapping to $p$.
		
		For each $p \in \mr{int}(P)$ and $0 \leq j \leq i+1$, apply Lemma \ref{lem.ak} to the family obtained from $F$ by identifying $(\Delta^i(p),\partial \Delta^i(p))$ with $(\Delta^{i},\partial \Delta^i)$ and identifying $S^{n-1} \times (a_j-\eta_j,a_j+\eta_j)$ with $S^{n-1} \times \bR$. Its conclusion is that there exists open neighborhoods $U_j(p)$ of $S^{n-1} \times \{a_j\} \subset S^{n-1} \times \bR$ and $V_j(p)$ of $p \in P$, such that for each $\mr{CAT}$-function $\eta \colon P \to [0,1]$ supported in $V_j(p)$ there is a deformation $F_s$ supported in $V_j(p)$ in the domain and in $S^{n-1} \times (a_j-\eta_j,a_j+\eta_j)$ in the codomain, which starts at $F_0 = F$ and ends at an $F_1$ which satisfies $F_1(p')|_{U_j(p)} = F_1(p)|_{U_j(p)}$ for $p' \in \eta^{-1}(1)$. By replacing $U_j(p)$ with $U(p) \coloneqq \smash{\cap_{j=0}^{i+1}} U_j(p)$ and $V_j(p)$ by $V(p) \coloneqq \smash{\cap_{j=0}^{i+1}} V_j(p)$, we may without loss of generality assume that $U_j(p) = U_{j'}(p)$ and $V_j(p) = V_{j'}(p)$ for all $0 \leq j,j' \leq i+1$, and thus we henceforth drop the subscripts.
		
		Since $P$ is an $i$-dimensional $\mr{CAT}$-manifold, it has Lebesgue covering dimension $i$. This means that there are open subsets $\tilde{U}(p) \subset U(p)$ and $\tilde{U} \subset U$ which (a) cover $P$, and (b) have the property that each $q \in P$ is contained in at most $\dim(P)+1$ of these subsets. Since $P$ is compact, we may assume that only finitely many of these are non-empty: $\tilde{U}$ and  $\tilde{U}(p_k)$ for $1 \leq k \leq K$. Since $P$ is paracompact we may pick a subordinate partition of unity given by $\sigma_{\tilde{U}}$ and $\{\sigma_{\tilde{U}(p_k)}\}_{1 \leq k \leq K}$ consisting of $\mr{CAT}$-functions with values in $[0,1]$, and we let $\mu > 0$ denote the minimum over all $p \in P$ of $\max\{\sigma_{\tilde{U}}(p),\max \{\sigma_{\tilde{U}(p_k)}(p) \mid 1 \leq k \leq K\}\}$.
		
		The finite graph with a vertex for each element of this cover and edge whenever two elements have non-empty intersection, has vertices of valence $\leq i+1$. Hence its chromatic number is $\leq i+2$, and we can assign to each $1 \leq k \leq K$ a number $0 \leq c(k) \leq i+1$ such that for $c(k) \neq c(k')$ if $k \neq k'$ and $\tilde{U}(p_k) \cap \tilde{U}(p_{k'}) \neq \varnothing$.
		
		Fix a $\mr{CAT}$-function $\rho \colon [0,\infty) \to [0,1]$ that is the identity on $[0,1/2]$ and $1$ on $[1,\infty)$. Now apply for each $1 \leq k \leq K$ the deformation of $F$ obtained from the function $\rho(\frac{2}{\mu} \sigma_{\tilde{U}(p_k)}) \colon P \to [0,1]$ in the codomain interval $S^{n-1} \times (a_{c(k)}-\eta_{c(k)},a_{c(k)}+\eta_{c(k)})$. It is possible to do all of these simultaneously, because whenever $\tilde{U}(p_k) \cap \tilde{U}(p_{k'}) \neq \varnothing$ these deformations have disjoint support in the codomain. The result is a new family $\tilde{F}$ such that for each $p \in P$ there is at least one $t \in \bR$ such that the restriction of $\tilde{F}_{p'}$ to $S^{n-1} \times \{t\}$ is independent of $p'$ near $p$.\end{proof}
	
	\vspace{.5ex}
	
	The now-proven claim tells us there is an open cover $\{U_j\}_{j \in J}$ of $\mr{int}(P)$ with the property that for each $U_j$ there exists a $t_j$ such that for $p \in U_j$ the restriction of $F_{p'}$ to $S^{n-1} \times \{t\}$ is independent of $p'$ near $p$. We then construct the desired lift of $F$ to $||\Psi(M,\alpha)_\bullet||$ as follows. Pick a partition of unity subordinate to the aforementioned open cover, denoting the function supported in $U$ by $\sigma_U$ and the function with support in $U_j$ by $\sigma_j$. Then the slices and simplicial coordinates over $p \in P$ are given by $t_j$ with weight $\sigma_j(p)$, and if $p \in U$ we add to these the slices $\{t_k\}$ with weights $\{\sigma_U(p)s_k(p)\}$.\end{proof}

We need a naturality statement for the map of the previous proposition. Before stating it, we make the following definition.

\begin{definition}\label{def.collarresolvej} Let $j \colon \Psi \to \Gamma$ be a morphism of $\mr{CAT}$-invariant topological sheaves. The augmented semisimplicial space $\Gamma(M \rel j(a),\alpha,j)_\bullet$ is defined as follows.
	\begin{itemize}
		\item The space of $p$-simplices is given by a disjoint union over the indexing set of $(p+1)$-tuples $(\la t_0,b_0 \ra,\ldots,\la t_p,b_p \ra)$ of real numbers $t_i$ and boundary conditions $b_i \in \cB_\Psi(S^{n-1} \times \{t_i\})$ so that $t_0<\ldots<t_p$, of terms given by
		\begin{align*}&\Gamma(M \setminus S^{n-1} \times [t_0,t_p] \rel j(b_0) \sqcup j(b_p) \sqcup j(a)) \\
		&\qquad \times \prod_{i=0}^{p-1} \Gamma(S^{n-1} \times [t_i,t_{i+1}] \rel j(b_i) \sqcup j(b_{i+1})).\end{align*}
		\item The $(-1)$-simplices are given by $\Gamma(M \rel j(a))$.
		\item The $i$th face map forgets $t_i$ and $j(b_i)$, and uses the sheaf property to glue the elements of $\Gamma$.
		\item The augmentation map $\epsilon$ forgets the real numbers and boundary conditions, and uses the sheaf property to glue the elements of $\Gamma$.
	\end{itemize} 
\end{definition}

The condition that $j \colon \Psi(S^{n-1} \times \bR) \to \Gamma(S^{n-1} \times \bR)$ is a $\pi_0$-surjection in the next proposition, is implied by $\Psi(\bR^n) \to \Phi(\bR^n)$ being a weak equivalence by Theorem \ref{thm.gromovprecise}. Using the same proof as for Proposition \ref{prop.resolve}, but replacing Lemma \ref{lem.ak} by Lemma \ref{lem.akflex}, we obtain:

\begin{proposition}\label{prop.resolvefolex}Let $j \colon \Psi \to \Gamma$ be a map of $\mr{CAT}$-invariant topological sheaves. Suppose that $\Psi$ is microflexible, $\Gamma$ is flexible and $j \colon \Psi(S^{n-1} \times \bR) \to \Gamma(S^{n-1} \times \bR)$ be a $\pi_0$-surjection. Then the augmentation $\epsilon$ induces a weak equivalence $\epsilon \colon ||\Gamma(M \rel j(a),\alpha,j)_\bullet|| \to \Gamma(M \rel j(a))$.\end{proposition}

The map $j$ induces a semisimplicial map $j_\bullet$ between augmented semisimplicial objects $\Psi(M \rel a,\alpha)_\bullet$ and $\Gamma(M \rel a,\alpha,j)_\bullet$. The reason we used a different indexing set in Definition \ref{def.collarresolvej}, is to make the second part of the following proposition hold.

\begin{proposition}\label{prop.rescommute} The following diagram commutes			
	\[\begin{tikzcd} {||\Psi(M \rel a,\alpha)_\bullet||} \rar{||j_\bullet||} \dar & {||\Gamma(M \rel j(a),\alpha,j)_\bullet||} \dar \\
	\Psi(M \rel a) \rar[swap]{j} & \Gamma(M \rel j(a)) \end{tikzcd}\]
	and for each $k \geq 0$, the map $j_k$ induces a bijection on the indexing sets of the disjoint union appearing in the definitions of the spaces of $k$-simplices.
\end{proposition}

\subsection{The proof} In this subsection we prove Theorem \ref{thm.technical}. The proof is as follows:
\begin{enumerate}[(i)]
	\item We define a functor $\cat{F}^\Psi(M) \colon \cat{C}^\Psi \to \cat{S}$.
	\item We show that each morphism in $\cat{C}^\Psi$ induces on $\cat{F}^\Psi(M)$ a homology equivalence if condition (H) is satisfied, or weak equivalence if condition (W) is satisfied.
	\item Using a delooping result of McDuff-Segal, we prove that $j \colon \Psi(M \rel a) \to \Gamma(M \rel j(a))$ is a homology equivalence or weak equivalence when $M$ is path-connected, $S^{n-1} \subset \partial M \cap A$ and $a|_{S^{n-1}}$ is fillable. More precisely, we use Section \ref{sec:resolving} and (ii) to deduce this from the $h$-principle on open manifolds, here Theorem \ref{thm.gromovprecise}.
	\item This is then used to prove that $j \colon \Psi(M \rel b) \to \Gamma(M \rel j(b))$ is a homology equivalence or weak equivalence without conditions.
\end{enumerate}

\subsubsection{A functor}
Let us fix a manifold $M$ with $\partial M = S^{n-1}$. In this section we prove steps (i) and (ii).

\begin{definition}We let $\cat{F}^\Psi(M)$ be the functor $\cat{C}^\Psi \to \cat{S}$ sending an object $\la t,b\ra$ to the space $\Psi(M \cup_{S^{n-1}} (S^{n-1} \times [0,t]) \rel b)$ and a morphism $g \in \Psi(S^{n-1} \times [t_0,t_1] \rel b_0 \sqcup b_1)$ to the map
	\[g \odot - \colon \Psi\left(M \cup_{S^{n-1}} (S^{n-1} \times [0,t_0]) \rel b_0\right) \longto \Psi\left(M \cup_{S^{n-1}} (S^{n-1} \times [0,t_2]) \rel b_2\right)\]
	obtained by gluing on $g$ using the sheaf property.
\end{definition}

\begin{proposition}\label{prop.acteq} Let $\Psi$ be a microflexible $\mr{CAT}$-invariant topological sheaf on $n$-manifolds. For $0<t_1<t_2$ and an element
	\[g \in \cat{C}^\Psi(\la t_1,b_1 \ra, \la t_2,b_2 \ra ) = \Psi\left(S^{n-1} \times [t_1,t_2] \rel b_1 \sqcup b_2\right)\]
	consider the map
	\[\begin{tikzcd} \cat{F}^\Psi(M)(\la t_1,b_1 \ra ) = \Psi\left(M \cup_{S^{n-1}} (S^{n-1} \times [0,t_1]) \rel b_1\right) \dar{g \odot -} \\
	\cat{F}^\Psi(M)(\la t_2,b_2 \ra ) = \Psi\left(M \cup_{S^{n-1}} (S^{n-1} \times [0,t_2]) \rel b_2\right).\end{tikzcd}\]
	\begin{enumerate}[(i)]
		\item If $\Psi$ satisfies condition (H), the map $g \odot -$ is a semi-equivalence as in Definition \ref{def.semiequivalence}.
		\item If $\Psi$ satisfies condition (W) and $b_1$ (or equivalently $b_2$) is fillable, the map  $g \odot -$ is a weak equivalence.
	\end{enumerate}
\end{proposition}

\begin{proof}Any choice of embedding $\bR \hookrightarrow [0,t_0]$ induces an embedding $\alpha \colon S^{n-1} \times \bR \hookrightarrow M \cup_{S^{n-1}} (S^{n-1} \times [0,t_0])$. By Proposition \ref{prop.resolve} the map $\epsilon \colon ||\Psi(M \cup_{S^{n-1}} (S^{n-1} \times [0,t_0]) \rel b, \alpha)|| \to \Psi(M \cup_{S^{n-1}} (S^{n-1} \times [0,t_0]) \rel b)$ is a weak equivalence. The map $g \odot -$ induces a semisimplicial map
	\[\begin{tikzcd}\Psi\left(M \cup_{S^{n-1}} (S^{n-1} \times [0,t_1]) \rel b_1,\alpha\right)_\bullet \dar{(g \odot -)_\bullet} \\ \Psi\left(M \cup_{S^{n-1}} (S^{n-1} \times [0,t_2]) \rel b_2,\alpha\right)_\bullet \end{tikzcd}\]
	which is a levelwise homology equivalence (resp.\ weak equivalence) under condition (H) (resp.\ (W)) by Proposition \ref{prop.compcyl}. A geometric realization of a levelwise homology equivalence (resp.\ levelwise weak equivalence) of semisimplicial spaces is a homology equivalence (resp.\ weak equivalence), e.g.~by the geometric realization spectral sequence in Section 1.4 of \cite{rwebertsemi} (resp.\ Theorem 2.2 of \cite{rwebertsemi}).
\end{proof}

\subsubsection{Intermezzo on delooping} We now recall two well-known results by Segal and McDuff-Segal, for which we recommend \cite{rwebertsemi} as a reference. Recall a commutative diagram in $\cat{S}$
\[\begin{tikzcd} E \rar \dar{p} & E' \dar{p'} \\
B \rar{f} & B'.\end{tikzcd}\]
is \emph{homotopy cartesian} (resp.\ \emph{homology cartesian}) if $f \colon \pi_0(B) \to \pi_0(B')$ is surjective, and for all $b \in B$ the induced map $\mr{hofib}_b(p) \to \mr{hofib}_{f(b)}(p')$ is a weak equivalence (resp.\ homology equivalence).

\begin{theorem}[McDuff-Segal] \label{thm.mcduffsegal} Let $f_\bullet \colon E_\bullet \to B_\bullet$ be a map of semisimplicial objects in $\cat{S}$ such that for all $0 \leq i \leq p$ the diagram
	\[\begin{tikzcd} E_p \rar{d_i} \dar[swap]{f_p} & E_{p-1} \dar{f_{p-1}} \\
	B_p \rar[swap]{d_i} & B_{p-1}.\end{tikzcd}\]
	%\[\xymatrix{E_p \ar[r]^{d_i} \ar[d]_{f_p} & E_{p-1} \ar[d]^{f_{p-1}} \\
	%B_p \ar[r]_{d_i} & B_{p-1}}\]
	is a homotopy cartesian (resp.\ homology cartesian), then 
	\[\begin{tikzcd} E_0 \rar \dar[swap]{f_0} & {||E_\bullet||} \dar{||f_\bullet||} \\
	B_0 \rar & {||B_\bullet||} \end{tikzcd}\]
	%\[\xymatrix{E_0 \ar[r] \ar[d]_{f_0} & ||E_\bullet|| \ar[d]^{||f_\bullet||} \\
	%B_0 \ar[r] & ||B_\bullet||}\]
	is also homotopy cartesian (resp.\ homology cartesian).\end{theorem}

\begin{proof}The homotopy cartesian part is a consequence of the proof of Proposition 1.6 of \cite{segalcategories}. The statement of that proposition only differs from ours in its use of simplicial spaces and the thin geometric realization, which are replaced by semisimplicial spaces and the thick geometric realization in the first line of the proof. A further inspection of its proof shows that Proposition 3 of \cite{MSe} implies the homology cartesian part. For modern proofs, see Theorems 2.12 and 6.5 of \cite{rwebertsemi}.\end{proof}

Our application concerns a \emph{double-sided bar construction} for a (possibly non-unital) topological category $\cat{C}$ and two functors $F\colon \cat{C} \to \cat{S}$, $G \colon \cat{C}^\mr{op} \to \cat{S}$. We will also assume our category $\cat{C}$ has a functor $t \colon \cat{C} \to (\bR,<)$, where $(\bR,<)$ is the non-unital poset category.

\begin{definition}We let $B_\bullet(F,\cat{C},G)$ denote the semisimplicial space with space of $k$-simplices given by the disjoint union over $(k+1)$-tuples $(c_0,\ldots,c_k)$ of objects of $\cat{C}$ satisfying $t(c_0)<\ldots<t(c_k)$ of terms given by
	\[F(c_0) \times \left(\prod_{i=0}^{k-1} \cat{C}(c_i,c_{i+1})\right) \times G(c_k).\]
	The face maps $d_i$ are obtained by applying $F$ for $i=0$, composition in $\cat{C}$ for $0 < i < k$, and by applying $G$ for $i=k$.
	
	We let $B(F,\cat{C},G)$ denote the geometric realization $||B_\bullet(F,\cat{C},G)||$.\end{definition}

In this setting, the conditions in Theorem \ref{thm.mcduffsegal} amount to the following Lemma. To simplify notation, we note there are terminal functors $\ast \colon \cat{C} \to \cat{S}$ and $\ast \colon \cat{C}^\mr{op} \to \cat{S}$ sending every object to the point. We denote $B(\ast,\cat{C},\ast)$ by $B\cat{C}$, as $B_\bullet(\ast,\cat{C},\ast)$ coincides with the nerve of $\cat{C}$.

\begin{lemma}\label{lem.catloop} Let $\cat{C}$, $F$ and $G$ be above. Then the following hold:
	\begin{enumerate}[(a)]
		\item If all morphisms $f \colon c \to c'$ induce homology equivalences $f \circ - \colon \cat{C}(c'',c) \to \cat{C}(c'',c')$ if $t(c'')<t(c)$, $- \circ f \colon \cat{C}(c',c'') \to \cat{C}(c,c'')$ if $t(c')<t(c'')$, $F(f) \colon F(c) \to F(c')$ and $G(f) \colon G(c') \to G(c)$, then 
		\[F(c) \times G(c) \longto \mr{hofib}_c\left(B(F,\cat{C},G) \longto B\cat{C}\right)\]
		is a homology equivalence.
		\item If all morphisms $f \colon c \to c'$ induce weak equivalences $f \circ - \colon \cat{C}(c'',c) \to \cat{C}(c'',c')$ if $t(c'')<t(c)$, $- \circ f \colon \cat{C}(c',c'') \to \cat{C}(c,c'')$ if $t(c')<t(c'')$, $F(f) \colon F(c) \to F(c')$ and $G(f) \colon G(c') \to G(c)$, then 
		\[F(c) \times G(c) \longto \mr{hofib}_c\left(B(F,\cat{C},G) \longto B\cat{C}\right)\]
		is a weak equivalence.
\end{enumerate}\end{lemma}

\begin{remark}Theorem \ref{thm.mcduffsegal} has improvements involving homology with coefficients in various types of local systems. These could be used to give variations of Theorem \ref{thm.main} for homology equivalences with such coefficients. We have not found a use for this.\end{remark}

\subsubsection{The proof for spherical boundary} Fix a manifold $M$ with closed subset $A \subset M$ and a choice of embedded $S^{n-1} \subset \partial M \setminus A$. We will discuss when $\Psi(M \rel a \sqcup b) \to \Gamma(M \rel j(a \sqcup b))$ is a homology or weak equivalence for a boundary condition $a \in \cB_\Psi(A)$ and a (fillable) boundary condition $b \in \cB_\Psi(S^{n-1})$, that is, step (iii). This is the only place where we use delooping, as described in the previous section. We will also need the following easy lemma.

\begin{lemma}\label{lem.halfcontr} Let $\Psi$ be a $\mr{CAT}$-invariant topological sheaf. Then $\Psi(S^{n-1} \times [0,\infty) \rel b)$ is weakly contractible for all boundary conditions $b \in \cB_\Psi(S^{n-1} \times \{0\})$.\end{lemma}

\begin{proof}It suffices to prove that for all  compact parametrizing manifolds $P$ with boundary $\partial P$, any map $f \colon \partial P \to \Psi(S^{n-1} \times [0,\infty) \rel b)$ can be extended to a map $P \to \Psi(S^{n-1} \times [0,\infty) \rel b)$. Since we working in $\cat{S} = \cat{Sh^{CAT}}$, there exists an open neighborhood $U$ of $S^{n-1} \times \{0\}$ in $S^{n-1} \times [0,\infty)$ and $g \in \Psi(U)$ such that for all $t \in \partial P$ we have that $(f_t)|_U = g$. 
	
	Now pick a $\mr{CAT}$-isotopy $\psi_s$ of embeddings $S^{n-1} \times [0,\infty) \to S^{n-1} \times [0,\infty)$ such that \begin{enumerate}[(i)]
		\item $\psi_0 = \mr{id}$, 
		\item $\psi_1$ has image in $U$, 
		\item there is a neighborhood $V$ of $S^{n-1} \times \{0\}$ so that $\psi_s|_V = \mr{id}$ for all $s \in [0,1]$.
	\end{enumerate}
	Extend $g$ to a collar $\partial P \times [0,1]$ by $(t,s) \mapsto (f_t) \circ \psi_s$. On $\partial P \times \{1\}$, this is independent of $p \in \partial P$ and hence can be extended to $P$.\end{proof}

\begin{proposition}\label{prop.thmcylindricalbdy} Let $j \colon \Psi \to \Gamma$ be a map of $\mr{CAT}$-invariant sheaves. Suppose that $\Psi$ is microflexible, $\Gamma$ is flexible and $j \colon \Psi(\bR^n) \to \Gamma(\bR^n)$ is a weak equivalence.
	
	Let $M$ be a manifold, and $A \subset M$ a closed subset. Suppose that $S^{n-1} \subset \partial M \setminus A$, and $M \setminus S^{n-1}$ has no path components with compact closure. Fix boundary conditions $a \in \cB_\Psi(A)$ and $b \in \cB_\Psi(S^{n-1})$, and consider the map 
	\[j \colon \Psi(M \rel a \sqcup b) \longto \Gamma(M \rel j(a \sqcup b)).\]
	\begin{enumerate}[(i)]
		\item If $\Psi$ satisfies condition (H), this map is a homology equivalence.
		\item If $\Psi$ satisfies condition (W) and $b$ is fillable, this map is a weak equivalence.
	\end{enumerate}
\end{proposition}

\begin{proof}We claim that there are weak equivalences 
	\[B(\cat{F}^\Psi(M),\cat{C}^\Psi,\ast) \simeq \Psi(M \cup_{S^{n-1}} (S^{n-1} \times [0,\infty))),\]
	\[B(\cat{C}^\Psi) \simeq \Psi(S^{n-1} \times \bR).\]
	We give the proof in the first case, the second being similar. An embedding $\bR \hookrightarrow [0,\infty)$ gives an embedding $\alpha \colon S^{n-1} \times \bR \to M \cup_{S^{n-1}} (S^{n-1} \times [0,\infty))$. By Proposition \ref{prop.resolve} the map $\epsilon \colon || \Psi(M \cup_{S^{n-1}} (S^{n-1} \times [0,\infty)),\alpha)_\bullet|| \to  \Psi(M \cup_{S^{n-1}} (S^{n-1} \times [0,\infty)))$ is a weak equivalence. Now note that there is a semisimplicial map $\Psi(M \cup_{S^{n-1}} (S^{n-1} \times [0,\infty)),\alpha)_\bullet \to B_\bullet(\cat{F}^\Psi(M),\cat{C}^\Psi,\ast)$ obtained by levelwise projecting away terms of the form $\Psi(S^{n-1} \times [t_p,\infty) \rel b_p)$. These are contractible by Lemma \ref{lem.halfcontr} and thus the realization of this semisimplicial map is a weak equivalence, e.g.~by Theorem 2.2 of \cite{rwebertsemi}.
	
	It is easy to see from the proof that this weak equivalence in natural in $\Psi$ and hence it follows from Gromov's $h$-principle --- here Theorem \ref{thm.gromovprecise}, which uses the assumption that $\Psi(\bR^n) \to \Gamma(\bR^n)$ is a weak equivalence --- that in the commutative diagram
	\[\begin{tikzcd} B(\cat{F}^\Psi(M),\cat{C}^\Psi,\ast) \rar \dar & \dar  B(\cat{F}^{\Gamma}(M),\cat{C}^{\Gamma},\ast) \\
	B(\cat{C}^\Psi) \rar & B(\cat{C}^{\Gamma}) \end{tikzcd}\]
	%\[\xymatrix{B(\cat{F}^\Psi(M),\cat{C}^\Psi,\ast) \ar[r] \ar[d] & \ar[d]  B(\cat{F}^{\Gamma}(M),\cat{C}^{\Gamma},\ast) \ar[d] \\
	%B(\cat{C}^\Psi) \ar[r] & B(\cat{C}^{\Gamma})}\]
	the horizontal maps are weak equivalences. Hence their homotopy fibers are weak equivalent as well. By Lemma \ref{lem.catloop} and Propositions \ref{prop.compcyl} and \ref{prop.acteq}, the induced map on homotopy fibers over the object $(b,1)$ is homology equivalent (resp.\ weakly equivalent) to $j \colon \Psi(M \rel a \sqcup b) \to \Gamma(M \rel j(a \sqcup b))$. Thus this map is a homology equivalence (resp.\ weak equivalence).
\end{proof}

\subsubsection{General manifolds} We now finish the proof of Theorem \ref{thm.technical} by completing step (iv).

\begin{proof}[Proof of Theorem \ref{thm.technical}] Without loss of generality $M$ path-connected. Any closed subset $A$ is an intersection of locally finite simplicial complexes. This may be seen using a handle decomposition of $M$ (if $M$ is a 4-dimensional topological manifold, this may not exist and one needs to use that for any point $p \in M \setminus A$ the manifold $M \setminus \{p\}$ is smoothable). Thus we can reduce to the case $A$ is a locally finite simplicial complex and hence assume that there is a locally finite set $S$ of points in $M$ such that $M \setminus (S \cup A)$ has no compact components.
	
	For $S$ finite the proof is by induction over $|S|$. In the initial case $|S|=1$, we write $S = \{s\}$ and our goal is to prove that the map $j \colon \Psi(M \rel a) \to \Gamma(M \rel j(a))$ is a homology equivalence (resp.\ weak equivalence). There is an embedding $\bR^n \hookrightarrow M \setminus A$ sending the origin to $s$. From this we obtain an embedding $\alpha \colon S^{n-1} \times \bR \hookrightarrow M$. By Proposition \ref{prop.rescommute} we have a commutative diagram
	\[\begin{tikzcd} {||\Psi(M \rel a,\alpha)_\bullet||} \rar{||j_\bullet||} \dar[swap]{\simeq} & {||\Gamma(M \rel j(a),\alpha,j)_\bullet||} \dar{\simeq} \\
	\Psi(M \rel a) \rar[swap]{j} & \Gamma(M \rel j(a)) \end{tikzcd}\]
	%\[\xymatrix{||\Psi(M \rel a,\alpha)_\bullet|| \ar[r]^-{||j_\bullet||} \ar[d]_\simeq & ||\Gamma(M \rel j(a),\alpha,j)_\bullet|| \ar[d]^\simeq \\
	%\Psi(M \rel a) \ar[r]_-j & \Gamma(M \rel j(a))}\]
	and the vertical maps are weak equivalences by Propositions \ref{prop.resolve} and \ref{prop.resolvefolex}, where in the latter case one uses the comments preceding its statement.
	
	The map $j_k$ induces a bijection on the indexing sets of the spaces of $k$-simplices. We then use Proposition \ref{prop.thmcylindricalbdy} to see that $j_k$ is a homology (resp.\ weak equivalence) on each term. For weak equivalence we additionally use that by construction all non-empty terms have fillable boundary conditions. Hence $j_\bullet$ is a levelwise homology equivalence  (resp.\ weak equivalence) between semisimplicial spaces and thus so is its realization $j$, e.g.~by the geometric realization spectral sequence in Section 1.4 of \cite{rwebertsemi} (resp.\ Theorem 2.2 of \cite{rwebertsemi}).
	
	For the induction step, note that to deduce the case $|S \cup \{s\}| = |S|+1$ from the case $|S|$, one can use the same argument as above after replacing Proposition \ref{prop.thmcylindricalbdy} with the inductive hypothesis. Finally we need deduce the statement for locally finite $S$ from that for finite $S$: exhaust $M$ by compact submanifolds $M_i$, so that $|S \cap M_i|$ is finite, and apply the sheaf property.
\end{proof}

\section{Application I: Vassiliev's $h$-principle} \label{sec.vassiliev} In this section we discuss our first application, functions with moderate singularities.

\begin{convention}
	In this section, $\mr{CAT} = \mr{Diff}$ and thus all manifolds are smooth.
\end{convention}

\subsection{Sheaves of functions with restricted jets} \label{subsec.sheavesjets} Vassiliev's $h$-principle concerns functions from a smooth manifold to $\bR^n$ that do not have certain singularities, in the sense that their jets avoid certain subspaces of the jet space. We give precise definitions following Chapter 1 of \cite{emhpbook}, or Section II.2 of \cite{ggstable}. Let $Z$ be a smooth manifold and $C^\infty(\bR^n,Z)$ have the weak $C^\infty$-topology, i.e.\ a sequence converges if all derivatives converge on all compacts. 

Given a smooth map $f \colon \bR^n \to Z$, let $D^I_0(f)$ denote the mixed derivatives at $0 \in \bR^n$ with respect to $I = (i_1,\ldots,i_{|I|})$ with $1 \leq i_k \leq n$:
	\[D^I_0(f) \coloneqq \frac{\partial}{\partial x^{i_1}} \cdots \frac{\partial}{\partial x^{i_{|I|}}} f(0).\]
Note $D^\varnothing_0 f$ is the value of the function at the origin.

\begin{definition}\label{def.jets} Let $Z$ be a smooth manifold of dimension $z$. The \emph{$r$th jet space} $J^{(r)}(\bR^n,Z)$ of smooth maps from $\bR^n \to Z$ is given by the quotient space 
	\[J^{(r)}(\bR^n,Z) \coloneqq C^\infty(\bR^n,Z)/{\sim_r}\]
	along the equivalence relation $\sim_r$ on where $f \sim_r g$ if $D^I_0(f) = D^I_0(g)$ for all $I$ satisfying $0 \leq |I|\leq r$.
	
	We denote the quotient map $ C^\infty(\bR^n,Z) \to J^{(r)}(\bR^n,Z)$ by $j^r$ and call it the \emph{$r$-jet map}.
\end{definition}

\begin{example}\label{exam.jetpoly} Fixing a basis in $\bR^n$, $J^{(r)}(\bR^n,\bR^z)$ can be identified with all $z$-tuples of polynomials in $n$ variables of degree $\leq r$. Under this identification, the $r$-jet map $j^r \colon C^\infty(\bR^n,\bR^z) \to J^{(r)}(\bR^n,\bR^z)$	sends $f$ to the following $z$-tuple of polynomials: the coefficient of $x^I$ in the $j$th polynomial (where $I = (i_1,\ldots,i_{|I|})$ with $1 \leq i_k \leq n$, $0 \leq |I|\leq r$, and $1 \leq j \leq z$) is given by $D^I_0(f_j)$.
\end{example}

The topological group $\mr{Diff}_0(\bR^n)$ of diffeomorphisms of $\bR^n$ fixing the origin acts on the right on $J^{(r)}(\bR^n,Z)$ by composition. If $\bar{f} \in J^{(r)}(\bR^n,Z)$ is represented by $f$, we have
\begin{align*}J^{(r)}(\bR^n,Z) \times \mr{Diff}_0(\bR^n) &\longto J^{(r)}(\bR^n,Z) \\
(\bar{f},\psi) &\longmapsto \overline{f \circ \psi}.\end{align*}
This action factors over the quotient group 
\[G^{(r)} \coloneqq \mr{Diff}_0(\bR^n)/{\sim_r}\]
along the equivalence relation $\sim_r$ of Definition \ref{def.jets} restricted to diffeomorphisms.

We can replace $\bR^n$ by $M$. To do this, note that Definition \ref{def.jets} can be generalized to smooth maps $M \to Z$ by replacing $(\bR^n,0)$ with $(M,m)$. Varying $m \in M$, we get a space $J^{(r)}(M,Z)$ with map to $M$. This is a fiber bundle with fiber over $m \in M$ given by $J^{(r)}(T_m M,Z)$, which is called the \emph{$r$th jet bundle of maps from $M$ to $Z$}. 

A more explicit description of this bundle is as follows. There is a principal $G^{(r)}$-bundle $G^{(r)}(\bR^n,TM)$ with total space given by pairs $(m,\bar{g})$ of a point in $M$ and an $r$-jet of a diffeomorphism $g \colon \bR^n \to TM$ which preserves the origin. Then $G^{(r)}$ acts on the right by composition and we have
\[J^{(r)}(M,Z) \cong G^{(r)}(\bR^n,TM) \times_{G^{(r)}} J^{(r)}(\bR^n,Z),\]
where to make $G^{(r)}$ act on the left on $J^{(r)}(\bR^n,Z)$ we act by the inverse. A subset $\cD$ of $J^{(r)}(\bR^n,Z)$ is \emph{$\mr{Diff}$-invariant} if it is preserved by $G^{(r)}$. If so, then its complement $J^{(r)}(\bR^n,Z) \setminus \cD$ is also preserved by $G^{(r)}$ and we can define $\cF^f(M,\cD)$ and $\cF(M,\cD)$ as follows:

\begin{definition}\label{def.vascfdformal} Suppose that $\cD \subset J^{(r)}(\bR^n,Z)$ is $\mr{Diff}$-invariant.
	\begin{itemize}
		\item We define $\cF^f(M,\cD)$ to be the space of sections
		\[\cF^f(M,\cD) \coloneqq \Gamma(M,G^{(r)}_0(\bR^n,TM) \times_{G^{(r)}} (J^{(r)}(\bR^n,Z) \setminus \cD)).\]
		\item We define $\cF(M,\cD)$ to be the space of smooth functions $M \to Z$ whose $r$-jets do \emph{not} lie in $\cD$:
		\[\cF(M,\cD) \coloneqq (j^r)^{-1}(\cF^f(M,\cD)) \subset C^\infty(M,Z).\]
	\end{itemize} 
\end{definition}

The space $\cF(M,\cD)$ can be identified with the subspace of $\cF^f(M,\cD)$ consisting of those sections that are \emph{holonomic}, i.e.\ their $0$th jet determines the higher jets by taking derivatives. The map $j\colon \cF(M,\cD) \to \cF^f(M,\cD)$ is then identified with the inclusion of this subspace. 

We leave it to the reader to see that for any $\mr{Diff}$-invariant $\cD$ the assignments $M \mapsto \cF(M,\cD)$ and $M \mapsto \cF^f(M,\cD)$ are $\mr{Diff}$-invariant topological sheaves.

\subsection{Vassiliev's $h$-principle} In Theorem 9 of \cite{vassiliev}, Vassiliev proved a homological $h$-principle for sheaves of the form $\cF(-,\cD)$ under certain conditions on $\cD$. His proof uses Alexander duality and interpolation theory for analytic functions to reduce the result to a finite-dimensional statement. The statement of the conditions on $\cD$ uses the notion of a real semi-algebraic subset of a Euclidean space; this is by definition a finite union of subsets cut out by finitely many polynomial equalities and inequalities.

\begin{theorem}[Vassiliev] \label{thm.vassiliev} If $Z = \bR^z$ and $\cD$ is a closed real semi-algebraic set of codimension at least $n+2$, then $\cF(-,\cD)$ satisfies a homological $h$-principle on closed manifolds. That is, the map
	\[j\colon \cF(M,\cD \rel b) \longto \cF^f(M,\cD \rel j(b))\]
	is a homology equivalence for all compact smooth manifolds $M$ and boundary conditions $b \in \cB_\mr{\cF}(\partial M)$.\end{theorem}

If the codimension is at least $n+3$, $j$ is in fact weak equivalence, because Vassiliev proved it is a homology equivalence between 1-connected spaces. Vassiliev asked whether $j$ is also a weak equivalence if the codimension is $n+2$. We will see below that under mild conditions codimension $n+2$ indeed suffices. We will also see that $j$ is still a homology equivalence even if one relaxes the conditions that $Z = \bR^z$ and that $\cD$ is real semi-algebraic. This will follow by checking the conditions for Theorem \ref{thm.main}. 

The following criterion for microflexibility is well-known, as any open differential relation gives rise to a microflexible sheaf by Example 1.4.1.B of \cite{gromovhp}.

\begin{lemma}\label{lem.vasconda} If $\cD$ is closed, then $\cF(-,\cD)$ is microflexible.\end{lemma}

\begin{proof}We need to check that for all pairs of $Q \subset R$ of compact submanifolds with corners in $M$, the restriction map
	\[\cF(R \subset M,\cD) \longto \cF(Q \subset M,\cD)\]
	is a microfibration. We first suppose that $Z = \bR^p$. Consider a commutative diagram
	\[\begin{tikzcd} \Delta^i \dar \rar{f_0} & \cF(R \subset M,\cD)\dar \\
	\Delta^i \times [0,1] \rar[swap]{F_s} & \cF(Q \subset M,\cD).\end{tikzcd}\]
	%\[\xymatrix{\Delta^i \ar[d] \ar[r]^-{f_0} & \cF(R \subset M,\cD)\ar[d] \\
	%\Delta^i \times [0,1] \ar[r]_-{F_s} & \cF(Q \subset M,\cD)}\]
	There exist open neighborhoods $V$ of $R$ in $M$ and $U$ of $Q$ in $M$, and representatives $f_0 \colon \Delta^i \to \cF(V,\cD)$ and $F_s \colon \Delta^i \times [0,1] \to \cF(U,\cD)$. We may assume $U \subset V$. Now pick a smooth function $\eta \colon M \to [0,1]$ with support in $U$ and equal to $1$ near $Q$. Then we can define a family of smooth functions on $V$ by
	\[f_s(d) \coloneqq \left(m \mapsto \eta(m) \tilde{F}_s(d)(m) + (1-\eta(m))f_0(d)(m) \right).\]
	
	Near $Q \subset M$ this coincides with $F_s$ and for $s=0$ this is equal to $f_0$. Since the complement of $\cD$ is open, there exists some $\epsilon>0$ such that for all $s<\epsilon$ the family $f_s$ has $r$-jet avoiding $\cD$. Restricting to $\Delta^i \times [0,\epsilon]$ gives the desired partial lift.
	
	For general $Z$, we remark that in the previous argument addition can be replaced by any smooth function $\alpha \colon [0,1] \times W \to Z$ where $W$ is a neighborhood of the diagonal in $Z \times Z$, such that $\alpha_0 = \pi_1$, $\alpha_1 = \pi_2$ and on the diagonal $\Delta$ it is the projection onto $Z$. Such functions exist: upon picking a Riemannian metric, there is a neighborhood $W$ such that if $(z,z') \in W$ there is a unique geodesic $\gamma \colon [0,1] \to Z$ from $z$ to $z'$, which will depend smoothly on the endpoints. Then define $\alpha(t,z,z')$ to be $\gamma(t)$.
\end{proof}

For condition (H) we will use Thom's jet transversality theorem, Theorem 2.3.2 of \cite{emhpbook}, Section II.4 of \cite{ggstable}, or page 38 of \cite{arnoldsing}. 

\begin{definition}Let $Y$ be a smooth manifold. \begin{itemize}
		\item A subset $S$ of $Y$ is said to be \emph{stratified} if it can be written as a finite union $\cup_{i=1}^n S_i$ of locally closed smooth submanifolds $S_i$, called \emph{strata}, such that $\bar{S}_k = \cup_{k=i}^n S_i$.
		\item A map $f \colon X \to Y$ is said to be \emph{transverse} to a stratified subset $S$ of $Y$ if it is transverse to each stratum.\end{itemize}
\end{definition}  

\begin{theorem}[Thom] \label{thm.thom} Let $\cS \subset J^{(r)}(\bR^n,Z)$ be a $\mr{Diff}$-invariant stratified subset, then the set of $g \in \mr{C}^\infty(M,Z)$ with jets $j^r(g) \in \Gamma(M,J^{(r)}(M,Z))$ transverse to $\cS$ is open and dense.\end{theorem}

\begin{lemma}\label{lem.vascondc} If $\cD$ is a closed stratified subset with strata of codimension at least $n+2$, then $\cF(-,\cD)$ satisfies condition (H).\end{lemma}

\begin{proof} This will follow from Thom's jet transversality Theorem \ref{thm.thom}, which we will use to prove the following stronger statement. Let $\mr{Map}(-,Z)$ denote the $\mr{Diff}$-invariant sheaf of continuous functions to $Z$ and $\cat{C}^{\mr{Map}(-,Z)}_\cF$ be the full subcategory of $\cat{C}^{\mr{Map}(-,Z)}$ on objects $\la t,b\ra$ with $b$ a boundary condition of the sheaf $\cF(-,\cD)$, i.e.\ $b \in \cB_{\cF}(S^{n-1} \times \{0\})$. The inclusion $h \colon \cF(-,\cD) \to \mr{Map}(-,Z)$ induces a functor
	\[h \colon [\cat{C}^{\cF(-,\cD)}] \longto [\cat{C}^{\mr{Map}(-,Z)}_\cF],\]
	and we will show this is an isomorphism of categories. The lemma then follows by noting that $[\cat{C}^{\mr{Map}(-,Z)}_\cF]$ is a full subcategory of a groupoid, hence a groupoid.
	
	First we prove that $\pi_0(\cF(S^{n-1} \times [0,1],\cD \rel b_0 \sqcup b_1))$ contains an element in any given homotopy class, i.e.\ $h$ is a $\pi_0$-surjection. Given any continuous function $f \colon S^{n-1} \times [0,1] \to Z$ satisfying $b_0$, $b_1$ near $S^{n-1} \times \{0,1\}$, we can smooth it rel boundary and apply Theorem \ref{thm.thom} with $g=f$ and $\cS = \cD$. This implies that we can perturb $f$ relative to the boundary to obtain a smooth function with $r$-jet transverse to $\cD$. As we working over the manifold $S^{n-1} \times [0,1]$ of dimension $n$, which is strictly smaller than the codimension $n+2$ of $\cD$, transverse intersection means empty intersection.
	
	A similar argument says that any two homotopic functions can be connected by a path of smooth functions with $r$-jets avoiding $\cD$, i.e.\ $h$ is a $\pi_0$-injection. Take two functions $f_0, f_1 \in \cF(S^{n-1} \times [0,1],\cD \rel b_0 \sqcup b_1)$ in the same homotopy class rel boundary. They can be connected by a path $f_t$ of smooth functions satisfying $b_0$ and $b_1$. This path can be considered as a smooth function 
	\[F\colon S^{n-1} \times [0,1] \times [0,1] \longto Z.\]
	Now we apply Theorem \ref{thm.thom} with $g = F$ and $\cS = \pi^{-1}(\cD)$ with $\pi\colon J^{(r)}(\bR^{n+1},Z) \to J^{(r)}(\bR^n,Z)$ induced by composing with the inclusion $\bR^n \to \bR^{n+1}$ by taking the last coordinate $0$ (there is no other choice, as the origin has to go to the origin). Suppose that $Z = \bR^z$, which is the typical case by covering $Z$ with charts.  Under the identification with polynomials of Example \ref{exam.jetpoly}, the map $\pi\colon J^{(r)}(\bR^{n+1},\bR^z) \to J^{(r)}(\bR^n,\bR^z)$ is given by setting the variable $x_{n+1}$ equal to $0$. For this description it follows that $F$ has $r$-jet avoiding $\pi^{-1}(\cD)$ if and only if each $F(-,t)$ has $r$-jet avoiding $\cD$. Furthermore, it implies that $\pi$ is a submersion and thus $\cS$ has the same codimension as $\cD$. %As above, we may assume it is $\bR^z$. Then $\pi^{-1}(\cD)$ consists of all $z$-tuples of polynomials of degree $\leq r$ in $x_1,\ldots,x_{n+1}$, such that setting $x_{n+1}=0$ results in a $z$-tuple of polynomials of degree $\leq r$ in $x_1,\ldots,x_{n}$ which lie in $\cD$. But setting $x_{n+1}=0$ is simply a projection in terms of the coefficients of monomials. 
	
	Thus we can perturb $F$ to relative to $b_0$, $b_1$, $f_0$ and $f_1$, to have $r$-jet transverse to $\cD$. Because the codimension of $\cD$ was at least $n+2$ and we are working over the manifold $S^{n-1} \times [0,1] \times [0,1]$ of dimension $n+1 < n+2$, the intersection is still empty. We can thus connect $f_0$ and $f_1$ by a path in $\cF(S^{n-1} \times [0,1],\cD \rel b_0 \sqcup b_1)$.\end{proof}

Condition (W) does not follow from transversality and we will need additional assumptions on $\cD$. We will use that post-composition gives an action of $\mr{Diff}(Z)$ on $J^{(r)}(\bR^n,Z)$. If $Z = \bR^z$, then for all compactly supported diffeomorphisms $\rho \colon [0,\infty) \to [0,\infty)$ that are scaling near the origin, we have a diffeomorphism of $\bR^z$ given by 
\[x \longmapsto \begin{cases}\frac{\rho(||x||)}{||x||} \cdot x & \text{if $||x||>0$,} \\
0 & \text{otherwise.}\end{cases}\]
We say $\cD$ is \emph{radially invariant} if it invariant under the subgroup of $\mr{Diff}(\bR^z)$ of these diffeomorphisms. The \emph{standard linear map} $\Lambda \colon \bR^n \to \bR^z$ is the map $(x_1,\ldots,x_n) \mapsto (x_1,\ldots,x_z)$ if $z \leq n$ and $(x_1,\ldots,x_n) \mapsto (x_1,\ldots,x_n,0,\ldots,0)$ if $z>n$.

\begin{lemma}\label{lem.vaslinear} Suppose that $\cD$ is closed stratified subset with strata of codimension at least $n+2$. Suppose that $Z$ is path-connected and there exists a chart $e \colon \bR^z \hookrightarrow Z$ such that \begin{enumerate}[(i)] 
		\item $\cD|_{\bR^z}$ is radially invariant and invariant under a transitive subgroup of $\mr{Diff}(\bR^z)$ (e.g the translations), or
		\item $\cD|_{\bR^z}$ is radially invariant and $z \leq n+1$.
	\end{enumerate} Then for all fillable boundary conditions $b$, (i') the map $e \circ \Lambda$ has an $r$-jet which avoids $\cD$ and (ii') $b$ can be connected to $(e \circ \Lambda)|_{S^{n-1}}$.
\end{lemma}

In the proof of this lemma, the radial invariance plays no role. It is included for use in the next lemma.

\begin{proof}We start assuming that $Z = \bR^z$ and $e = \mr{id}$. We claim  that for each $a \in \bR^z$ there exists an $f \in \cF(\bR^n,\cD)$ such that $f(0) = a$ and $Df(0)$ is of maximal rank.
	
	Under hypothesis (ii), this follows because the $r$-jets that take value $a$ and are of maximal rank are a $\mr{Diff}$-invariant subset of codimension $z$ (assumed to be $\leq n+1$) in $J^{(r)}(\bR^n,\bR^z)$ and thus the set of these $r$-jets must have non-empty intersection with $J^{(r)}(\bR^n,\bR^z)\setminus \cD$. Under hypothesis (i), since some subgroup of diffeomorphisms preserving $\cD|_{\bR^z}$ acts transitively on $\bR^z$, we only need to find such $f$ for one $a \in \bR^z$. This is always possible because the $r$-jets that are of maximal rank are codimension $0$.
	
	Let $f \in \cF(\bR^n,\cD)$ such that $f(0) = a$ and $Df(0)$ is of maximal rank. Then by the implicit function theorem there exists a local diffeomorphism $\phi\colon \bR^n \to \bR^n$ fixing the origin such that $f \circ \phi$ is given by $x \mapsto \Lambda(x)+a$. Since $\cD$ is $\mr{Diff}$-invariant, this implies that the germ at $a$ of $\Lambda$ is not in $\cD$, and since $a$ was arbitrary that $\Lambda$ lies in $\cF(\bR^n,\cD)$. This proves (i'). For (ii'), we note the proof of Lemma \ref{lem.vasconda} and the fact that $\bR^z$ is contractible imply that for any two $b_0$, $b_1$ the set $\pi_0(\cF(S^{n-1} \times [1,2],\cD \rel b_0 \sqcup b_1))$ consists of a single element.
	
	Next, we describe the argument when $Z \neq \bR^z$ or $e \neq \mr{id}$. By the proof of Lemma \ref{lem.vasconda} the path components of $[\cat{C}^{\cF(-,\cD)}]$ are in bijection with homotopy classes of maps $S^{n-1} \to Z$ under composition. Only the trivial homotopy class is fillable, so we may assume that $b$ has image in the chart $e$. The argument above implies that the standard linear map $\Lambda\colon \bR^n \to \bR^z$ in that chart has $r$-jet avoiding $\cD$, proving part (i'). For part (ii'), use that both $b$ and $(e \circ \Lambda)|_{S^{n-1}}$ are necessarily in the trivial homotopy class.
\end{proof}

\begin{lemma}\label{lem.vascondd} Suppose that one of the conditions of Lemma \ref{lem.vaslinear} is satisfied. Then $\cF(-,\cD)$ satisfies condition (W).\end{lemma}

\begin{proof}By Lemma \ref{lem.vaslinear}(i') there exists a chart $e \colon \bR^z \to Z$ such that $e \circ \Lambda$ has $r$-jet avoiding $\cD$ and by (ii') we may assume this is in the unique fillable path component. We may as well identify the image of $e$ with $\bR^z$ to simplify notation. The relevant element $\iota$ of $\cF(S^{n-1} \times [1,2], \cD \rel b_s \sqcup b_e)$ (note it determines $b_s$ and $b_e$) is the restriction of $\Lambda$ to $D^n_2 \setminus \mr{int}(D^n)$.  We need to prove that
	\begin{align*}\iota \circledcirc -\colon \cF(S^{n-1} \times [t_0,1],\cD \rel b_0 \sqcup b_s) &\longto \cF(S^{n-1} \times [t_0,2],\cD \rel b_0 \sqcup b_e), \\
	- \circledcirc \iota \colon \cF(S^{n-1} \times [2,t_1],\cD \rel b_e \sqcup b_1) &\longto \cF(S^{n-1} \times [1,t_1],\cD \rel b_s \sqcup b_1),\end{align*}
	are weak equivalences if $t_0<1$ resp.\ $2<t_1$. We give a proof in the first case, the second being similar. For the homotopy inverse we pick a family of compactly supported diffeomorphisms $\rho_t\colon [0,\infty) \to [0,\infty)$ for $t \in [0,1]$ satisfying:
	\begin{enumerate}[(i)]
		\item  $\rho_0 = \mr{id}$,
		\item $\rho_t$ is the identity on $[0,t_0]$, and
		\item $\rho_t$ is given by $s \mapsto s-t$ near $2$.
	\end{enumerate} 
	Let $\phi_t$ denote the family of radial diffeomorphisms given by $x \mapsto \frac{\rho_t(||x||)}{||x||} \cdot x$ if $||x||>0$ and $0$ otherwise, extended to all of $Z$ by the identity. The proposed homotopy inverse $r$ to $\iota \circledcirc -$ maps $f$ to the function $\phi_1 \circ f \circ \phi_1^{-1}$, in other words shrinking $f$ onto the smaller cylinder $S^{n-1} \times [t_0,1]$. A homotopy $h_t$ from the identity map on $\cF(S^{n-1} \times [t_0,1],\cD \rel b_0 \sqcup b_s)$ to  $r \circ (\iota \circledcirc -)$ is given by
	\[h_t(f)(x) = \begin{cases}
	(\phi_1 \circ \phi_t^{-1} \circ f \circ \phi_t \circ \phi_1^{-1})(x) & \text{if $\rho_t(\rho^{-1}_1(||x||)) \leq 1$,} \\
	\Lambda(x) & \text{otherwise,}
	\end{cases}\]
	and a homotopy $h'_t$ from the identity on $\cF(S^{n-1} \times [t_0,2],\cD \rel b_0 \sqcup b_e)$ to $(\iota \odot -) \circ r$ is given by
	\[h'_t(f)(x) = \begin{cases}
	(\phi_{1-t} \circ f \circ \phi_{1-t}^{-1})(x) & \text{if $||x|| \leq 2-t$,} \\
	\Lambda(x) & \text{otherwise.}
	\end{cases}\]
	The conditions on $\cD$ in Lemma \ref{lem.vaslinear} guarantee that these maps and homotopies have $r$-jets avoiding $\cD$.\end{proof}

Theorem \ref{thm.main} and the previous lemma's imply a generalization of Vassiliev's $h$-principle, here stated in a bit more generality than in Corollary \ref{cor.vassiliev}. Since a real semi-algebraic subset is a stratified subset \cite{loj}, this Corollary implies Theorem \ref{thm.vassiliev}.

\begin{corollary}\label{cor.vassiliev2} Let $Z$ be a smooth manifold and $\cD \subset J^{(r)}(\bR^n,Z)$ be a closed stratified subset with strata of codimension at least $n+2$. Then $\cF(-,\cD)$ satisfies a homological $h$-principle on closed manifolds: the map
	\[j\colon \cF(M,\cD \rel b) \longto \cF^f(M,\cD \rel j(b))\]
	is a homology equivalence for all $n$-dimensional manifolds $M$ and boundary conditions $b \in \cB_\cF(\partial M)$.
	
	Suppose additionally that $\cD$ is $\mr{Diff}(Z)$-invariant (or satisfies one of the two conditions in Lemma \ref{lem.vaslinear}), then this map is in fact a weak equivalence.\end{corollary}

\begin{proof}The first condition in Lemma \ref{lem.vaslinear} is satisfied when $\cD$ is $\mr{Diff}(Z)$-invariant, and by treating each path component of $Z$ separately, we may assume $Z$ is path-connected. We have already checked the conditions on $\cF(-,\cD)$ for Theorem \ref{thm.technical}: \begin{itemize}
		\item microflexibility: this was checked in Lemma \ref{lem.vasconda}.
		\item condition (H): this was checked in Lemma \ref{lem.vascondc}.
		\item condition (W): under the hypothesis on $\cD$, this was checked in Lemma \ref{lem.vascondd}.
	\end{itemize}
	
	Since $\cF^f(-,\cD)$ is a sheaf of sections, it is flexible. Hence it suffices to check that the map $\cF(\bR^n,\cD) \to J^{(r)}(\bR^n,Z) \setminus \cD$ which assigns a function its germ at the origin, is a weak equivalence. This is Lemma \ref{lem.vasrn}.\end{proof}

\begin{lemma}\label{lem.vasrn} If $\cD$ is a closed $\mr{Diff}$-invariant subset of $J^{(r)}(\bR^n,Z)$, then the map which assigns to a function its germ at the origin
	\[\cF(\bR^n,\cD) \longto J^{(r)}(\bR^n,Z) \setminus \cD\]
	is a weak equivalence.\end{lemma}

\begin{proof}The evaluation maps $\mr{ev}$ and $\mr{ev}^f$ fit into a commutative diagram
	\[\begin{tikzcd} \cF(\bR^n,\cD) \arrow{rr}{j} \arrow{rd}[swap]{\mr{ev}} & &  J^{(r)}(\bR^n,Z) \setminus \cD \arrow{ld}{\mr{ev}^f} \\
	& Z. & \end{tikzcd}\]
	%\[\xymatrix{\cF(\bR^n,\cD) \ar[rr]^-j \ar[rd]_-{\mr{ev}} & &  J^{(r)}(\bR^n,Z) \setminus \cD \ar[ld]^-{\mr{ev}^f} \\
	%& Z }\]
	These maps are Serre fibrations, using the isotopy extension theorem and $\mr{Diff}$-invariance of $\cD$. It thus suffices to prove that the map on fibers is a weak equivalence. That is, we must prove that for all $i \geq 0$ we can find a dotted lift in each commutative diagram
	\[\begin{tikzcd} \partial \Delta^i \rar{f} \dar & \mr{ev}^{-1}(a) \dar \\
	\Delta^i \rar[swap]{F} \arrow[dotted]{ru} & (\mr{ev}^f)^{-1}(a),\end{tikzcd}\] 
	%\[\xymatrix{\partial \Delta^i \ar[r]^-f \ar[d] & \mr{ev}^{-1}(a) \ar[d] \\
	%\Delta^i \ar[r]_-F \ar@{.>}[ru] & (\mr{ev}^f)^{-1}(a)}\] 
	making the top triangle commute and the bottom triangle commute up to homotopy rel $\partial \Delta^i$. 
	
	We pick a chart $\bR^z$ in $Z$ so that the origin goes to $a$ and show how to reduce to $\bR^z = Z$. By zooming in on the origin, we can simultaneously homotope $f$ and $F$ so that they have image in $\bR^z$. Applying this homotopy changes the diagram by a homotopy of diagrams, so it suffices to prove that for all $i \geq 0$ we can find a dotted lift in each commutative diagram
	\[\begin{tikzcd} \partial \Delta^i \rar{f} \dar & \cF_0(\bR^n,\cD') \dar \\
	\Delta^i \rar[swap]{F} \arrow[dotted]{ru} & J^{(r)}_0(\bR^n,\bR^z) \setminus \cD',\end{tikzcd}\] 
	%\[\xymatrix{\partial \Delta^i \ar[r]^-f \ar[d] & \cF_0(\bR^n,\cD') \ar[d] \\
	%\Delta^i \ar[r]_-F \ar@{.>}[ru] & J^{(r)}_0(\bR^n,\bR^z) \setminus \cD'}\] 
	making the top triangle commute and the bottom triangle commute up to homotopy rel $\partial \Delta^i$. Here $\cD'$ is given by $J^{(r)}_0(\bR^n,\bR^z) \cap \cD$ and $\cF_0(\bR^n,\cD')$ denotes the space of functions $\bR^n \to \bR^z$ preserving the origin and with $r$-jet avoiding $\cD$. Finding such a lift is possible if and only if the map $j^r\colon \cF_0(\bR^n,\cD') \longto J^{(r)}_0(\bR^n,\bR^z) \setminus \cD'$ is a weak equivalence.
	
	Let $\cP_{(r)}(\bR^n,\bR^z) \setminus \cD'$ be the space of $z$-tuples of polynomials $\bR^n \to \bR$ of degree $\leq r$, which at the origin take value $0$ and have $r$-jet avoiding $\cD'$; by Example \ref{exam.jetpoly} $\cP_{(r)}(\bR^n,\bR^z) \setminus \cD' \cong J^{(r)}_0(\bR^n,\bR^z) \setminus \cD'$. 
	
	To prove that $j^r\colon \cF_0(\bR^n,\cD') \to \cP_{(r)}(\bR^n,\bR^z) \setminus \cD'$ is a weak equivalence, we use Taylor's theorem. Taylor's theorem tells us that any smooth function $f\colon \bR^n \to \bR$ can be written as $f(x) = p_r(f)(x) + \sum_{|\alpha| = r+1} x^\alpha h_\alpha(x)$, where $p_r(f)$ is a polynomial of degree $r$ and $h_\alpha(x)$ is a function $\bR^n \to \bR$ such that $\lim_{x \to 0} h_\alpha(x) = 0$.  Similarly, if $f$ is a function $\bR^n \to \bR^z$ with $z > 1$, we can apply Taylor approximation to each component separately. We use the same notation. 
	
	If $g \in \cF_0(\bR^n,\cD')$, for all $s \in [0,1]$
	\[P_s(g)(x) \coloneqq sp_r(g)(x) + (1-s) \sum_{|\alpha| = r+1} x^\alpha h_\alpha(x)\]
	has $r$-jet at the origin avoiding $\cD'$. This follows from the fact that the $r$-jet of $g_s$ at the origin is independent of $s$ and the statement is true for $s=0$. Each function $P_s(g)$ does not necessarily lie in $\cF_0(\bR^n,\cD')$, but since $\cD$ is closed there exists some $\epsilon > 0$ such that $P_s(g)$ restricted to $D^n_\epsilon(0)$ does. 
	
	Now suppose we are given a commutative diagram
	\[\begin{tikzcd} \partial \Delta^i \rar{f} \dar & \cF_0(\bR^n,\cD') \dar[swap]{j^r} \\
	\Delta^i \rar[swap]{F} \arrow[dotted]{ru} & \cP_{(r)}(\bR^n,\bR^z) \setminus \cD', \end{tikzcd}\]
	%\[\xymatrix{\partial \Delta^i \ar[r]^-f \ar[d] & \cF_0(\bR^n,\cD') \ar[d]^{j^r} \\
	%\Delta^i \ar[r]_-F \ar@{.>}[ur] & \cP_{(r)}(\bR^n,\bR^z) \setminus \cD'}\]
	then it suffices to produce a dotted lift $L$ making the top triangle commute and the bottom triangle commute up to homotopy rel $\partial \Delta^i$. Since Taylor approximation is continuous in the function, there is a continuous function $\epsilon \colon \partial \Delta^i \to (0,1)$ such that $P_s(f(d))$ restricted to the disk $D^n_{\epsilon(d)}(0)$ of radius $\epsilon(d)$ around the origin lies in $\cF_0(\bR^n,\cD')$.
	
	Pick a family of embeddings $\eta \colon \partial \Delta^i \times [0,1] \to \mr{Emb}(\bR^n,\bR^n)$ such that (i) each $\eta(d,s)$ is the identity near the origin, (ii) $\eta$ restricted to $\Delta^i \times \{0\}$ is the identity, and (iii) $\eta(d,1)$ has image in $D^n_{\epsilon(d)}(0)$. We write $\Delta^i \cong \Delta^i \cup_{\partial \Delta^i \times \{2\}} \partial \Delta^i \times [0,2]$ and define $L$ as:
	\[L(d) \coloneqq \begin{cases} f(d') \circ \eta(d',s) & \text{if $d = (d',s) \in \partial \Delta^i \times [0,1]$,} \\
	P_{s-1}(f(d') \circ \eta(d',1)) & \text{if $d = (d',s) \in \partial \Delta^i \times [1,2]$,} \\
	F(d) \circ \eta(d,1) & \text{otherwise.}\end{cases}\]
	This is the desired lift.
\end{proof}

\begin{example}Let $\cD = \cD_{MW}$ be the subset of $J^2(\bR^n,\bR)$ consisting of the $2$-jets of $f$ such that $f(0) = 0$, $Df'(0) = 0$ and the Hessian $D^2 f(0)$ is degenerate. This is $\mr{Diff}$-invariant and has codimension $n+2$, and so Vassiliev's homological $h$-principle applies to $\cF(-,\cD_{MW})$. This was a crucial ingredient in the original proof of the Madsen-Weiss theorem, see Section 4 of \cite{madsenweiss}. It is not used in later proofs, e.g.~\cite{gmtw}.	Now remark that $\cD_{MW}$ is also radially invariant, so satisfies condition (ii) of Lemma \ref{lem.vaslinear}. Thus Corollary \ref{cor.vassiliev2} says $\cF(-,\cD_{MW})$ in fact satisfies a homotopical $h$-principle. \end{example}

\subsection{Maps to the line} As an example we discuss the smallest set of singularities sufficient to apply Corollary \ref{cor.vassiliev} to maps to $\bR$. That is, we explain which singularities of maps $\bR^n \to \bR$ one needs to include to get $\cD$ to be sufficient codimension. Given a smooth map $g \colon \bR^n \to \bR^m$, we say that $f$ is \emph{of the form $g$ near $p$} if there exist charts around $p$ and $f(p)$ such that $f$ is equal to $g$ near the origin in the coordinates from these charts.

\begin{definition}Let $M$ be a smooth manifold of dimension $n$ and $f\colon M \to \bR$ be a smooth map. \begin{itemize}
		\item A point $p \in M$ is a \emph{critical point} of $f$ if $df(p) = 0$.
		\item A critical point $p$ of $f$ is said to be a \emph{Morse singularity of index $k$} if near $p$ the function $f$ is of the form:
		\[f(x) = -\sum_{i=1}^k x_i^2+\sum_{j=k+1}^n x_j^2.\]
		\item A critical point $p$ of $f$ is said to be a \emph{birth-death singularity of index $k+\frac{1}{2}$} if near $p$ the function $f$ is of the form
		\[f(x) = x_1^3 -\sum_{i=2}^k x_i^2+\sum_{j=k+1}^n x_j^2.\]
	\end{itemize}
\end{definition}

See Figure \ref{fig.circlefamily} for examples.

\begin{definition}Let $\cG(M)$ be the space of smooth functions $f\colon M \to \bR$ that only have Morse or birth-death singularities. This is called the space of \emph{generalized Morse functions}.\end{definition}

The following $h$-principle was proven in a range by Igusa by singularity theory \cite{igusagmfh}, on homology by Vassiliev using Alexander duality and interpolation techniques \cite{vassiliev} and in general by Eliashberg and Mishachev using wrinkling \cite{emgmf}. 

\begin{corollary}[Eliashberg-Mishachev] \label{cor.gmf} The map
	\[j\colon \cG(M \rel b) \longto \cG^f(M \rel j(b))\]
	is a weak equivalence for all $M$ and boundary conditions $b \in \cB_{\cG}(\partial M)$.
\end{corollary}

\begin{proof}To apply Corollary \ref{cor.vassiliev}, one remarks that the set $\cD \subset J^{(3)}(\bR^n,\bR)$ of Morse or birth-death germs is $\mr{Diff}(\bR)$-invariant and hence it suffices to prove it is of sufficient codimension. This follows from the well-known results of Morse and Cerf that generic smooth functions have only Morse singularities and generic 1-parameter families of smooth functions have only Morse and birth-death singularities.\end{proof}
	
\begin{remark}\label{rem.boardmanformula}In general, one may use the Boardman codimension formula as in Section 2.4-2.6 of \cite{arnoldsing}. This says that singular subset $\Sigma$ in $M$ of a generic smooth map $f\colon M \to \bR^z$ can be stratified by strata $\Sigma^I$. These are defined recursively. The top stratum $\Sigma^i(f)$ is the subset where the rank drops by $i$. Similarly the other strata are inductively defined by setting $\Sigma^{i_1,\ldots,i_k,i}$ to be the subset of $\Sigma^{i_1,\ldots,i_k}$ where the derivative of $f\colon M \to \bR^z$ restricted to $\Sigma^{i_1,\ldots,i_k}$ drops by $i$. A necessary condition is thus that $m \geq i_1 \geq i_2 \geq \ldots \geq i_k \geq 0$. 
	
	Boardman's codimension formula says that the codimension of $\Sigma^I$ is given by
	\begin{align*}\nu_I(m,1) &\coloneqq (1-m+i_1)\mu(i_1,i_2,\ldots,i_k)-(i_1-i_2)\mu(i_2,i_3,\ldots,i_k)\\
	&\qquad -\ldots-(i_k-i_{k-1})\mu(i_k),\end{align*}
	where $\mu(i_1,\ldots,i_l)$ is the number of sequences of integers $j_1,\ldots,j_k$ such that $m \geq j_1 \geq \ldots \geq j_k \geq 0$, $i_r \geq j_r$ and $j_1 > 0$. If the value of this expression is negative, no such singularities occur, and if the value of this expression is $0$, it is a non-singular point. We are interested when $\nu_I(m,1) \leq m+1$, as those strata occur in 1-parameter families. Let's start with the case $k=1$. In that case $\nu_{i_1}(m,1) = (1-m+i_1)i_1$, so this is positive if $i_1 \geq m$ and then $\leq m+1$ only if $i_1 = m$. Since $\Sigma^I \subset \Sigma^{i_1}$, we conclude that for $I$ to occur it must start with $m$. In the case $\nu_{m,i_2}$ we get a value $\leq m+1$ only if $i_2=1$ and in that case in fact $\nu_{m,1}$ equals $2m-(m-1) = m+1$. This makes clear that no higher strata can occur, as the codimension of $\Sigma^I$ must increase. Finally one checks that $\Sigma^m$ and $\Sigma^{m,1}$ indeed correspond to Morse and birth-death singularities.\end{remark}

The homotopy type of $\cG(\bR^n)$ is known. The argument below was given in Section 3 of \cite{igusagmf}. We recall it for the convenience of the reader and will use it in Section \ref{sec.framed}.

\begin{lemma}[Igusa] \label{lem.gmfhomtype} $\cG(\bR^n)$ is weakly equivalent to the join $S^{n-1} * P$, where $P$ is the homotopy pushout of the following diagram:
	\[\begin{tikzcd} &[-40pt] \frac{O(n)}{O(0) \times O(n-1)} \arrow{rd} \arrow{ld} &[-40pt] &[-40pt] \frac{O(n)}{O(1) \times O(n-2)}\arrow{rd} \arrow{ld} &[-40pt] &[-40pt] \frac{O(n)}{ O(n-1) \times O(0)} \arrow{rd} \arrow{ld} &[-40pt] \\
	\frac{O(n)}{O(0) \times O(n)} & & \frac{O(n)}{O(1) \times O(n-1)}& &\qquad \ldots\qquad  & & \frac{O(n)}{O(n) \times O(0)}. \end{tikzcd}\]
	%\[\resizebox{\textwidth}{!}{\xymatrix@C=1em{ & \frac{O(n)}{O(0) \times O(n-1)} \ar[rd] \ar[ld] & & \frac{O(n)}{O(1) \times O(n-2)}\ar[rd] \ar[ld] & & \frac{O(n)}{ O(n-1) \times O(0)} \ar[rd] \ar[ld] & \\
	%\frac{O(n)}{O(0) \times O(n)} & & \frac{O(n)}{O(1) \times O(n-1)}& &\qquad \ldots\qquad  & & \frac{O(n)}{O(n) \times O(0)}}}\]
\end{lemma}

\begin{proof}[Sketch of proof] By the proof of Corollary \ref{cor.vassiliev}, $\cG(\bR^n)$ is weakly equivalent to the space of $3$-jets at the origin of functions $\bR^n \to \bR$ with only Morse or birth-death singularities at the origin. We can identify this with the space $\cP^\cG_3(\bR^n)$ of polynomials of degree $\leq 3$ in $n$ variables satisfying the same condition.
	
	We will next prove that $\cP^\cG_3(\bR^n)$ is homotopy equivalent to $S^{n-1} * \cE^\cG_3(\bR^n)$, where $\cE^\cG_3(\bR^n)$ is the subspace of $\cP^\cG_3(\bR^n)$ of polynomials with value and first derivative at the origin given by $0$. This follows since $\cE^\cG_3(\bR^n) \cong \bR^{\dim \cP_3 - n - 1} \setminus C$ for a closed subset $C$, while $\cP^\cG_3(\bR^n) = \bR \times (\bR^{\dim \cP^\cG_3-1} \setminus C)$ (see Lemma 3.1 of \cite{igusagmf} for more details).
	
	We end by proving that $\cE^\cG_3(\bR^n)$ is homotopy equivalent to the homotopy pushout $P$. Let $\bar{A}_k$ be the subspace of polynomials with Morse singularities of index $k$ or birth-death singularities of index $k\pm\frac{1}{2}$ at the origin. Let $B_{k+1/2}$ be the subspace of polynomials with a birth-death singularity of index $k+\frac{1}{2}$ at the origin. Then $\cE_3(\bR^n)$ is a push out:
	\[\begin{tikzcd} & B_{1/2} \arrow{rd}  \arrow{ld} & & B_{1+1/2}  \arrow{rd} \arrow{ld} & & B_{n-1/2} \arrow{rd} \arrow{ld} & \\
	\bar{A}_0 & & \bar{A}_1 & & \ldots & & \bar{A}_n.\end{tikzcd}\]
	The maps are cofibrations, so this is also a homotopy push out. Hence it remains to identify $B_k$ and $\bar{A}_k$, and the maps between these.
	
	Let $D = D^2_0 f$ be the Hessian of $f$ at the origin. Sending $f \in B_{k+1/2}$ to the splitting $(D_+,D_-,\ker(D))$ of $T_0 \bR^n$ in positive eigenspace, negative eigenspace and kernel, and recording the framing of $\ker(D)$ gives a map $B_{k+1/2} \to O(n)/(O(k) \times O(n-1-k))$. This is a fibration with convex fibers, so a weak equivalence. Similarly, there is a map sending $f \in \bar{A}_k$ to $O(n)/(O(k) \times O(n-k))$. If $f \in A_k$, it sends it to $(D_+,D_-)$ and if $f \in B_{k\pm 1/2}$ we use the previous maps composed with the inclusion into $O(n)/(O(k) \times O(n-k))$. This is again a fibration with convex fibers. It also identifies the maps $B_{k \pm 1/2} \to \bar{A}_k$ with the standard inclusions of Grassmannians.
\end{proof}

\begin{example}\label{exfamilys2} Our $h$-principle implies that $\cG(S^1) \simeq LS^2$. Since the tangent bundle of $S^1$ is trivial, to prove this we only need to identify $\cG(\bR)$. By Lemma \ref{lem.gmfhomtype} it is $S^0 * P$, and $P \simeq S^1$, being the homotopy push out of
	\[\begin{tikzcd} & \frac{O(1)}{O(0) \times O(0)} \simeq \ast \sqcup \ast \arrow{rd} \arrow{ld} & \\
	\frac{O(1)}{O(0) \times O(1)} \simeq \ast & & \frac{O(1)}{O(1) \times O(0)} \simeq \ast,\end{tikzcd}\]
	%\[\xymatrix{ & \frac{O(1)}{O(0) \times O(0)} \simeq * \sqcup * \ar[rd] \ar[ld] & \\
	%\frac{O(1)}{O(0) \times O(1)} \simeq * & & \frac{O(1)}{O(1) \times O(0)} \simeq *}\]
	where $* \sqcup *$ corresponds to birth-death singularities with opposite positive direction of $D^3 f$ on the kernel $K_p$ of $D^2 f$. The two $*$'s correspond to a local minimum and maximum respectively. So the answer is $\cG(\bR) \simeq S^2$. A generator of $\pi_2(\cG(\bR))$ is given as follows. Write $S^2 = S^0 * S^1$ with coordinates $(\epsilon,t,\theta)$ with $\epsilon = \pm 1$, $t \in [0,1]$ and $\theta \in S^1$ under the equivalence relation $(1,0,\theta) \sim (-1,0,\theta)$. Then the following is a generator of $\pi_2$, see Figure \ref{fig.circlefamily}:
	\[v(\epsilon,t,\theta) \coloneqq \epsilon t x + \sin(\theta) x^2 + \cos(\theta)x^3.\]
\end{example}

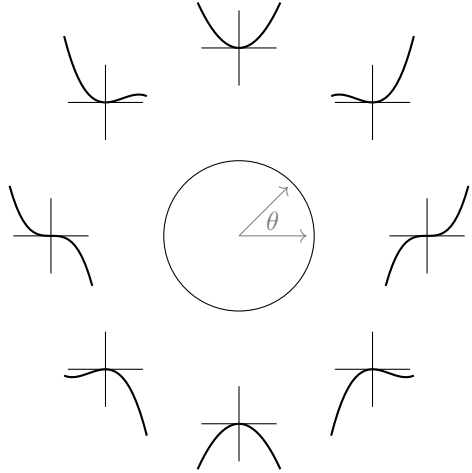
\begin{figure}
	\centering
	\begin{tikzpicture}
	\draw (0,0) circle (1cm);
	\draw [->,black!50!white] (0,0) -- (0.65,0.65);
	\draw [->,black!50!white] (0,0) -- (0.9,0);
	\node at (0.45,0.2) [black!50!white] {$\theta$};
	\begin{scope}[shift={(2.5,0)},scale=0.5]
	\draw (-1,0) -- (1,0);
	\draw (0,-1) -- (0,1);
	\draw[thick,domain=-1.1:1.1] plot (\x,{\x*\x*\x});
	\end{scope}
	\begin{scope}[shift={(-2.5,0)},scale=0.5]
	\draw (-1,0) -- (1,0);
	\draw (0,-1) -- (0,1);
	\draw[thick,domain=-1.1:1.1] plot (\x,{-\x*\x*\x});
	\end{scope}
	\begin{scope}[shift={(0,2.5)},scale=0.5]
	\draw (-1,0) -- (1,0);
	\draw (0,-1) -- (0,1);
	\draw[thick,domain=-1.1:1.1] plot (\x,{\x*\x});
	\end{scope}
	\begin{scope}[shift={(0,-2.5)},scale=0.5]
	\draw (-1,0) -- (1,0);
	\draw (0,-1) -- (0,1);
	\draw[thick,domain=-1.1:1.1] plot (\x,{-\x*\x});
	\end{scope}
	\begin{scope}[shift={({0.71*2.5},{0.71*2.5})},scale=0.5]
	\draw (-1,0) -- (1,0);
	\draw (0,-1) -- (0,1);
	\draw[thick,domain=-1.1:1.1] plot (\x,{.8*\x*\x+.6*\x*\x*\x});
	\end{scope}
	\begin{scope}[shift={({0.71*2.5},{-0.71*2.5})},scale=0.5]
	\draw (-1,0) -- (1,0);
	\draw (0,-1) -- (0,1);
	\draw[thick,domain=-1.1:1.1] plot (\x,{-.8*\x*\x+.6*\x*\x*\x});
	\end{scope}
	\begin{scope}[shift={({-0.71*2.5},{0.71*2.5})},scale=0.5]
	\draw (-1,0) -- (1,0);
	\draw (0,-1) -- (0,1);
	\draw[thick,domain=-1.1:1.1] plot (\x,{.8*\x*\x-.6*\x*\x*\x});
	\end{scope}
	\begin{scope}[shift={({-0.71*2.5},{-0.71*2.5})},scale=0.5]
	\draw (-1,0) -- (1,0);
	\draw (0,-1) -- (0,1);
	\draw[thick,domain=-1.1:1.1] plot (\x,{-.8*\x*\x-.6*\x*\x*\x});
	\end{scope}
	\end{tikzpicture}

	\caption{A family of generalized Morse functions indexed by the circle. At the origin, the functions at $\theta \in (0,\pi)$ have a Morse singularity of index $0$, those at $\theta \in (\pi,2\pi)$ have a Morse singularity of index $1$, and those at $\theta = 0$ and $\theta = \pi$ have a birth-death singularity of index $\frac{1}{2}$.}
	\label{fig.circlefamily}
\end{figure}

\subsection{Maps to the plane}
One application of the study of generalized Morse functions is pseudoisotopy theory \cite{igusastab}. To study higher-dimensional versions of pseudoisotopy theory, one needs $h$-principles for maps to $\bR^k$. To find out which singularities one needs to allow, one uses the codimension formula's from \cite{arnoldsing} as in Remark \ref{rem.boardmanformula}.

The case of maps to the plane was considered in \cite{weissreis, weissreish}, and there Reis and Weiss checked which singularities to include to get sufficient codimension. For precise definitions we will refer the reader to their work, but the conclusion is that a generic smooth map $M \to \bR^2$ only has \emph{fold} singularities, and a generic $1$-parameter family of smooth maps only has \emph{fold, cusp, lips, beak-to-beak} and \emph{swallowtail} singularities. These singularities are invariant under diffeomorphisms of $M$ and $\bR^2$. Let $\cG_\mr{2}(M)$ denote the functions $M \to \bR^2$ that only have such singularities. The jet map $j\colon \cG_\mr{2}(M \rel b) \to \smash{\cG^{f}_\mr{2}}(M \rel j(b))$ was shown to be a homology equivalence in \cite{weissreish}, but we prove it is actually a weak equivalence:

\begin{corollary}The map
	\[j\colon \cG_\mr{2}(M \rel b) \longto \cG^{f}_\mr{2}(M \rel j(b))\]
	is a weak equivalence for all $n$-dimensional manifolds $M$ and boundary conditions $b \in \cB_{\cG_{2}}(\partial M)$.\end{corollary}

\section{Application II: framed functions}\label{sec.framed} Our next application is the contractibility of the space of framed functions.

\begin{convention}In this section $\mr{CAT} = \mr{Diff}$, so all manifolds are smooth.\end{convention}

\subsection{Motivation and definition} While Corollary \ref{cor.gmf} is useful, it is not optimal for most applications because the homotopy type of $\cG(\bR^n)$ is non-trivial. Thus it may not be easy to prove that families exist or extend, even with an $h$-principle. Framed functions are designed to fix this, adding data to generalized Morse functions to get spaces $\cG_\mr{fr}(M)$ which not only satisfy an $h$-principle but also have the property that $\cG_\mr{fr}(\bR^n)$ is contractible. 

The problem exhibited by the generator of $\pi_2(\cG(\bR))$ in Example \ref{exfamilys2}, is that while the birth-death singularities have a preferred direction at the origin the local minimum and maximum do not. Indeed, recall that $\cG(\bR) \simeq S^0 * P$, with the latter a join of $S^0$ -- coming from the two possible signs of non-zero derivative at the origin -- with the homotopy push out $P$:
\[\begin{tikzcd} \ast \sqcup \ast \rar \dar & \ast \dar \\
\ast \rar & P \simeq S^1\end{tikzcd}\]
where $* \sqcup *$ corresponds birth-death singularities with opposite directions and the two $*$'s correspond to the local minimum and maximum. If we had decorations singling out a preferred direction at the local maximum, $P$ would be replaced by the contractible pushout $P'$ of $* \sqcup *$ mapping into $*$ and $* \sqcup *$:
\[\begin{tikzcd}\ast \sqcup \ast \rar \dar & \ast \dar \\
\ast \sqcup \ast \rar & P' \simeq \ast.\end{tikzcd}\]

In general the problem is that the space of Morse functions behaves more like a Grassmannian than a Stiefel manifold, which Igusa fixed by modifying the definitions to include a framing of the negative eigenspaces of the Hessian, defined as follows. Pick a Riemannian metric $g$ on $M$. If $f$ has a Morse singularity at $p$, then we can consider the Hessian $D^2 f$ as a linear map $T_p M \to T_p M$. This linear map is invertible if the Hessian is non-degenerate and has real eigenvalues since the Hessian was symmetric. By definition, the negative eigenspace is the subspace of $T_p M$ spanned by the eigenvectors corresponding to negative eigenvalues. 

Next, suppose $f$ has a birth-death singularity at $p$. Then $D^2 f$ as a linear map $T_p M \to T_p M$ has a one-dimensional kernel $K_p$. On the orthogonal complement of this kernel it is invertible and has real eigenvalues. On the kernel $K_p$, the third derivative gives a well-defined homogeneous map $\kappa \colon K_p \to \bR$ of degree 3.

\begin{definition}\label{def.framing} Suppose $M$ has a Riemannian metric $g$ and a generalized Morse function $f \colon M \to \bR$. \begin{itemize}
		\item A \emph{framing at a Morse singularity of index $k$} of $f$ is a choice of $k$ orthonormal basis vectors for the negative eigenspace.
		\item A \emph{framing at a birth-death singularity of index $k+\frac{1}{2}$} of $f$ is a choice of $k+1$ orthonormal vectors, such that the first $k$ are a basis for the negative eigenspace and the last vector is the unique unit vector in the kernel $K$ with $\kappa$ having positive value on it.
\end{itemize}\end{definition}

Recall that $\cG(M)$ denotes the space of generalized Morse functions on $M$. For an $f \in \cG(M)$, the subsets $\smash{\Sigma^m_k}(f)$ and $\smash{\Sigma^{m,1}_{k+1/2}}(f)$ of $M$ denote respectively the sets of index $k$ Morse singularities and index $k+\frac{1}{2}$ birth-death singularities of $f$.

\begin{definition}We can define the \emph{space $\cG_\mr{fr}(M)$ of framed of functions on $M$} as an object of $\cat{Sh^{CAT}}$. On a parametrizing manifold $P$, it is given by maps $P \to \cG(M) \times \mr{Riem}(M)$ with a framing as in Definition \ref{def.framing} for each critical point. These framings should be continuous in the sense that the basis vectors $\zeta_j$ of the framings are continuous when considered as sections of $TM$ over the singularity sets $\Sigma^m_i \subset P \times M$ for $0 \leq i \leq m$ and $\Sigma^{m,1}_{i+1/2} \subset P \times M$ for $0 \leq i \leq m-1$.\end{definition}

\subsection{The $h$-principle} The following was proven by Igusa in a range using singularity theory \cite{igusaframed}, by Lurie using obstruction-theoretic techniques \cite{lurietft}, by Eliashberg and Mishachev using wrinkling \cite{emframed}, and by Galatius in unpublished work using techniques similar to ours (in fact, his proof inspired this paper). Its main applications are the cobordism hypothesis \cite{lurietft} and the construction of higher Reidemeister torsion \cite{igusatorsion}.

\begin{corollary}[Lurie, Eliashberg-Mishachev, Galatius] \label{cor.framed} For all manifolds $M$ and boundary conditions $b \in \cB_\mr{\cG_\mr{fr}}(\partial M)$, the space of framed functions $\cG_\mr{fr}(M \rel b)$ is contractible.\end{corollary}

This follows directly from Theorem \ref{thm.main} and the following lemma's. The first shows that space of framed functions on $\bR^n$ is contractible --- Theorem 2.4 of \cite{igusaframed} --- and the second says our $h$-principle applies.

\begin{lemma}[Igusa] $\cG_\mr{fr}(\bR^n)$ is contractible.\end{lemma}

\begin{proof}The space of Riemannian metrics is convex, hence contractible, and thus we disregard it. Carrying through the argument of Lemma \ref{lem.gmfhomtype}, $\cG_\mr{fr}(\bR^n)$ is weakly equivalent to the join $S^{n-1} * P'$, where $P'$ is the homotopy pushout of the following diagram:
	
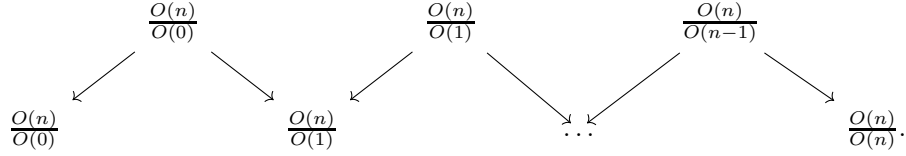
\[\begin{tikzcd}  & \frac{O(n)}{O(0)} \arrow{rd} \arrow{ld} & & \frac{O(n)}{O(1) }\arrow{rd} \arrow{ld} & & \frac{O(n)}{ O(n-1)} \arrow{rd} \arrow{ld} & \\
	\frac{O(n)}{O(0)} & & \frac{O(n)}{O(1)}& &\ldots  & & \frac{O(n)}{O(n)}. \end{tikzcd}\]
	This homotopy pushout can be computed inductively by iterated homotopy pushouts, starting at the right. Let $P'_i$ be the homotopy pushout of
	
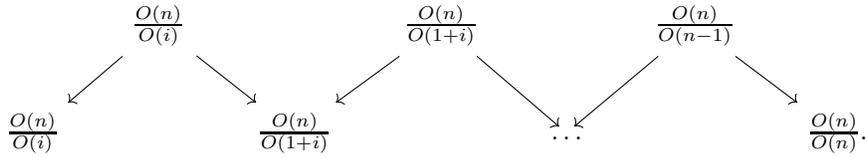
\[\begin{tikzcd} &[-5pt] \frac{O(n)}{O(i)} \arrow{rd}\arrow{ld} &[-5pt] &[-5pt] \frac{O(n)}{O(1+i)}  \arrow{rd} \arrow{ld} &[-5pt] &[-5pt] \frac{O(n)}{ O(n-1)}  \arrow{rd} \arrow{ld} &[-5pt] \\
	\frac{O(n)}{O(i)} & & \frac{O(n)}{O(1+i)}& &\ldots  & & \frac{O(n)}{O(n)}.\end{tikzcd}\]
	If $i = n$ we simply get $O(n)/O(n) \simeq *$, and $P'_i$ is obtained as the homotopy pushout of the diagram
	
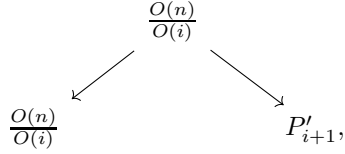
\[\begin{tikzcd} & \frac{O(n)}{O(i)} \arrow{rd} \arrow{ld} & \\
	\frac{O(n)}{O(i)} & & P'_{i+1},\end{tikzcd}\]
	and thus $P'_i \simeq P'_{i+1} \simeq *$.
\end{proof}

\begin{lemma}$\cG_\mr{fr}(-)$ is microflexible and satisfies condition (W).\end{lemma}

\begin{proof}Checking microflexibility is done as in Lemma \ref{lem.vasconda} and carrying along the framings. To check condition (H) holds, we could use Theorem 1.6 of \cite{igusaframed}, which says that the space $\cG_\mr{fr}(M,b)$ is $(n-1)$-connected. However, an elementary argument as in Lemma 1.5 of \cite{igusaframed} suffices.
	
	Indeed, for condition (H), Lemma \ref{lem.vascondc} says generalized Morse functions underlying a $1$-parameter framed functions can be connected by a one-parameter family of generalized Morse functions. We can assume that there is at most one birth-death singularity in $\mr{int}(M)$ for each $t \in [0,1]$ and will do an induction on the number of birth-death singularities.
	
	Starting at $t=0$, we can extend the framing until we reach the first birth-death singularity. If it is a birth singularity, i.e.\ two Morse singularities appear when first there were none, one can pick a framing arbitrarily at the birth-death singularity and extend using a local model. 
	
	If it is a death singularity, i.e.\ two Morse singularities disappear, one has to be careful, because the framings might not match up. In that case we need to line up an orthonormal $k$-tuple and the orthonormal $(k+1)$-tuple so that (i) the $k$-tuple coincides with the first $k$ vectors of $(k+1)$-tuple and (ii) the first basis vector of the latter goes to the positive direction of the kernel of the Hessian in birth-death singularity. Pick any path of $(k+1)$-tuples to satisfy the second condition. We want find a path of $k$-tuples to the remaining $k$ vectors in the $(k+1)$-tuple. Since the space of $k$-tuples with fixed orientation is path-connected, we can do so if the two orthonormal $k$-tuples have the same orientation. Hence it suffices to show how reverse the direction of one of these $k$ vectors. Figure \ref{figorientation} shows how to do this in the 1-dimensional case by introducing two additional Morse singularities for a short time period. In the higher-dimensional case, one simply takes the product of this picture with the appropriate quadratic form.
	
	Finally, for the additional requirements of condition (W) one takes $\iota$ to be $\Lambda \colon \bR^n \to \bR$, the projection to the first coordinate, restricted to the annulus $D^n_2 \setminus \mr{int}(D^n_1)$.\end{proof}

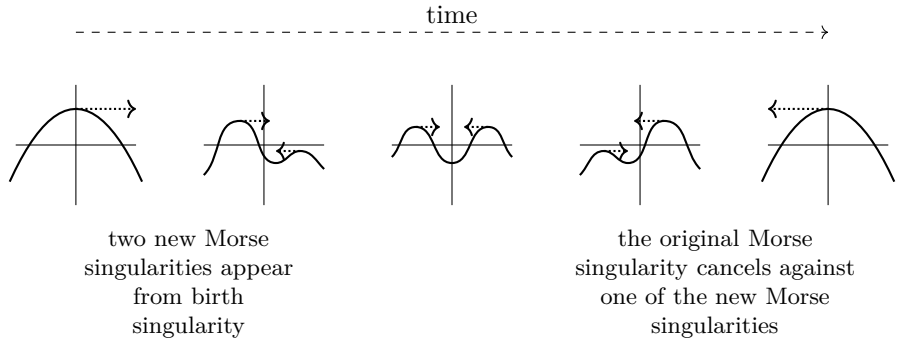
\begin{figure}
	
	\begin{tikzpicture}
	\draw [dashed,->] (-5,1.5) -- (5,1.5);
	\node at (0,1.5) [above] {time};
	\begin{scope}[shift={(-5,0)},scale=0.8]
	\draw (-1,0) -- (1,0);
	\draw (0,-1) -- (0,1);
	\draw[thick,domain=-1.1:1.1] plot (\x,{.6-\x*\x});
	\draw [thick,densely dotted,->] (0,.6) -- (1,.6);
	\end{scope}
	\begin{scope}[shift={(-2.5,0)},scale=0.8]
	\draw (-1,0) -- (1,0);
	\draw (0,-1) -- (0,1);
	\draw[thick] (-1,-.5) to[in=180] (-.4,.4) to[out=0,in=180]  (.2,-.3) to[out=0,in=180] (.6,-.1) to[out=0] (1,-.4);
	\draw [thick,densely dotted,->] (-.4,.4) -- (.1,.4);
	\draw [thick,densely dotted,->] (.6,-.1) -- (.2,-.1);
	\end{scope}
	\begin{scope}[shift={(0,0)},scale=0.8]
	\draw (-1,0) -- (1,0);
	\draw (0,-1) -- (0,1);
	\draw[thick] (-1,-.2) to[in=180] (-.6,.3) to[out=0,in=180]  (0,-.3) to[out=0,in=180] (.6,.3) to[out=0] (1,-.2);
	\draw [thick,densely dotted,->] (-.6,.3) -- (-.2,.3);
	\draw [thick,densely dotted,->] (.6,.3) -- (.2,.3);
	\end{scope}
	\begin{scope}[shift={(2.5,0)},scale=0.8]
	\draw (-1,0) -- (1,0);
	\draw (0,-1) -- (0,1);
	\draw[thick] (-1,-.5) to[in=180] (-.6,-.1) to[out=0,in=180] (-.2,-.3) to[out=0,in=180]   (.4,.4) to[out=0] (1,-.4);
	\draw [thick,densely dotted,->] (.4,.4) -- (-.1,.4);
	\draw [thick,densely dotted,->] (-.6,-.1) -- (-.2,-.1);
	\end{scope}
	\begin{scope}[shift={(5,0)},scale=0.8]
	\draw (-1,0) -- (1,0);
	\draw (0,-1) -- (0,1);
	\draw[thick,domain=-1.1:1.1] plot (\x,{.6-\x*\x});
	\draw[thick,densely dotted,->] (0,.6) -- (-1,.6);
	\end{scope}
	\node at (-3.5,-1) [below] {\parbox{3cm}{\small \centering two new Morse singularities appear from birth singularity}};
	\node at (3.5,-1) [below] {\parbox{4cm}{\small \centering the original Morse singularity cancels against one of the new Morse singularities}};
	\end{tikzpicture}	
	
	%\centering{
	%\resizebox{\textwidth}{!}{\import{}{orientationchange3.pdf_tex}}
	\caption{The simplest situation where the orientation of framing is changed by introducing a pair of birth-death singularities.}
	\label{figorientation}
	%}
	
\end{figure}

\section{Application III: foliations} \label{sec.foliations} Our next goal is to study certain spaces of foliations and reprove several famous results of Mather and Thurston, full proofs of which remain relatively hard to find in the literature.

We start by recalling basic definitions of foliation theory. A \emph{codimension $k$ $\mr{CAT}$-foliated atlas} for a manifold $N$ consists of an open cover $N$ by $U_i$ with charts $\phi_i \colon \bR^k \times \bR^{n-k} \to U_i \subset N$, so that the transition functions $\phi_j^{-1} \phi_i \colon \phi_i^{-1}(U_i \cap U_j) \to \phi_j^{-1}(U_i \cap U_j)$ are $\mr{CAT}$-isomorphisms of the form $(x,y) \mapsto (f(x),g(x,y))$.

\begin{definition}A \emph{codimension $k$ $\mr{CAT}$-foliation} $\cF$ on an $n$-dimensional manifold $N$ is a maximal codimension $k$ foliated atlas. \end{definition}

The subsets of the form $\{x\} \times \bR^{n-k}$ in a chart are called \emph{plaques} and they glue together to immersed $(n-k)$-dimensional manifolds called \emph{leaves}. One can give equivalent definitions of foliations in terms of leaves, and in the smooth case in terms of integrable distributions.  

If $p \colon N \times B \to B$ is the projection, then a foliation $\cF$ on $N \times B$ is said to be \emph{transverse to $p$} if the leaves are transverse to the fibers $p^{-1}(b)$. We will be interested in the situation where the foliation is of codimension $n$, i.e.\ the leaves have the same dimension as $B$. In that case transversality is equivalent to the existence of charts of the following form: near each $e \in E$ we have a chart $\phi_B \colon \bR^k \to B$ near $p(e)$ and foliated chart $\phi \colon \bR^k \times \bR^{n} \to E$ near $e$ so that $\pi \circ \phi = \phi_B \circ \pi_1$. We shall take this as the definition, avoiding questions about the exact definition of transversality when $\mr{CAT} = \mr{PL},\mr{Top}$.

\begin{definition}\label{def.folspace} If $M$ is an $n$-dimensional $\mr{CAT}$-manifold, we let $\mr{Fol}_\mr{CAT}(M)$ denote the element of $\cat{Sh^{CAT}}$ which assigns to a parametrizing manifold $P$ the set of $\mr{CAT}$-foliations of codimension $n$ on $P \times M$ that are transverse to the projection to $P$. We call it the \emph{space of $\mr{CAT}$-foliations on $M$}.\end{definition}

It is easiest to see that $M \mapsto \mr{Fol}_\mr{CAT}(M)$ is a $\mr{CAT}$-invariant topological sheaf on $n$-manifolds by thinking in terms of leaves; if $e \colon N \to M$ is an embedding, the inverse image under $\mr{id} \times e$ of a partition of $P \times M$ into leaves transverse to the projection to $P$ gives such a partition of $P \times N$, and similarly for families of embeddings.

\begin{figure}
	\centering
	\begin{tikzpicture}
		\draw (0,0) rectangle (6,3);
		\node at (3,0) [below] {$\Delta^1$};
		\node at (0,1.5) [left] {$[0,1]$};
		\foreach \i in {1,...,5}
		{
		\draw [Mahogany] (0,{\i/30*3}) -- (6,{\i/30*3});
		}
		\foreach \i in {25,...,29}
		{
			\draw [Mahogany] (0,{\i/30*3}) -- (6,{\i/30*3});
		}
		\foreach \i in {6,...,24}
		{
			\draw [Mahogany] plot [smooth] coordinates {(0,{\i/30*3}) (1,{\i/30*3})  ({3},{\i/30*3+(\i-6)*(24-\i)*(1-\i/24)/100}) (5,{\i/30*3})  (6,{\i/30*3-(\i-6)*(24-\i)/400})};
		}
	
		\draw [decorate,decoration={brace,amplitude=4pt},xshift=3pt,yshift=0pt]
		(6,3) -- (6,2.4) node [black,midway,xshift=.5cm,yshift=0cm] {$\cF_0$};
		
		\draw [decorate,decoration={brace,amplitude=4pt},xshift=3pt,yshift=0pt]
		(6,0.6) -- (6,0) node [black,midway,xshift=.5cm,yshift=0cm] {$\cF_0$};
	\end{tikzpicture}
	\caption{A $1$-simplex of $\mr{Fol}_\mr{Diff}([0,1] \rel \cF_0)$, where we take the boundary condition $\cF_0$ to be near the endpoints $\{0,1\} \subset [0,1]$.}
\end{figure}
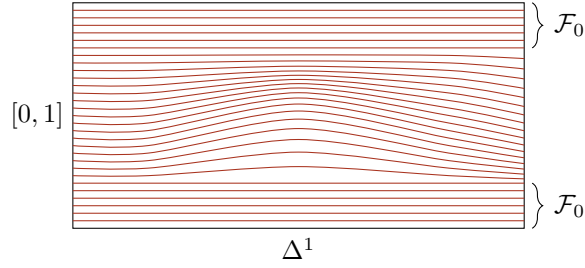

Note that $\mr{Fol}_\mr{CAT}(M)$ has a unique $0$-simplex $\cF_0$, and when writing $\mr{Fol}_\mr{CAT}(M \rel \cF_0)$ we shall take the boundary condition to be near $\partial M$ unless mentioned otherwise. The space $\mr{Fol}_\mr{CAT}(M \rel \cF_0)$ has an interpretation related to classifying spaces of groups of $\mr{CAT}$-isomorphisms. To state it, let $\mr{CAT}^\delta_\partial(M)$ denote the discrete group underlying the group of $\mr{CAT}$-isomorphisms of $M$ that are the identity on a neighborhood of $\partial M$. The inclusion $\mr{CAT}^\delta_\partial(M) \to \mr{CAT}_\partial(M)$ induces a map on classifying spaces $B\mr{CAT}^\delta_\partial(M) \to B\mr{CAT}_\partial(M)$.

\begin{lemma}\label{lem.folfibseq} If $M$ is compact there is a fiber sequence 
	\[\mr{Fol}_\mr{CAT}(M \rel \cF_0)\longto B\mr{CAT}^\delta_\partial(M) \longto B\mr{CAT}_\partial(M).\]
\end{lemma}

\begin{proof}$B\mr{CAT}_\partial(M)$ classifies bundles with fiber $M$ and structure group $\mr{CAT}_\partial(M)$, while $B\mr{CAT}^\delta_\partial(M)$ classifies such bundles with a codimension $n$ foliation transverse to the projection which is fiberwise supported away from a neighborhood of the boundary (this uses that $M$ is compact). Thus the homotopy fiber classifies trivial bundles with fiber $M$ but a possibly non-trivial codimension $n$ foliation transverse to the projection which is supported away from a neighborhood of the boundary.\end{proof}

The space $\mr{Fol}_\mr{CAT}(M)$ is easy to study in our framework, as it has only $\cF_0$ as a $0$-simplex. In particular $\mr{Fol}_\mr{CAT}(M \rel \cF_0)$ is path-connected and thus condition (H) is automatically satisfied, as in Example \ref{exam.grouplikesingle}.

\begin{remark}Condition (W) is \emph{not} satisfied. To see why, let us consider 
	\[\cF_0 \circledcirc - \colon \mr{Fol}_\mr{CAT}(S^{n-1} \times [t_0,t_1] \rel \cF_0) \longto \mr{Fol}_\mr{CAT}(S^{n-1} \times [t_0,t_2] \rel \cF_0).\] A reasonable guess for a homotopy inverse is $(\mr{id}_{S^{n-1}} \times \lambda)_*$ with $\lambda \colon [t_0,t_2] \to [t_0,t_1]$ a $\mr{CAT}$-isomorphism equal to translation near the boundary. However, this can not be a weak equivalence, since the corresponding map $B\mr{CAT}^\delta_\partial(S^{n-1} \times [t_0,t_2]) \to B\mr{CAT}^\delta_\partial(S^{n-1} \times [t_0,t_1])$ in Lemma \ref{lem.folfibseq} is not (e.g.~take $\pi_1$). Indeed, it is only a \emph{homology} equivalence. 
	
We can see more explicitly why this fails: $\cF_0 \circledcirc (\mr{id}_{S^{n-1}} \times \lambda)_*(\cF)$ is obtained from $\cF$ by shrinking $\cF$ onto $S^{n-1} \times [t_0,t_1]$ and extending it by $\cF_0$. Any attempt to ``deform away'' the additional part where it equals $\cF_0$ picks up a non-trivial foliation and hence does not preserve the basepoint.\end{remark}

 Thus if we want to prove a homological $h$-principle, it suffices to prove that $\mr{Fol}_\mr{CAT}(-)$ is microflexible. This is well-known, e.g.~Section 4 of \cite{siebenmannicm}, and is a consequence of the following lemma.

\begin{lemma}\label{lem.prodfol} Suppose $M$ is an $n$-dimensional manifold and $K \subset M$ compact. Let $U \subset M$ be an open neighborhood of $K$. Then for any map $h \colon \Delta^i \to \mr{Fol}_\mr{CAT}(M)$ (represented by a foliation $\mathcal{H}_0$ of $V \times M$ with $V$ a neighborhood of $\Delta^i$ in $\bR^i$), and any extension $H \colon \Delta^i \times [0,1] \to \mr{Fol}_\mr{CAT}(U)$ of its restriction to $U$ (represented by a foliation $\cH$ on $W \times [0,1] \times U$ with $W$ a neighborhood of $\Delta^i$ in $\bR^i$), there exist
	\begin{itemize}
		\item a neighborhood $W' \subset W$ of $\Delta^i$ in $\bR^i$,
		\item a neighborhood $U' \subset U$ of $K$,
		\item a real number $\epsilon > 0$, and
		\item a map $G \colon W' \times [0,\epsilon] \to \mr{Emb}^\mr{CAT}(U',U)$,
	\end{itemize}
	such that $G|_{W' \times \{0\}}$ is the identity and the adjoint $\tilde{G} \colon W' \times [0,\epsilon] \times U' \hookrightarrow W' \times [0,\epsilon] \times U$ has the property that $\tilde{G}^* \cH = \pi^* (\cH_0|_{W' \times U'})$ with $\pi \colon W' \times [0,\epsilon] \times U' \to W' \times U'$ the projection.
\end{lemma}

\begin{proof}Without loss of generality $V = W$. The map $G$ will be obtained by parallel transport along leaves. For every $x \in \Delta^i \times \{0\} \times U$, there exists a real number $\eta_x$ so that $\mr{int}(D^i_{\eta_x}(x)) \subset W$, and a foliated chart $\phi_x \colon \mr{int}(D^i_{\eta_x}(x)) \times [0,\epsilon_x) \times \bR^n \hookrightarrow W \times [0,1] \times U$ over $\mr{int}(D^i_{\eta_x}(x)) \times [0,\epsilon)$. Using this chart, we see that for any point $y = (d,t,u) \in W \times [0,1] \times U$ in the domain of this chart, there is a unique $\psi(y)$ in $W \times \{0\} \times U$ so that $y$ is obtained by parallel transport of $\psi(y)$ along the path $t \mapsto (d,t)$ in $W \times [0,1]$. Since the parallel transport is unique if it exists, $\psi(y)$ is independent of the choice of chart. 
	
	By compactness of $\Delta^i$ and $K$ there exists a finite number of such charts so that the open subsets $\phi_j(\mr{int}(D^i_{\eta_j}(x_j)) \times \{0\} \times \bR^n) \subset W \times \{0\} \times U$ cover $\Delta^i \times \{0\} \times K$. Let $\epsilon$ the minimum of the finitely many $\epsilon_i$'s, and let $W'$ and $U'$ be open neighborhoods of $\Delta^i$ and $K$ so that $W' \times \{0\} \times U' \subset \bigcup_j \phi_j(\mr{int}(D^i_{\eta_j}(x_j)) \times \{0\} \times \bR^n)$.
	
	We define $G(d,t,u)$ to be the point $z \in W' \times [0,\epsilon] \times U$ obtained by parallel transport of $u$ along the path $t \mapsto (d,t)$ in $W' \times [0,\epsilon)$. Uniqueness of parallel transport implies this is an embedding, and by construction $\tilde{G}^* \cH = \pi^* (\cH_0|_{W' \times U'})$.
\end{proof}

\begin{lemma}$\mr{Fol}_\mr{CAT}(-)$ is microflexible.\end{lemma}

\begin{proof}It suffices to prove that the restriction map
	\[\mr{Fol}_\mr{CAT}(R \subset M) \longto \mr{Fol}_\mr{CAT}(Q \subset M)\]
	is a microfibration for $Q \subset R$ compact submanifolds with corners in $M$. In each comutative diagram
	\[\begin{tikzcd} \Delta^i \times \{0\} \rar{h} \dar & \mr{Fol}_\mr{CAT}(R \subset M) \dar \\
	\Delta^i \times [0,1] \rar[swap]{H} & \mr{Fol}_\mr{CAT}(Q \subset M)\end{tikzcd}\]
	%\[\xymatrix{\Delta^i \times \{0\} \ar[r]^-{h} \ar[d] & \mr{Fol}_\mr{CAT}(R \subset M) \ar[d] \\
	%\Delta^i \times [0,1] \ar[r]_-H & \mr{Fol}_\mr{CAT}(Q \subset M)}\]
	we must find a partial lift. Note that the map $h$ is represented by a neighborhood $W$ of $\Delta^i$ in $\bR^i$, a neighborhood $U$ of $R$ in $M$ and a foliation $\cH_0$ on $W \times \{0\} \times U$ transverse to the projection onto $W \times \{0\}$. Similarly $H$ is represented by a neighborhood $W$ of $\Delta^i$ in $\bR^i$, a neighborhood $V$ of $Q$ in $M$ and a foliation $\cH$ on $W \times [0,1] \times V$ transverse to the projection onto $W \times [0,1]$. We may assume $V \subset U$. Now apply Lemma \ref{lem.prodfol} to the compact subset $Q$ of the manifold $V$, to obtain an embedding $G$ such that $G^* \cH$ is a product (we may assume $W' = W$ and $U' = U$ in the notation of that lemma). 
	
	Using isotopy extension as in Theorem 6.5 of \cite{siebenmanndeformation}, we can find a family compactly-supported $\mr{CAT}$-isomorphisms $\Lambda \colon W \times [0,\epsilon] \to \mr{CAT}_c(U)$ so that $\Lambda|_{W \times \{0\}}$ is the identity and $\Lambda$ agrees with $G$ in a neighborhood of $Q$. Pushing forward the product foliation $\cH_0 \times [0,\epsilon]$ along $\Lambda$ gives a foliation on $W \times [0,\epsilon] \times U$.\end{proof}

The following is closely related to Theorems 4 and 5 of \cite{thurstonfol}, and was stated before as Corollary \ref{cor.thurston}.

\begin{theorem}[Thurston] \label{thm.thurston} There is a flexible $\mr{CAT}$-invariant topological sheaf $\mr{Fol}_{\mr{CAT}}^f(-)$, whose values are weakly equivalent to a space of sections with fiber $\mr{Fol}_\mr{CAT}(\bR^n)$, so that the map
	\[j\colon \mr{Fol}_\mr{CAT}(M \rel \cF_0) \longto \mr{Fol}_{\mr{CAT}}^f(M \rel j(\cF_0))\]
	is a homology equivalence for all $\mr{CAT}$-manifolds $M$.\end{theorem}

\begin{remark}For $\mr{CAT} = \mr{Diff}, \mr{Top}$, one can identify the fibers of the section space weakly equivalent to $\mr{Fol}_{\mr{CAT}}^{f}(M)$ as the homotopy fiber of a map $B\Gamma^\mr{CAT}_n \to B\mr{CAT}(n)$, where $B\Gamma^\mr{CAT}_n$ is the so-called \emph{codimension $n$ Haefliger space}. The case $\mr{CAT} = \mr{PL}$ is more subtle, see \cite{gelfandpl1,gelfandpl2}.\end{remark}

This theorem has the following corollary, which appears on page 306 of \cite{thurstonfol}. See also \cite{mcduff,matherlectures,narimanlg}.

\begin{corollary}[Mather-Thurston] For all manifolds $M$, $\tilde{H}_*(\mr{Fol}_\mr{Top}(M \rel \cF_0)) = 0$, and $B\mr{Top}_c^\delta(\mr{int}(M)) \to B\mr{Top}_c(\mr{int}(M))$ is a homology equivalence.\end{corollary}

\begin{proof}It is now helpful to think of $\bR^n$ as the interior of $D^n$, so that $\mr{Top}_c(\bR^n) \cong \mr{Top}_\partial(D^n)$. This allows us to see that $B\mr{Top}_c(\bR^n) \simeq *$ by the Alexander trick. We use this in Lemma \ref{lem.folfibseq} with $M = D^n$ to see that $\tilde{H}_*(\mr{Fol}_{\mr{Top}}(D^n \rel \cF_0)) \cong \tilde{H}_*(B\mr{Top}^\delta_c(\bR^n))$. In \cite{matherhomology} Mather proved that the latter vanishes.
	
Theorem \ref{thm.thurston} thus says that $\Omega^n(\mr{Fol}_{\mr{Top}}^f(\bR^n))$ is acyclic. Since its an $n$-fold loop space and hence simple, this implies it is weakly contractible. Thus $\smash{\pi_k(\mr{Fol}_{\mr{Top}}^f(\bR^n))} = 0$ for $k \geq n$. For the remaining homotopy groups, for $k < n$ Gromov's $h$-principle on open manifolds says that $\smash{\pi_k(\mr{Fol}_{\mr{Top}}^f(\bR^n))}$ is in bijection with concordance classes of codimension $n$ topological foliations on the $n$-dimensional manifold $S^k \times \bR^{n-k}$, but there is only one such foliation. We conclude that $\mr{Fol}_{\mr{Top}}^f(\bR^n)$ is weakly contractible, cf.~Theorem 3 of \cite{thurstonfol}, so using Theorem \ref{thm.thurston} and Lemma \ref{lem.folfibseq} the corollary follows.\end{proof}

\appendix

\section{Categories of spaces} \label{app.spaces} In this appendix we describe a category of ``spaces'' that in our opinion is most convenient for studying $h$-principles: sheaves on parametrizing manifolds. We will also discuss three other choices, and their advantages and disadvantages.

\subsection{Sheaves on parametrizing manifolds}\label{subsec.sheavesparam} A convenient category of spaces $\cat{S}$ is the category $\cat{Sh^{CAT}}$ of sheaves on parametrizing manifolds. Before getting into the details, we state three reasons to prefer $\cat{Sh^{CAT}}$ as one's notion of spaces:
\begin{itemize}
	\item It has convenient technical properties, in particular concerning colimits.
	\item It is easy to do constructions locally in the parametrizing object.
	\item It is a natural setting for geometric objects such as spaces of smooth structures or foliations. For example, it was used in \cite{madsenweiss} and \cite{galatiuseliashberg}.
\end{itemize}

In the remainder of this appendix we will define $\cat{Sh^{CAT}}$ and show that the constructions of this paper make sense in this context.

\subsubsection{Basic definitions} We fix a category of manifolds, i.e.\ let $\mr{CAT}$ be $\mr{Diff}$, $\mr{PL}$ or $\mr{Top}$.

\begin{definition} We let $\cat{CAT}$ be the category with objects the $\mr{CAT}$-manifolds with empty boundary, and morphisms the $\mr{CAT}$-maps. We will call its objects \emph{parametrizing manifolds}.
\end{definition}

\begin{remark}It may be desirable to restrict to submanifolds of $\bR^\infty$, so that $\cat{CAT}$ is a small category. We will ignore this.\end{remark}

\begin{definition}Let $\cat{Sh^{CAT}}$ be the category defined as follows:
	\begin{itemize}
		\item Objects are the \emph{sheaves on parametrizing manifolds}, i.e.\ functors $F \colon \cat{CAT}^\mr{op} \to \cat{Set}$ such that for all $\mr{CAT}$-manifolds $X$ and all open covers  $\{U_i\}$ of  $X$ the following diagram is an equalizer of sets
		\[\begin{tikzcd}F(X) \rar & \prod_i F(U_i) \arrow[shift left=.5ex]{r} \arrow[shift left=-0.5ex]{r}& \prod_{i,j} F(U_i \cap U_j ).\end{tikzcd}\]
		\item Morphisms are natural transformations.
	\end{itemize}
\end{definition}

Many desirable properties of the category of sheaves on parametrizing manifolds follow from the fact that it is a Grothendieck topos. In particular it is complete and cocomplete.

We often think of $F(X)$ as $X$-indexed families of objects and hence of $F$ as a moduli space. In fact, if $X$ is a $\mr{CAT}$-manifold then the Yoneda lemma says that the representable functor $\mr{CAT}(-,X)$ has the property that the set of natural transformation $\mr{CAT}(-,X) \to F$ equals $F(X)$. Thus if we use the shorthand denoting $\mr{CAT}(-,X)$ by $X$, the following alternative notion is unambiguous:

\begin{convention}$X \smash{\overset{f}{\to}} F$ is alternative notation for an element $f$ of $F(X)$.
\end{convention}

The following example should also illuminate this notation.

\begin{example}\label{exam.scatshcat} There is a functor $S \colon \cat{Top} \to \cat{Sh^{CAT}}$ given by $S(Z) \coloneqq (X \mapsto \mr{Map}(X,Z))$, with $\mr{Map}(X,Z)$ the set of continuous maps $X \to Z$.\end{example}

\subsubsection{Extension to pairs} It is helpful to consider a relative version of $F(X)$, defined on pairs of a manifold and a closed subset.

\begin{definition}Suppose we are given a closed subset $A \subset X$, an open subset $U \subset X$ containing $A$, and an element $a \in F(U)$. Let $\cat{I}_A$ be the directed set of open neighborhoods $W \subset U$ of $A$ (where there is a unique morphism $W \to W'$ if $W' \subset W$). We then define $F(X,A \rel a)$ as
	\[F(X,A \rel a) \coloneqq \underset{W \in \cat{I}_A}{\mr{colim}}\, F(X,W \rel a|_W),\]
	where $F(X,W \rel a|_W)$ is the inverse image of $a|_W$ under the restriction map $F(X) \to F(W)$.\end{definition}

One easily sees this definition only depends on the germ of $a$ near $A$. We will occasionally drop $A$ from the notation, when it is clear from the context.

\subsubsection{Extension to manifolds with corners}

There is a natural extension of a sheaf on parametrizing manifolds from manifolds without boundary to manifolds with corners.

\begin{definition}\label{def.paramcorners} If $X$ is a manifold with corners contained in $Y \in \cat{CAT}$, then let $\cat{I}_X$ be the directed set of open neighborhoods $U$ of $X$ in $Y$ (where there is a unique morphism $U \to U'$ if $U' \subset U$). We define
	\[F(X \subset Y) \coloneqq \underset{U \in \cat{I}_X}{\mr{colim}}\, F(U).\]
\end{definition}

There is also a relative version of $F(X \subset Y)$ with respect to a germ $a$ in $A \subset X$, the notation for which is $F(X \subset Y, A \rel a)$.

Note that if we have $X \subset Y$ and $X \subset Y'$ as above, with $Y$ and $Y'$ of the same dimension, then there is an isomorphism $F(X \subset Y) \cong F(X \subset Y')$, which depends only on a choice of an embedding $\phi \colon U \hookrightarrow Y'$ of neighborhood $U$ of $X$ in $Y$ such that $\phi|_X = \mr{id}|_X$. We will occasionally drop $Y$ from the notation, when it is clear from the context.

\subsubsection{Internal mapping spaces} The category $\cat{Sh^{CAT}}$ has certain internal mapping spaces. Suppose $X$ is a manifold with corners contained in $Y \in \cat{CAT}$, $A \subset X$ is closed and $a \in F(A \subset X)$. 

\begin{definition}We set $\mr{Map}_\cat{Sh^{CAT}}(X,F \rel a)$ to be the sheafification of the presheaf that assigns to a parametrizing manifold $P$ the set $F(X \times P \rel a \times P)$, where $a \times P$ denotes the pullback of $a$ along the map $X \times P \to X$.\end{definition}

\subsubsection{Weak equivalences of sheaves on parametrizing manifolds} There are several equivalent ways to define weak equivalences of sheaves on parametrizing manifolds.

Firstly, we can extract two weakly equivalent simplicial sets out of $F$. To do so, note there is a faithful functor $\Delta \to \cat{CAT}$, sending $[n]$ to the extended $n$-simplex $\Delta^n_e$, defined as the span of the basis vectors $e_i$ in $\bR^{n+1}$. Using these manifolds we can construct a simplicial set $\Delta_e(F)_\bullet$ with $\Delta_e(F)_n \coloneqq F(\Delta^n_e)$, and this gives a functor $\Delta_e \colon \cat{Sh^{CAT}} \longto \cat{sSet}$. Alternatively, we can think of the standard simplices as manifolds with corners, and define a simplicial set $\Delta(F)_n \coloneqq F(\Delta^n \subset \Delta^n_e)$, 
giving a functor $\Delta \colon \cat{Sh^{CAT}} \longto \cat{sSet}$. Note that for general $F$, neither $\Delta_e(F)$ or $\Delta(F)$ is Kan.

\begin{lemma}\label{lem.deltaedeltaweq} Restriction gives a natural weak equivalence $\Delta_e \to \Delta$ of functors $\cat{Sh^{CAT}} \to \cat{sSet}$. \end{lemma}

%\begin{proof}Let $F$ be a sheaf on parametrizing manifolds. By the simplicial approximation theorem, it suffices to prove that for all $i \geq 0$ and in each commutative diagram 
%\[\xymatrix{\partial D^{i} \ar[r] \ar[d] & \Delta_e(F) \ar[d] \\
%D^i \ar[r] \ar@{.>}[ru] & \Delta(F)}\]
%with $(D^i,\partial D^{i})$ a triangulated simplicial pair and all maps simplicial, there exists a dotted lift making diagram commute. By definition of $\Delta(F)$ as a colimit, every $k$-simplex $\sigma$ of $D^i$ is represented by an $f_\sigma \in F(U_\sigma)$, where $U_\sigma$ is an open neighborhood of $\Delta^k$ in $\Delta^k_e$. There are $\mr{CAT}$-embeddings $\phi_\sigma \colon \Delta^k_e \to U_\sigma$ that are the identity near $U_\sigma$. Then $\phi_\sigma^*f_\sigma$ lies in $F(\Delta^k_e)$ and represents the same element of $\Delta(F)_k$. Applying this argument to all simplices $\sigma$ in $D^i \setminus \partial D^{i}$ provides the desired lift.
%\end{proof}

%Remark that the proof of this lemma did not use the sheaf property. 

Since disks and spheres are manifolds, we can define homotopy groups of a sheaf on parametrizing manifolds without leaving the world of sheaves. If $f_0 \in F(\ast)$, then $\pi_n(F,f_0)$ is given by the equivalence classes of $f \in F(D^n,\partial D^n \rel f_0)$, where we remark that $f_0$ can be pulled back along the unique map from a neighborhood $U$ of $\partial D^n$ in $D^n$ to $\ast$. The equivalence relation says that $f$ is equivalent to $f'$ if there exists an element $\bar{f} \in F(D^n \times [0,1],\partial D^{n-1} \times [0,1] \rel f_0)$ such that $\bar{f}|_{D^n \times \{0\}} = f$ and $\bar{f}|_{D^n \times \{1\}} = f'$. This coincides with $\pi_0$ of the internal mapping space object from $D^n$ to $F$ relative to $f_0$ on $\partial D^n$, or from $S^n \to F$ relative to $f_0$ on $\ast$.

We can also define the relative homotopy $\pi_n(\eta)$ of a map $\eta \colon F \to G$. It is given by equivalence classes of pairs $(g,f)$ of $g \in G(D^n)$ and $f \in F(\partial D^n)$ such that $g|_{\partial D^n} = \eta(f)$. The equivalence relation says that $(g,f)$ is equivalent to $(g',f')$ if there exists a $\bar{g} \in G(D^n \times [0,1])$ and $\bar{f} \in F(\partial D^n \times [0,1])$ such that $\bar{g}|_{\partial D^n \times [0,1]} = \eta(\bar{f})$, $\bar{f}|_{D^n \times \{1\}}$, $\bar{g}|_{D^n \times \{0\}} = g$, $\bar{f}|_{\partial D^n \times \{0\}} = f$, $\bar{g}|_{D^n \times \{1\}} = g'$, and $\bar{f}|_{\partial D^n \times \{1\}} = f'$. It is a standard argument that $\pi_n(\eta) = 0$ for all $n$ if and only if $\eta$ induces an isomorphism on $\pi_n$ for all $n$ and base points $f_0$.

\begin{lemma}\label{lem.wegshparam} Let $\eta \colon F \to G$ be a map of sheaves on parametrizing manifolds. Then the following are equivalent:
	\begin{enumerate}[(i)]
		\item For all $n$ and $f_0$, $\eta \colon F \to G$ induces a bijection $\pi_n(F,f_0) \to \pi_n(G,\eta(f_0))$.
		\item $\eta \colon \Delta(F) \to \Delta(G)$ is a weak equivalence.
	\end{enumerate} 
	By Lemma \ref{lem.deltaedeltaweq} we may replace $\Delta$ by $\Delta_e$. By the remarks preceding the lemma, we may also phrase everything in terms of relative homotopy groups.
\end{lemma}

\begin{definition}\label{def.wegshparam}We say that a map $\eta \colon F \to G$ of sheaves on parametrizing manifolds is a \emph{weak equivalence} if one of the equivalent conditions of Lemma \ref{lem.wegshparam} is satisfied.\end{definition}

Generalizing the remark about internal $\pi_n$ being defined in terms of $\pi_0$ of an internal mapping space object, we have the following lemma.

\begin{lemma}There is a natural weak equivalence between $|\Delta_e(\mr{Map}_\cat{Sh^{CAT}}(X,F \rel a))|$ and $\mr{Map}_\cat{Top}(X,|\Delta_e(F)| \rel a)$.\end{lemma}

\subsubsection{Homology equivalences of sheaves on parametrizing manifolds}

The discussion of homology equivalences involves the generalized homology theory $\Omega^\mr{SO}_*$ known as \emph{oriented bordism}.  We can define $\Omega^\mr{SO}_*$ on $F \in \cat{Sh^{CAT}}$ directly. The abelian group $\Omega^\mr{SO}_k(X)$ is given by equivalence classes of pairs $(M,f)$ of a $k$-dimensional oriented closed smooth manifold $M$ and a element $f \in F(X)$, under the equivalence relation of oriented bordism. This equivalence relation says that $(M,f)$ and $(M',f')$ are equivalent if there is a $(k+1)$-dimensional oriented compact smooth manifold $W$ together with a $\bar{f} \in F(W)$ such that $\partial W \cong M \cup \bar{M}'$ (where $\bar{M}'$ is $M'$ with the opposite orientation) and $\bar{f}|_M = f$, $\bar{f}|_{M'} = f'$. 

%Note that if $f: M \to X$ and $f': M \to X$ are homotopic, then $(M,f)$ and $(M,f')$ are bordant.

There is also a notion of relative oriented bordism groups for a map $\eta \colon F \to G$. The abelian group $\Omega^\mr{SO}_k(\eta)$ is given by equivalence classes of triples $(M,g,f)$ of a smooth oriented manifold $M$ with boundary $\partial M$, $g \in G(M)$ and $f \in F(\partial M)$ such that $g|_{\partial M} = \eta(f)$. The equivalence relation says that $(M,g,f)$ and $(M',g',f')$ are equivalent if there exists a bordism $(W,\partial W)$ of smooth oriented manifolds with boundary from $(M,\partial M)$ to $(M',\partial M')$, together with $\bar{g} \in G(W)$ and $\bar{f} \in F(\partial W)$, such that $\bar{g}|_{\partial W} = \eta(\bar{f})$ $\bar{g}|_{M} = g$, $\bar{f}|_{\partial M} = f$, $\bar{g}|_{M'} = g'$ and $\bar{f}|_{\partial M'} = f'$. It is a standard result that the relative bordism groups $\Omega^\mr{SO}_*(\eta)$ are $0$ if and only if $\eta$ induces an isomorphism on $\Omega^\mr{SO}_*$.

It is a well-known fact that a map $X \to Y$ of simplicial sets or topological spaces induces an isomorphism on homology if and only if it induces an isomorphism on oriented bordism. A reference for this is Appendix B of \cite{galatiuseliashberg}, but it also follows from the Atiyah-Hirzebruch spectral sequence.

\begin{lemma}\label{lem.homologequivalencedef} Let $\eta \colon F \to G$ be a map of sheaves on parametrizing manifolds. The following are equivalent: \begin{enumerate}[(i)]
		\item $\eta \colon F \to G$ induces an isomorphism on $\Omega^\mr{SO}_*$,
		\item $\eta \colon \Delta(F) \to \Delta(G)$ is an oriented bordism equivalence,
		\item $\eta \colon \Delta(F) \to \Delta(G)$ is a homology equivalence.
	\end{enumerate}
	Note that by Lemma \ref{lem.deltaedeltaweq} we may replace $\Delta$ by $\Delta_e$. By the remarks preceding the lemma, we may also phrase everything in terms of relative oriented bordism or homology groups.
\end{lemma}

%\begin{proof}The equivalence of (i) and (ii) is proven as in Lemma \ref{lem.wegshparam}. The equivalence of (ii) and (iii) is a consequence of and 
%\end{proof}

\begin{definition}We say that a map $\eta \colon F \to G$ of sheaves on parametrizing manifolds is a \emph{homology equivalence} if one of the equivalent properties of Lemma \ref{lem.homologequivalencedef}  is satisfied.\end{definition}

\subsection{Other choices for the category of spaces} \label{subsec.classicalspaces} In this section we discuss three other categories of spaces that can take the role of $\cat{S}$ in the paper, and compare them to $\cat{Sh^{CAT}}$.

\subsubsection{$\cat{Top}$, {topological spaces}} 	The obvious choice for a category of spaces is the category $\cat{Top}$ of topological spaces. Unfortunately, this category does not behave well with respect to colimits, as we will discuss now.

One might expect that maps out of a compact space commute with filtered colimits, i.e.\ that compact spaces are compact objects in the categorical sense. This is false: every compact metric space is the filtered colimit of its countable subsets with the subspace topology, but clearly not every map from a compact space into a compact metric space has countable image. Page 50 of \cite{hovey} (also see the errata) has a counterexample where the colimit is sequential and $X$ is the two point space with indiscrete topology. 

However, it is true that a map $g \colon X \to \mr{colim}_{\cat{I}}\, F$ factors over some $F(i)$ under some more restrictive conditions on $X$ and the diagram $F$. Recall that an inclusion $X \to Y$ is \emph{relatively $T_1$} if for any open $U$ in $X$ and any $z \in Y \setminus U$, there is an open subset $V$ of $Y$ such that $U \subset V$ and $z \notin V$. The following is Lemma A.3 of \cite{duggerisaksen}.

\begin{lemma}\label{lem.colimit1}Any map $X \to \mr{colim}_{\cat{I}}\, F$ factors over some $F(i)$ if $I$ is sequential, all maps $F(i) \to F(j)$ are relatively $T_1$, and $X$ is compact.\end{lemma}

\subsubsection{$\cat{sSet}$, {simplicial sets}} Simplicial sets are well-behaved with respect to colimits. In particular, every map $X \to \mr{colim}_\cat{I} F$ factors over some $F(i)$. An additional advantage of simplicial sets is that it is easy to write down spaces of smooth structures or foliations. Their main disadvantage is that it is not easy to do local constructions. In particular, many of our proofs would require barycentric subdivisions.

\subsubsection{$\cat{QTop}$, {quasitopological spaces}}

Quasitopological spaces are a concept originally due to Spanier and Whitehead, and Gromov used them to replace spaces in \cite{gromovhp}. They are well-behaved with respect to colimits and allow for easy local constructions. We repeat their definition for the convenience of the reader:

\begin{definition}A \emph{quasitopological space} consists of a set $A$ and for each topological space $Z$ a subset $A(Z)$ of $\mr{Fun}(Z,A)$ (the functions of underlying sets from $Z$ to $A$) called ``continuous.'' These have to satisfy:
	\begin{enumerate}[(i)]
		\item If $f \colon Z \to Z'$ is continuous and $g \in A(Z')$, then $g \circ f \in A(Z)$.
		\item If $\{U_i\}$ is an open cover of $Z$, then $f \colon Z \to A$ is in $A(Z)$ if all $f|_{U_i} \in A(U_i)$.
		\item If $Z = Z_1 \cup Z_2$ with $Z_1$, $Z_2$ closed, then $f \in A(Z)$ if $f|_{Z_i} \in A(Z_i)$.
	\end{enumerate}
	
	A map of quasitopological spaces is a map $F \colon A \to A'$ such that for all $Z$ we have $(F \circ -)(A(Z)) \subset A'(Z)$. We denote this category $\cat{QTop}$.\end{definition}

It seems unnecessary that $Z$ ranges over all topological spaces, including very wild ones. Indeed, other types of quasitopological spaces are appear in the literature. For example, Essay V of \cite{kirbysiebenmann} restricts the test spaces from all spaces to either Hausdorff compacta or polyhedra. Furthermore, quasitopological spaces do not allow for the full generality of sheaves one might want to treat (e.g. foliations and smooth structures), without making some unnatural choices.

\subsubsection{Comparing categories of spaces} We explain how to compare these different categories of spaces. We start with $\cat{Top}$, $\cat{QTop}$ and $\cat{sSet}$. 

There is a functor $E \colon \cat{QTop} \to \cat{sSet}$ given by $E(A)_n \coloneqq A(\Delta^n)$.  There is also a functor $Q \colon \cat{Top} \to \cat{QTop}$ given by $Q(X) \coloneqq (Z \mapsto \mr{Map}(Z,X))$, with $\mr{Map}(Z,X)$ the set of continuous maps.  There is a commutative diagram
\[\begin{tikzcd} \cat{Top} \arrow{rd}[swap]{\mr{Sing}} \rar{Q} & \cat{QTop} \dar{E} \\ & \cat{sSet}.\end{tikzcd}\]
%\[\xymatrix{\cat{Top} \ar[rd]_{\mr{Sing}} \ar[r]^Q & \cat{QTop} \ar[d]^{E} \\ & \cat{sSet}}\]
The functor $\mr{Sing}$ is part of a Quillen equivalence between simplicial sets with the Quillen model structure and topological spaces with the Quillen model structure. We expect there exists also a model structure on quasitopological spaces such that $E$ is part of a Quillen equivalence with simplicial sets with the Quillen model structure, but do not know a reference for this.

In Subsection \ref{subsec.sheavesparam} we gave functors relating $\cat{Top}$, $\cat{sSet}$ and $\cat{Sh^{CAT}}$. In particular, recall that there is a functor $S \colon \cat{Top} \to \cat{Sh^{CAT}}$, given by $S(Z) \coloneqq (P \mapsto \mr{Map}(P,Z))$, with $\mr{Map}(X,Z)$ the set of continuous maps $X \to Z$. There is also a functor $S' \colon \cat{QTop} \to \cat{Sh^{CAT}}$ given by $S'(A) \coloneqq (P \mapsto A(P))$. There is a commutative diagram
\[\begin{tikzcd} \cat{Top} \arrow{rd}[swap]{S} \rar{Q} & \cat{QTop} \dar{S'} \\ & \cat{Sh^{CAT}}.\end{tikzcd}\]
%\[\xymatrix{\cat{Top} \ar[rd]_{S} \ar[r]^Q & \cat{QTop} \ar[d]^{S'} \\ & \cat{Sh^{CAT}}}\]
We also gave a functor $\Delta_e \colon \cat{Sh^{CAT}} \to \cat{sSet}$, which does not satisfy $\Delta_e \circ S = \mr{Sing}$. However, we do have the following.

\begin{lemma}There is a natural weak equivalence $\Delta_e \circ S \to \mr{Sing}$ of functors $\cat{Top} \to \cat{sSet}$.
\end{lemma}

A weak or homology equivalence in $\cat{Sh^{CAT}}$ becomes a weak or homology equivalence of simplicial sets upon applying $\Delta_e$. Thus this lemma says there is no loss in working with $\cat{Sh^{CAT}}$ if one is interested in the homotopy or homology groups of a space.

However, in contrast to quasitopological spaces, we can make a more precise comparison. In Section 6.1 of \cite{cisinskisheaves}, Cisinski discusses geometric models for the homotopy category of spaces. In particular, he constructs a model structure on $\cat{Sh^{CAT}}$ which is Quillen equivalent to $\cat{sSet}$ with the Quillen model structure.

\bibliographystyle{amsalpha}
\bibliography{cell}

\def\cprime{$'$} \def\cprime{$'$}
\providecommand{\bysame}{\leavevmode\hbox to3em{\hrulefill}\thinspace}
\providecommand{\MR}{\relax\ifhmode\unskip\space\fi MR }
% \MRhref is called by the amsart/book/proc definition of \MR.
\providecommand{\MRhref}[2]{%
  \href{http://www.ams.org/mathscinet-getitem?mr=#1}{#2}
}
\providecommand{\href}[2]{#2}
\begin{thebibliography}{GTMW09}

\bibitem[AGZV12]{arnoldsing}
V.~I. Arnold, S.~M. Gusein-Zade, and A.~N. Varchenko, \emph{Singularities of
  differentiable maps. {V}olume 1}, Modern Birkh\"auser Classics,
  Birkh\"auser/Springer, New York, 2012, Classification of critical points,
  caustics and wave fronts, Translated from the Russian by Ian Porteous based
  on a previous translation by Mark Reynolds, Reprint of the 1985 edition.
  \MR{2896292}

\bibitem[BL74]{burglashof}
Dan Burghelea and Richard Lashof, \emph{The homotopy type of the space of
  diffeomorphisms. {I}, {II}}, Trans. Amer. Math. Soc. \textbf{196} (1974),
  1--36; ibid. 196\ (1974), 37--50. \MR{0356103 (50 \#8574)}

\bibitem[Cis03]{cisinskisheaves}
Denis-Charles Cisinski, \emph{Faisceaux localement asph\'eriques}, preprint
  (2003), \url{http://www.math.univ-toulouse.fr/~dcisinsk/mtest2.pdf}.

\bibitem[DI04]{duggerisaksen}
Daniel Dugger and Daniel~C. Isaksen, \emph{Topological hypercovers and {$\Bbb
  A^1$}-realizations}, Math. Z. \textbf{246} (2004), no.~4, 667--689.
  \MR{2045835 (2005d:55026)}

\bibitem[Dot14]{dotto}
Emanuele Dotto, \emph{A relative {$h$}-principle via cobordism-like
  categories}, An alpine expedition through algebraic topology, Contemp. Math.,
  vol. 617, Amer. Math. Soc., Providence, RI, 2014, pp.~133--155. \MR{3243396}

\bibitem[EGM11]{galatiuseliashberg}
Yakov Eliashberg, S{\o}ren Galatius, and Nikolai Mishachev,
  \emph{Madsen-{W}eiss for geometrically minded topologists}, Geom. Topol.
  \textbf{15} (2011), no.~1, 411--472. \MR{2776850}

\bibitem[EM00]{emgmf}
Y.~M. Eliashberg and N.~M. Mishachev, \emph{Wrinkling of smooth mappings. {II}.
  {W}rinkling of embeddings and {K}. {I}gusa's theorem}, Topology \textbf{39}
  (2000), no.~4, 711--732. \MR{1760426 (2001g:57058)}

\bibitem[EM02]{emhpbook}
Y.~Eliashberg and N.~Mishachev, \emph{Introduction to the {$h$}-principle},
  Graduate Studies in Mathematics, vol.~48, American Mathematical Society,
  Providence, RI, 2002. \MR{1909245 (2003g:53164)}

\bibitem[EM12]{emframed}
Y.~M. Eliashberg and N.~M. Mishachev, \emph{The space of framed functions is
  contractible}, Essays in mathematics and its applications, Springer,
  Heidelberg, 2012, pp.~81--109. \MR{2975585}

\bibitem[ERW17]{rwebertsemi}
Johannes Ebert and Oscar Randal-Williams, \emph{Semi-simplicial spaces},
  \url{https://arxiv.org/abs/1705.03774}, 2017.

\bibitem[GF73]{gelfandpl1}
I.~M. Gel{\cprime}fand and D.~B. Fuks, \emph{{${\rm PL}$} foliations},
  Funkcional. Anal. i Prilo\v zen. \textbf{7} (1973), no.~4, 29--37.
  \MR{0339195}

\bibitem[GF74]{gelfandpl2}
\bysame, \emph{{${\rm PL}$} foliations. {II}}, Funkcional. Anal. i Prilo\v zen.
  \textbf{8} (1974), no.~3, 7--11. \MR{0418115}

\bibitem[GG73]{ggstable}
M.~Golubitsky and V.~Guillemin, \emph{Stable mappings and their singularities},
  Springer-Verlag, New York-Heidelberg, 1973, Graduate Texts in Mathematics,
  Vol. 14. \MR{0341518 (49 \#6269)}

\bibitem[Gro86]{gromovhp}
Mikhael Gromov, \emph{Partial differential relations}, Ergebnisse der
  Mathematik und ihrer Grenzgebiete (3) [Results in Mathematics and Related
  Areas (3)], vol.~9, Springer-Verlag, Berlin, 1986. \MR{864505 (90a:58201)}

\bibitem[GTMW09]{gmtw}
S{\o}ren Galatius, Ulrike Tillmann, Ib~Madsen, and Michael Weiss, \emph{The
  homotopy type of the cobordism category}, Acta Math. \textbf{202} (2009),
  no.~2, 195--239. \MR{2506750 (2011c:55022)}

\bibitem[Hov99]{hovey}
Mark Hovey, \emph{Model categories}, Mathematical Surveys and Monographs,
  vol.~63, American Mathematical Society, Providence, RI, 1999. \MR{1650134
  (99h:55031)}

\bibitem[Igu84a]{igusagmfh}
Kiyoshi Igusa, \emph{Higher singularities of smooth functions are unnecessary},
  Ann. of Math. (2) \textbf{119} (1984), no.~1, 1--58. \MR{736559 (85k:57034)}

\bibitem[Igu84b]{igusagmf}
\bysame, \emph{On the homotopy type of the space of generalized {M}orse
  functions}, Topology \textbf{23} (1984), no.~2, 245--256. \MR{744854
  (86m:57034)}

\bibitem[Igu87]{igusaframed}
\bysame, \emph{The space of framed functions}, Trans. Amer. Math. Soc.
  \textbf{301} (1987), no.~2, 431--477. \MR{882699 (88g:57034)}

\bibitem[Igu88]{igusastab}
\bysame, \emph{The stability theorem for smooth pseudoisotopies}, $K$-Theory
  \textbf{2} (1988), no.~1-2, vi+355. \MR{972368 (90d:57035)}

\bibitem[Igu02]{igusatorsion}
\bysame, \emph{Higher {F}ranz-{R}eidemeister torsion}, AMS/IP Studies in
  Advanced Mathematics, vol.~31, American Mathematical Society, Providence, RI;
  International Press, Somerville, MA, 2002. \MR{1945530}

\bibitem[KS77]{kirbysiebenmann}
Robion~C. Kirby and Laurence~C. Siebenmann, \emph{Foundational essays on
  topological manifolds, smoothings, and triangulations}, Princeton University
  Press, Princeton, N.J.; University of Tokyo Press, Tokyo, 1977, With notes by
  John Milnor and Michael Atiyah, Annals of Mathematics Studies, No. 88.
  \MR{0645390 (58 \#31082)}

\bibitem[Loj64]{loj}
S.~Lojasiewicz, \emph{Triangulation of semi-analytic sets}, Ann. Scuola Norm.
  Sup. Pisa (3) \textbf{18} (1964), 449--474. \MR{0173265 (30 \#3478)}

\bibitem[Lur09]{lurietft}
Jacob Lurie, \emph{On the classification of topological field theories},
  Current developments in mathematics, 2008, Int. Press, Somerville, MA, 2009,
  pp.~129--280. \MR{2555928 (2010k:57064)}

\bibitem[Mat71]{matherhomology}
John~N. Mather, \emph{The vanishing of the homology of certain groups of
  homeomorphisms}, Topology \textbf{10} (1971), 297--298. \MR{0288777}

\bibitem[Mat11]{matherlectures}
\bysame, \emph{On the homology of {H}aefliger's classifying space},
  Differential Topology (Vinicio Villani, ed.), Springer Berlin Heidelberg,
  Berlin, Heidelberg, 2011, pp.~71--116.

\bibitem[May72]{M}
J.~P. May, \emph{The geometry of iterated loop spaces}, Springer-Verlag,
  Berlin-New York, 1972, Lecture Notes in Mathematics, Vol. 271. \MR{0420610
  (54 \#8623b)}

\bibitem[McD80]{mcduff}
Dusa McDuff, \emph{The homology of some groups of diffeomorphisms}, Comment.
  Math. Helv. \textbf{55} (1980), no.~1, 97--129. \MR{569248}

\bibitem[MS76]{MSe}
D.~McDuff and G.~Segal, \emph{Homology fibrations and the ``group-completion''
  theorem}, Invent. Math. \textbf{31} (1975/76), no.~3, 279--284. \MR{0402733
  (53 \#6547)}

\bibitem[MW07]{madsenweiss}
Ib~Madsen and Michael Weiss, \emph{The stable moduli space of {R}iemann
  surfaces: {M}umford's conjecture}, Ann. of Math. (2) \textbf{165} (2007),
  no.~3, 843--941. \MR{2335797}

\bibitem[Nar17]{narimanlg}
Sam Nariman, \emph{A local to global argument on low dimensional manifolds},
  \url{https://arxiv.org/abs/1706.04602}, 2017.

\bibitem[Rap17]{raptis}
George Raptis, \emph{On {S}erre microfibrations and a lemma of {M}.\ {W}eiss},
  Glasgow Mathematical Journal (2017), 1--{\tiny }9.

\bibitem[RW12]{weissreis}
Rui Reis and Michael~S. Weiss, \emph{Smooth maps to the plane and {P}ontryagin
  classes {P}art {I}: {L}ocal aspects}, Port. Math. \textbf{69} (2012), no.~1,
  41--67. \MR{2900651}

\bibitem[RW14]{weissreish}
Rui Reis and Michael Weiss, \emph{Functor calculus and the discriminant
  method}, Q. J. Math. \textbf{65} (2014), no.~3, 1069--1110. \MR{3261981}

\bibitem[Seg74]{segalcategories}
Graeme Segal, \emph{Categories and cohomology theories}, Topology \textbf{13}
  (1974), 293--312. \MR{0353298 (50 \#5782)}

\bibitem[Sie71]{siebenmannicm}
L.~C. Siebenmann, \emph{Topological manifolds}, 133--163. \MR{0423356}

\bibitem[Sie72]{siebenmanndeformation}
\bysame, \emph{Deformation of homeomorphisms on stratified sets. {I}, {II}},
  Comment. Math. Helv. \textbf{47} (1972), 123--136; ibid. 47 (1972), 137--163.
  \MR{0319207}

\bibitem[Sma59]{smaleimm}
Stephen Smale, \emph{The classification of immersions of spheres in {E}uclidean
  spaces}, Ann. of Math. (2) \textbf{69} (1959), 327--344. \MR{0105117 (21
  \#3862)}

\bibitem[Thu74]{thurstonfol}
William Thurston, \emph{Foliations and groups of diffeomorphisms}, Bull. Amer.
  Math. Soc. \textbf{80} (1974), 304--307. \MR{0339267}

\bibitem[Vas92]{vassiliev}
V.~A. Vassiliev, \emph{Complements of discriminants of smooth maps: topology
  and applications}, Translations of Mathematical Monographs, vol.~98, American
  Mathematical Society, Providence, RI, 1992, Translated from the Russian by B.
  Goldfarb. \MR{1168473 (94i:57020)}

\bibitem[Wei15]{weissimm}
Michael Weiss, \emph{Immersion theory for homotopy theorists}, 2015.

\end{thebibliography}

\end{document}